\documentclass[11pt]{amsart}
\title[]{Fractional Brauer configuration algebras II: covering theory.}
\author{Nengqun Li and Yuming Liu*}

\address{Nengqun Li
\newline School of Mathematics
\newline Liaoning Normal University
\newline Dalian 116029
\newline P.R.China}
\email{linengqun@lnnu.edu.cn}

\address{Yuming Liu
\newline School of Mathematical Sciences
\newline Laboratory of Mathematics and Complex Systems
\newline Beijing Normal University
\newline Beijing 100875
\newline P.R.China}
\email{ymliu@bnu.edu.cn}

\date{version of \today}
\usepackage{amsmath}
\usepackage{amssymb}
\usepackage{amsxtra}
\usepackage{abstract}
\usepackage{mathrsfs}
\usepackage{mathtools}
\usepackage[all]{xy}
\usepackage{amscd}
\usepackage{amsthm}
\usepackage[dvips]{graphicx}
\usepackage{ulem}
\usepackage{color}
\usepackage{appendix}
\usepackage{tikz}

\newtheorem{Thm}{Theorem}[section]
\newtheorem{Lem}[Thm]{Lemma}
\newtheorem{Def}[Thm]{Definition}
\newtheorem{Cor}[Thm]{Corollary}
\newtheorem{Prop}[Thm]{Proposition}
\newtheorem{Ex1}[Thm]{Example}
\newtheorem{Rem1}[Thm]{Remark}

\newcommand{\lra}{\longrightarrow}

\newcommand{\ra}{\rightarrow}
\newcommand{\sdp}{\times\kern-.2em\vrule height1.1ex depth-.05ex}
\newcommand{\epi}{\lra \kern-.8em\ra}

 \normalbaselines
\begin{document}
\renewcommand{\thefootnote}{\alph{footnote}}
\setcounter{footnote}{-1} \footnote{* Corresponding author.}
\setcounter{footnote}{-1} 
\renewcommand{\thefootnote}{\alph{footnote}}
\setcounter{footnote}{-1} 

\maketitle

\begin{abstract}
In this paper, we develop a covering theory for the fractional Brauer configurations and connect it with the coverings of the associated quivers with relations in the sense of Mart\'inez-Villa and de la Pe\~na. Among the results, we show the following: (1) The universal cover of any fractional Brauer configuration is simply connected and we construct explicitly the universal cover of fractional Brauer configurations of type MS; (2) The fundamental group of a fractional Brauer configuration $E$ of type S is isomorphic to the fundamental group of the associated quiver with relations $(Q_E,I_E)$; (3) A (regular) covering of fractional Brauer configurations induces a (Galois) covering of the associated fractional Brauer configuration categories; (4) Set up an analogy of Van Kampen theorem for fractional Brauer configurations and apply it to calculate the fundamental group of any connected Brauer configuration.
\end{abstract}

\section{introduction}

In a previous paper \cite{LL}, we introduced fractional Brauer configurations (abbr. f-BCs) and fractional Brauer configuration algebras (abbr. f-BCAs), which are generalization of Brauer configurations (abbr. BCs) and Brauer configuration algebras (abbr. BCAs) respectively introduced by Green and Schroll \cite{GS}. Moreover, we introduced various types of fractional Brauer configurations, including fractional Brauer configuration of type S (abbr. $f_s$-BC), fractional Brauer configuration of type MS (abbr. $f_{ms}$-BC) and so on, and introduced the corresponding algebras $f_s$-BCA, $f_{ms}$-BCA and so on. Note that all the above algebras are generalizations of Brauer graph algebras (abbr. BGA). The relations of these algebras can be illustrated by the following diagram:

\medskip
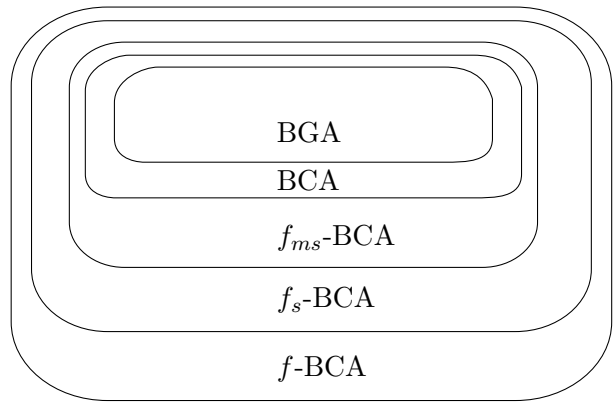
\begin{center}	
		
		\begin{tikzpicture}[x=0.55pt,y=0.45pt,yscale=-1,xscale=1]
			\draw   (170,90) .. controls (170,70) and (190,60) .. (200,60) -- (397,60) .. controls (412,60) and (425,65.97) .. (430,86.9) -- (430,120) .. controls (430,140) and (412,140) .. (397,140) -- (190,140) .. controls (190,140) and (170,140) .. (170,120) -- cycle ;
			\draw   (150,80) .. controls (150,60) and (170,50) .. (180,50) -- (417,50) .. controls (432,50) and (445,55.97) .. (450,76.9) -- (450,150) .. controls (450,170) and (432,170) .. (417,170) -- (170,170) .. controls (170,170) and (150,170) .. (150,150) -- cycle ;
			\draw   (139,76.9) .. controls (139,55.97) and (155.97,39) .. (176.9,39) -- (423.1,39) .. controls (444.03,39) and (461,55.97) .. (461,76.9) -- (461,190.6) .. controls (461,211.53) and (444.03,228.5) .. (423.1,228.5) -- (176.9,228.5) .. controls (155.97,228.5) and (139,211.53) .. (139,190.6) -- cycle ;
			\draw   (113,73.3) .. controls (113,44.42) and (136.42,21) .. (165.3,21) -- (445.7,21) .. controls (474.58,21) and (498,44.42) .. (498,73.3) -- (498,230.2) .. controls (498,259.08) and (474.58,282.5) .. (445.7,282.5) -- (165.3,282.5) .. controls (136.42,282.5) and (113,259.08) .. (113,230.2) -- cycle ;
			\draw   (99,75.7) .. controls (99,39.14) and (128.64,9.5) .. (165.2,9.5) -- (445.8,9.5) .. controls (482.36,9.5) and (512,39.14) .. (512,75.7) -- (512,274.3) .. controls (512,310.86) and (482.36,340.5) .. (445.8,340.5) -- (165.2,340.5) .. controls (128.64,340.5) and (99,310.86) .. (99,274.3) -- cycle ;
			
			\draw (280,104) node [anchor=north west][inner sep=0.75pt]   [align=left] {BGA};
			\draw (280,145) node [anchor=north west][inner sep=0.75pt]   [align=left] {BCA};
			\draw (280,191) node [anchor=north west][inner sep=0.75pt]   [align=left] {$f_{ms}$-BCA};
			\draw (280,243) node [anchor=north west][inner sep=0.75pt]   [align=left] {$f_s$-BCA};
			\draw (280,301) node [anchor=north west][inner sep=0.75pt]   [align=left] {$f$-BCA};

		\end{tikzpicture} \quad .
		
	\end{center}

\medskip
Our motivation to introduce f-BCAs comes from two aspects. The first one is that the class of Brauer configuration algebras is not closed under derived equivalence. The second one is that it is natural to enlarge the algebras from the scope of symmetric algebras to the scope of self-injective algebras. Our study in \cite{LL} showed that the class of f-BCAs behaves well with respect to the above two aspects, for example, it is shown that the class of representation-finite f-BCAs of type S coincides the class of standard representation-finite self-injective algebras and therefore closed under derived equivalence (up to Morita equivalence).

In the present paper, we develop a covering theory for f-BCs and discuss it with the connection of the covering theory for the associated fractional Brauer configuration categories (abbr. f-BCC) and for the associated f-BCAs.

\medskip
In Section 2, we recall the definition of fractional Brauer configuration and fractional Brauer configuration category. We also recall various types of fractional Brauer configurations, and the admissible ideal presentation of a fractional Brauer configuration category of type S.

\medskip
In Section 3, we establish a covering theory for f-BCs. In Section 3.1, we define morphisms and coverings between two f-BCs. In Section 3.2, we introduce walks (which are analog of paths in a topological space) of an f-BC, and define a homotopy relation $\sim$ on the set of walks. Using the homotopy relation we define fundamental groups (groupoids) of an f-BC (see Definition \ref{definition-fundamental-group-of-f-BC}). One reason for us to introduce fundamental group is to compare two coverings $f_1:E_1\rightarrow E$ and $f_2:E_2\rightarrow E$ for a given f-BC $E$ in Section 3.3:

\begin{Prop}{\rm(see Proposition \ref{existence of morphism})}
Let $E$, $E_{1}$, $E_{2}$ be f-BCs with $E_{1}$ connected, $f_{1}:E_{1}\rightarrow E$ a morphism of f-BCs, $f_{2}:E_{2}\rightarrow E$ a covering of f-BCs. For $e_i\in E_{i}$ $(i=1,2)$ with $f_{1}(e_1)=f_{2}(e_2)$, there exists a morphism $\phi:E_{1}\rightarrow E_{2}$ of f-BCs such that $f_{1}=f_{2}\phi$ and $\phi(e_1)=e_2$ if and only if $f_{1*}(\Pi(E_1,e_1))\subseteq f_{2*}(\Pi(E_2,e_2))$. Moreover, such $\phi$ is unique if it exists.
\end{Prop}

Moreover, for an f-BC $E$ with a base point $e$, we construct its universal cover $\widetilde{E}$ at $e$ by using all homotopy classes of walks of $E$ starting at $e$. We show that $\widetilde{E}$ is simply connected (that is, $\widetilde{E}$ is connected and has trivial fundamental group), and show that the natural covering $p:\widetilde{E}\rightarrow E$ satisfies the universal property:

\begin{Prop}{\rm(see Corollary \ref{property of universal cover})}
For each covering $f:E'\rightarrow E$ and for all $\overline{w}\in\widetilde{E}$, $e'\in E'$ with $p(\overline{w})=f(e')$, there exists a unique covering $\phi:\widetilde{E}\rightarrow E'$ such that $\phi(\overline{w})=e'$ and $p=f\phi$.
\end{Prop}

We also define regular covering of f-BCs and describe it using the quotient of f-BCs:

\begin{Thm}{\rm(see Theorem \ref{regular-covering-explicit})}
Let $E,E'$ be connected f-BCs and let $f:E\rightarrow E'$ be a regular covering. Then Aut$(f)$ acts admissibly on $E$, and there exists an isomorphism $r:E/\mathrm{Aut}(f)\xrightarrow{\sim} E'$ such that the diagram
$$\xymatrix{
		& E \ar[dr]^{f}\ar[dl]_{p} &  \\
		E/\mathrm{Aut}(f)\ar[rr]_{r}^{\sim} & & E'
	}$$
commutes, where $p:E\rightarrow E/\mathrm{Aut}(f)$ is the natural projection.
\end{Thm}

In Section 3.4 we deal with f-BC of type MS. For an $f_{ms}$-BC $E$, we can construct its universal cover $\mathbb{Z}B_{(E,e)}$ explicitly using special walks. More precisely, we have:

\begin{Cor}{\rm(see Corollary \ref{universality})}
Let $E=(E,P,L,d)$ be an $f_{ms}$-BC and $e\in E$, and define $q:\mathbb{Z}B_{(E,e)}\rightarrow E$, $(w,n)\mapsto g^{n d(t(w))}\cdot t(w)$. Then $q$ is a covering of f-BCs. Let $\widetilde{E}$ be the universal cover of $E$ at $e$ and $p:\widetilde{E}\rightarrow E$, $\overline{v}\mapsto t(v)$. Then there exists a unique isomorphism of f-BCs $\phi:\mathbb{Z}B_{(E,e)}\rightarrow\widetilde{E}$ such that $q=p\phi$ and $\phi(((e||e),0))=\overline{(e||e)}$.
\end{Cor}

\medskip
In Sections 4 and 5, we study the relation between the covering (resp. fundamental groups) of f-BCs and the covering (resp. fundamental groups) of the associated f-BCCs. We first show that the fundamental group of an $f_s$-BC $E$ is isomorphic to the fundamental group (in the sense of Mart\'inez-Villa and de la Pe\~na) of the associated quiver with relations $(Q_E,I_E)$:

\begin{Thm}{\rm(see Theorem \ref{fundamental-group-of-fs-BC-and-fundamental-group-of-algebra-and-isomorphic} and Corollary \ref{iso-of-fundamental-gp-of-f-BC-and-fundamental-gp-of-quiver-with-aadmissible-relation})}
Let $E$ be a connected $f_s$-BC and $e\in E$. Then there is an isomorphism $\Pi(E,e)\cong\Pi(Q_E,I_E)$, where $\Pi(E,e)$ and $\Pi(Q_E,I_E)$ denote the fundamental group of $E$ and the fundamental group of the quiver with relations $(Q_E,I_E)$ respectively. Moreover, $\Pi(E,e)$ is also isomorphic to the fundamental group of the associated quiver with admissible relations $(Q'_E,I'_E)$.
\end{Thm}

In Section 5.1, we show that a covering of f-BCs can induce a covering between the associated quivers with relations, and therefore induces a covering functor between the associated fractional Brauer configuration categories (abbr. f-BCC):

\begin{Prop}{\rm(see Theorem \ref{covering of f-BCs induces covering functor} and Proposition \ref{induce-Galois-covering-of-categories})}
Let $E$, $E'$ be connected $f_s$-BCs, $\phi:E\rightarrow E'$ be a covering (resp. regular covering). Then $\phi$ induces a covering (resp. Galois covering) $f:(Q_E,I_E)\rightarrow (Q_{E'},I_{E'})$ of the associated quivers with relations. Especially, $\phi$ induces a covering functor (resp. Galois covering functor) $F:\Lambda_{E}\rightarrow\Lambda_{E'}$.
\end{Prop}

Moreover, using a criterion on determining when a locally bounded category is simply connected, we show in Section 5.3 that each fractional Brauer configuration category defined by an $f_{ms}$-BC is standard:

\begin{Prop}{\rm(see Corollary \ref{type-ms-is-standard})}
If $E$ is a connected $f_{ms}$-BC and $\widetilde{E}$ is its universal cover, then $\Lambda_{\widetilde{E}}$ is simply connected. Especially, $\Lambda_E$ is standard.
\end{Prop}

\medskip
In last section, we calculate the fundamental groups of BCs. We first define sub-f-BCs of an f-BC, and define their union and intersection. Then we establish an analogy of the Van Kampen theorem:

\begin{Prop}{\rm(see Proposition \ref{Van-Kampen})}
Let $E$ be an f-BC, which is an admissible union of a family of sub-f-BCs $\{E_{\alpha}\}_{\alpha\in I}$. Let $A$ be a subset of $\bigcap_{\alpha\in I}E_{\alpha}$ such that for each $\alpha\in I$, $A$ meets each connected component of $E_{\alpha}$. Then the groupoid $\Pi(E,A)$ is the direct limit of the groupoids $\Pi(E_{\alpha},A)$.
\end{Prop}

Using this proposition, we can reduce the calculation of the fundamental group of BCs to the calculation of the fundamental group of BGs (see Corollary \ref{fundamental group of BC and BG}), and moreover reduce the calculation of the fundamental group of BGs to two special cases (see Lemma \ref{a calculation of fundamental group} and Lemma \ref{a calculation of fundamental groupoid}). The fundamental groups of a BC is given as follows:

\begin{Thm}{\rm(see Theorem \ref{fundamental group of BC})}
Let $E$ be a connected Brauer configuration with $n$ vertices $v_1$, $\cdots$, $v_n$. Let $k_l$ be the number of $l$-gons of $E$ for each $l\geq 2$. Let $d_i=d_f(v_i)$ be the f-degree of $v_i$ for each $1\leq i\leq n$, and let $r=\sum_{l\geq 2} (l-1)k_l-n+1$. Then the fundamental group of $E$ is isomorphic to
$$F\langle a_1,\cdots,a_n,b_1,\cdots,b_r\rangle/\langle a_{1}^{d_1}=\cdots=a_{n}^{d_n}, a_{1}^{d_1} b_1=b_1 a_{1}^{d_1},\cdots, a_{1}^{d_1} b_r=b_r a_{1}^{d_1}\rangle,$$
where $F\langle a_1,\cdots,a_n,b_1,\cdots,b_r\rangle$ denotes the free group on the set $\{a_1,\cdots,a_n,b_1,\cdots,b_r\}$. In particular, the fundamental group of a connected Brauer configuration is finitely presented.
\end{Thm}

The covering theory developed in this paper is practical and efficiently used in a series subsequent studies \cite{LL2,LL3,GLL}. We also notice that there is already a discussion of coverings on BGs in \cite{GSS}, aiming a connection with coverings on BGAs. However, there are no discussions on the universal cover and on the fundamental group. On the contrary, the fundamental group and the universal cover play an important role in our covering theory for f-BCs. For example, we will determine the representation type of an $f_{ms}$-BGA by using the fundamental group of the defining $f_{ms}$-BG in \cite{LL2} and classify all tame and wild $f_{ms}$-BCAs based on the universal cover of the defining $f_{ms}$-BCs in \cite{LL3}.

\section*{Data availability} The datasets generated during the current study are available from the corresponding author on reasonable request.

\section*{Acknowledgements} The first author is supported by the Doctoral Research Startup Fund of Liaoning Normal University (No.603260070011).

\section{Preliminaries}

Throughout this paper, we fix $k$ to be a field and $G$ an infinite cyclic group generated by $g$. We write a path $p$ in a quiver from right to left and denote by $s(p)$ and $t(p)$ the source and the terminal of $p$, respectively. Let us first recall the definition of fractional Brauer configuration and fractional Brauer configuration category.

\begin{Def} \label{f-BC} {\rm(\cite[Definition 3.3]{LL})}
A fractional Brauer configuration (abbr. f-BC) is a quadruple $E=(E,P,L,d)$, where $E$ is a $G$-set, $P$ and $L$ are two partitions of $E$ (for each $e$ in $E$, denote by $P(e)$ and $L(e)$ the equivalence classes of $e$ under the partitions $P$ and $L$, respectively), and $d: E\rightarrow \mathbb{Z}_{>0}$ is a function, such that the following conditions hold. \\
$(f1)$ $L(e)\subseteq P(e)$ and $P(e)$ is a finite set for each $e\in E$. \\
$(f2)$ If $L(e_1)=L(e_2)$, then $P(g\cdot e_1)=P(g\cdot e_2)$. \\
$(f3)$ If $e_1$, $e_2$ belong to same $G$-orbit, then $d(e_1)=d(e_2)$. \\
$(f4)$ $P(e_1)=P(e_2)$ if and only if $P(g^{d(e_1)}\cdot e_1)=P(g^{d(e_2)}\cdot e_2)$. \\
$(f5)$ $L(e_1)=L(e_2)$ if and only if $L(g^{d(e_1)}\cdot e_1)=L(g^{d(e_2)}\cdot e_2)$. \\
$(f6)$ $L(g^{d(e)-1}\cdot e)\cdots L(g\cdot e)L(e)$ is not a proper subsequence of $L(g^{d(h)-1}\cdot h)\cdots L(g\cdot h)L(h)$ for all $e,h\in E$.
\end{Def}

According to \cite[Section 3]{LL}, the elements of $E$ are called {\it angles}, the $G$-orbits of $E$ are called {\it vertices}, the subsets of $E$ of the form $P(e)$ are called {\it polygons} (or an $n$-gon if $P(e)$ contains $n$ angles), and the function $d: E\rightarrow \mathbb{Z}_{>0}$ is called the {\it degree function}. If $v$ is a vertex of $E$ which is a finite set, define the {\it fractional-degree (abbr. f-degree)} $d_f(v)$ of $v$ to be $\frac{d(v)}{\mid v\mid}$; If the f-degree of each vertex of $E$ is an integer, then $E$ is said to have {\it integral f-degree}. Moreover, $E$ is called {\it f-degree-free} if $d_f(v)\equiv 1$. The permutation $\sigma: E\rightarrow E$, $e\rightarrow g^{d(e)}\cdot e$ on $E$ is called the {\it Nakayama automorphism} of $E$.

Note that if $E$ is an f-BC such that each polygon of $E$ has exactly two angles (or two half-edges), then we call $E$ a {\it fractional Brauer graph (abbr. $f$-BG)}. We call the polygons in an f-BG as edges.

According to \cite{LL}, every fractional Brauer configuration $E=(E,P,L,d)$ can be visualized as some diagram $\Gamma(E)$, where each $n$-gon in $\Gamma(E)$ corresponds to an $n$-gon (of the form $P(e)$) of $E$, and each $e\in E$ corresponds to an angle of the polygon $P(e)$ in $\Gamma(E)$. Note that if $E$ contains only $2$-gons then $\Gamma(E)$ is a graph.

\begin{Def} \label{f-BC algebra} {\rm(\cite[Definition 4.1 and Definition 4.4]{LL})}
For an f-BC $E=(E,P,L,d)$, the fractional Brauer configuration category (abbr. f-BCC) associated with $E$ is a $k$-category $\Lambda_{E}=kQ_{E}/I_{E}$, where $Q_{E}$ is a quiver defined as follows: $(Q_{E})_{0}=\{P(e)\mid e\in E\}$, $(Q_{E})_{1}=\{L(e)\mid e\in E\}$ with $s(L(e))=P(e)$ and $t(L(e))=P(g\cdot e)$, and $I_{E}$ is the ideal of path category $kQ_{E}$ generated by the relations $(fR1)$, $(fR2)$ and $(fR3)$:
\begin{itemize}
\item[$(fR1)$] $L(g^{d(e)-1-k}\cdot e)\cdots L(g\cdot e)L(e)-L(g^{d(h)-1-k}\cdot h)\cdots L(g\cdot h)L(h)$, where $k\geq 0$, $P(e)=P(h)$ and $L(g^{d(e)-i}\cdot e)=L(g^{d(h)-i}\cdot h)$ for $1\leq i\leq k$.
\item[$(fR2)$] Paths of the form $L(e_n)\cdots L(e_2)L(e_1)$ with $\bigcap_{i=1}^n g^{n-i}\cdot L(e_i)=\emptyset$ for $n>1$.
\item[$(fR3)$] Paths of the form $L(g^{n-1}\cdot e)\cdots L(g\cdot e)L(e)$ for $n>d(e)$.
\end{itemize}
Moreover, if $E$ is a finite $f$-BC, then we define $A_{E}=(\bigoplus_{x,y\in (Q_E)_0}\Lambda_{E}(x,y))^{op}$ (which is a finite dimensional $k$-algebra) and call $A_E$ a fractional Brauer configuration algebra (abbr. f-BCA).
\end{Def}

According to \cite[Theorem 4.7]{LL}, every fractional Brauer configuration category $\Lambda_E$ is a locally bounded $k$-category.

We also recall the definition of various types of fractional Brauer configurations from \cite{LL}. First we recall fractional Brauer configuration of type S (abbr. $f_s$-BC).

For an f-BC $E=(E,P,L,d)$, recall from \cite[Definition 3.10]{LL} that a {\it standard sequence} of $E$ is a sequence of the form $p=(g^{n-1}\cdot e,\cdots,g\cdot e,e)$, where $e\in E$ and $0\leq n\leq d(e)$ (we define $p=()_e$ when $n=0$ and call it the trivial sequence at $e$). The source $s(p)$ (resp. terminal $t(p)$) of a standard sequence $p=(g^{n-1}\cdot e,\cdots,g\cdot e,e)$ is defined to be $e$ (resp. $g^n\cdot e$) if $n>0$, and both the source $s(p)$ and the terminal $t(p)$ of the trivial sequence $p=()_e$ are defined to be $e$. By Definition \ref{morphism and covering of pre-configurations} below, the image of a standard sequence under a morphism of f-BCs is also a standard sequence.

Let $E$ be an f-BC. For a standard sequence $p=(g^{n-1}\cdot e,\cdots,g\cdot e,e)$ of $E$, we can define two associated standard sequences $$\prescript{\wedge}{}{p}=(g^{d(e)-1}\cdot e,\cdots,g^{n+1}\cdot e,g^{n}\cdot e), \quad p^{\wedge}=(g^{-1}\cdot e,g^{-2}\cdot e,\cdots,g^{n-d(e)}\cdot e),$$ and define a formal sequence $$L(p)=L(g^{n-1}\cdot e)\cdots L(g\cdot e)L(e),$$
if $p=()_e$ is a trivial sequence at $e$, we define $L(p)=1_{P(e)}$. For two standard sequences $p$, $q$, define $p\equiv q$ if $L(p)=L(q)$ (in this case $p,q$ are called {\it identical}). For a standard sequence $p$, denote by $[p]=\{$ standard sequence $q\mid q\equiv p\}$. For a set $\mathscr{X}$ of standard sequences, $[\mathscr{X}]$ can be defined in the same way, and denote by $\prescript{\wedge}{}{\mathscr{X}}=\{\prescript{\wedge}{}{p}\mid p\in \mathscr{X}\}$ (resp. $\mathscr{X}^{\wedge}=\{p^{\wedge}\mid p\in \mathscr{X}\}$).

\begin{Def} {\rm(\cite[Definition 3.13]{LL})}
An f-BC $E$ is said to be of type S (or $E$ is an $f_s$-BC in short) if it satisfies additionally the following condition:

$(f7)$ For standard sequences $p\equiv q$, $[[\prescript{\wedge}{}{p}]^{\wedge}]=[[\prescript{\wedge}{}{q}]^{\wedge}]$ or $[\prescript{\wedge}{}{[p^{\wedge}]}]=[\prescript{\wedge}{}{[q^{\wedge}]}]$.
\end{Def}

According to \cite[Theorem 4.21]{LL}, every fractional Brauer configuration category $\Lambda_E$ of an $f_s$-BC $E$ is a locally bounded Frobenius $k$-category. Moreover, if $E$ is a finite $f_s$-BC, then $A_E$ is a finite dimensional Frobenius algebra, and the Nakayama automorphism $\sigma$ of $E$ induces an algebra automorphism $\sigma$ of $A_E$, which is equal to the usual Nakayama automorphism of the Frobenius algebra $A_E$.

\begin{Def}{\rm(\cite[Definition 3.15]{LL})}\label{Def-fms-BC}
An f-BC $E=(E,P,L,d)$ is said to be of type MS (or $E$ is an $f_{ms}$-BC in short) if the partition $L$ of $E$ is trivial, that is, $L(e)=\{e\}$ for each $e\in E$.
\end{Def}

Clearly every f-BC of type MS is of type S. A finite $f_{ms}$-BC $E=(E,P,L,d)$ with integral f-degree and with no $1$-gons is called a {\it Brauer configuration (abbr. BC)}. Moreover, a BC containing only $2$-gons is called a {\it Brauer graph (abbr. BG)}. A Brauer graph $E$ whose diagram $\Gamma(E)$ is a tree, such that there exists at most one vertex with f-degree larger than $1$ (such a vertex is called the {\it exceptional vertex} of $E$), is called a {\it Brauer tree (abbr. BT)}.

Finally we recall from \cite{LL} the admissible ideal presentation of a fractional Brauer configuration category of type S.

\begin{Def}{\rm(\cite[Definition 4.9]{LL})}\label{relation R}
Let $E=(E,P,L,d)$ be an f-BC and let $$\mathscr{E}=\{L(p)\mid p\mbox{ is a standard sequence of }E\},$$ which is a set of paths of $Q_E$. Define a relation $R$ on $\mathscr{E}$ as follows: for $u$, $v\in\mathscr{E}$, $uRv$ if and only if there exist some standard sequences $p$, $q$ of $E$ such that $u=L(p)$, $v=L(q)$ and $\prescript{\wedge}{}{p}\equiv \prescript{\wedge}{}{q}$.
\end{Def}

Note that two paths $u$, $v\in\mathscr{E}$ with $uRv$ have the same source and the same terminal. If $E$ is an $f_s$-BC, it follows from \cite[Lemma 4.14]{LL} that $R$ is an equivalence relation on $\mathscr{E}$.

\begin{Def}{\rm(\cite[Definition 6.1]{LL})}\label{reduced-arrow}
Let $E$ be an $f_s$-BC. For a set $\mathscr{C}$ of paths in $Q_E$ and for two vertices $x$, $y$ of $Q_E$, denote by $\prescript{}{y}{\mathscr{C}}_x$ the subset of $\mathscr{C}$ consists of paths with source $x$ and terminal $y$.
\begin{itemize}
  \item[(1)] We call an arrow $\alpha$ of $Q_E$ from $x$ to $y$ reduced, if there exists a path $p\in\prescript{}{y}{\mathscr{E}}_x$ of length $\geq 2$ such that $\alpha R p$.
  \item[(2)] Denote by $\mathscr{N}$ a complete set of representatives of non-reduced arrows of $Q_E$ under the equivalence relation $R$, and define a subquiver $Q'_E$ of $Q_E$: $$(Q'_E)_0=(Q_E)_0, \quad (Q'_E)_1=\bigsqcup_{x,y\in (Q_E)_0}\prescript{}{y}{\mathscr{N}}_x.$$
  \item[(3)] Denote by $\rho:kQ'_E\rightarrow \Lambda_{E}=kQ_E/I_E$ the natural $k$-linear functor, and let $I'_E=\mathrm{ker}\rho$.
      \end{itemize}
\end{Def}

If $E$ is an $f_s$-BC, then it follows from \cite[Lemma 6.2]{LL} that $I'_E$ is an admissible ideal in $kQ'_E$ and $\Lambda_E$ is isomorphic to the category $kQ'_E/I'_E$.

\section{Covering theory for fractional Brauer configurations}

\subsection{Morphisms and coverings between f-BCs}
\

For two f-BCs $E$, $E'$, we define the morphism between them as follows.

\begin{Def}\label{morphism and covering of pre-configurations}
Let $E=(E,P,L,d)$, $E'=(E',P',L',d')$ be two f-BCs. A map $f:E\rightarrow E'$ is said to be a morphism of f-BCs if it satisfies the following three conditions:

(1) $f$ is a morphism of $G$-sets, that is, $f(g^i\cdot e)=g^i\cdot f(e)$ for each $e\in E$ and each $i\in\mathbb{Z}$;

(2) $f(P(e))\subseteq P'(f(e))$ and $f(L(e))\subseteq L'(f(e))$ for each $e\in E$;

(3) $f$ preserves the degree, that is, $d'(f(e))=d(e)$ for each $e\in E$.

Moreover,
\begin{itemize}
\item if $f$ has an inverse morphism from $E'$ to $E$, then $f$ is said to be an isomorphism of f-BCs;
\item if $f$ induces bijections $P(e)\rightarrow P'(f(e))$ and $L(e)\rightarrow L'(f(e))$ for all $e\in E$, then $f$ is said to be a covering of f-BCs.
\end{itemize}
\end{Def}

\begin{Ex1}\label{example-Nakayama-automorphism}
If $E=(E,P,L,d)$ is an f-BC, then by the conditions $(f3), (f4)$ and $(f5)$ in Definition \ref{f-BC}, the Nakayama automorphism $\sigma$ of $E$ is an automorphism of the f-BC $E$.
\end{Ex1}

\begin{Ex1}\label{example-covering}
Let $E=(E,P,L,d)$ be the f-BG given by the diagram

\vspace{0.5cm}
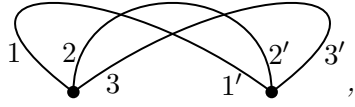
\begin{center}
\tikzset{every picture/.style={line width=0.75pt}}
\begin{tikzpicture}[x=15pt,y=15pt,yscale=1,xscale=1]
\fill (0,0) circle (0.5ex);
\fill (5,0) circle (0.5ex);
\node at(-1.5,1) {$1$};
\node at(-0.1,1) {$2$};
\node at(1,0.2) {$3$};
\node at(6.6,1) {$3'$};
\node at(5.2,1) {$2'$};
\node at(4,0.2) {$1'$};
\node at(7,0) {,};
\draw    (0,0) .. controls (-4,3) and (1,3) .. (5,0) ;
\draw    (0,0) .. controls (0,3) and (5,3) .. (5,0) ;
\draw    (0,0) .. controls (4,3) and (9,3) .. (5,0) ;
\end{tikzpicture}
\end{center}
where $L(e)=\{e\}$ and $d(e)=2$ for every $e\in E$; and let $E'=(E',P',L',d')$ be the f-BG given by the diagram

\vspace{0.5cm}
\begin{center}
\tikzset{every picture/.style={line width=0.75pt}}
\begin{tikzpicture}[x=15pt,y=15pt,yscale=1,xscale=1]
\fill (0,0) circle (0.5ex);
\fill (5,0) circle (0.5ex);
\node at(1,0.5) {$x$};
\node at(4,0.5) {$y$};
\node at(6,0) {,};
\draw    (0,0)--(5,0);
\end{tikzpicture}
\end{center}
where $L'(e)=\{e\}$ and $d'(e)=2$ for every $e\in E'$. Then there is a covering $f:E\rightarrow E'$ of f-BGs which is given by $f(1)=f(2)=f(3)=x$, $f(1')=f(2')=f(3')=y$.
\end{Ex1}

\begin{Ex1}\label{example-non-covering}
Let $E$ be the f-BG in Example \ref{example-covering}, and let $E''=(E'',P'',L'',d'')$ be the f-BG given by the diagram

\vspace{0.5cm}
\begin{center}
\tikzset{every picture/.style={line width=0.75pt}}
\begin{tikzpicture}[x=15pt,y=15pt,yscale=1,xscale=1]
\fill (0,0) circle (0.5ex);
\fill (5,0) circle (0.5ex);
\node at(1,0.5) {$x$};
\node at(4,0.5) {$y$};
\node at(6,0) {,};
\draw    (0,0)--(5,0);
\end{tikzpicture}
\end{center}
where $L''(x)=\{x,y\}$ and $d''(e)=2$ for every $e\in E'$. Then there is a morphism $f:E\rightarrow E''$ of f-BGs which is given by $f(1)=f(2)=f(3)=x$, $f(1')=f(2')=f(3')=y$. Note that $f$ is not a covering because $f$ does not induce a bijective map from $L(1)$ to $L''(x)$.
\end{Ex1}

We observe that a morphism $f:E\rightarrow E'$ of f-BCs induces naturally a map from standard sequences of $E$ to standard sequences of $E'$, and $f$ sends identical standard sequences of $E$ to identical standard sequences of $E'$. Moreover, if $f:E\rightarrow E'$ is a covering of f-BCs, then we have the following two lemmas.

\begin{Lem}\label{preserves parallel walk}
Let $E=(E,P,L,d)$ and $E'=(E',P',L',d')$ be f-BCs, and let $f:E\rightarrow E'$ be a covering of f-BCs. If $p$, $q$ are standard sequences of $E$ with $f(p)\equiv f(q)$ and $P(s(p))=P(s(q))$, then $p\equiv q$.
\end{Lem}

\begin{proof}
Since $f(p)\equiv f(q)$, $p$ and $q$ have the same length. We shall give a proof by induction on the length $l(p)$ of $p$. If $l(p)=0$, then both $p$ and $q$ are trivial sequences. Since $P(s(p))=P(s(q))$, $p=()_{s(p)}$ and $q=()_{s(q)}$ are identical. Suppose that the conclusion holds for $l(p)<n$. When $l(p)=n$, let $p=(g^{n-1}\cdot e,\cdots,g\cdot e,e)$ and $q=(g^{n-1}\cdot h,\cdots,g\cdot h,h)$. Since $f$ is a covering, it induces bijections $P(e)\rightarrow P'(f(e))$ and $L(e)\rightarrow L'(f(e))$. Since $f(p)\equiv f(q)$, $L'(f(e))=L'(f(h))$ and $f(h)\in L'(f(e))$. Therefore $f(h)=f(h')$ for some $h'\in L(e)$. Since $h$, $h'\in P(e)$ and $f(h)=f(h')$, we have $h=h'$. Therefore $L(e)=L(h)$ and $P(g\cdot e)=P(g\cdot h)$. Let $p'=(g^{n-1}\cdot e,\cdots,g\cdot e)$ and $q'=(g^{n-1}\cdot h,\cdots,g\cdot h)$, by induction, $p'\equiv q'$. Therefore $p\equiv q$.
\end{proof}

\begin{Lem}\label{bijection}
Let $E=(E,P,L,d)$ and $E'=(E',P',L',d')$ be f-BCs, $f:E\rightarrow E'$ be a covering of f-BCs, and $p$ be a standard sequence of $E$. Then $f$ induces a bijection $[p]\rightarrow [f(p)]$.
\end{Lem}

\begin{proof}
This is obviously true if the length $l(p)$ of $p$ is zero. Thus we may assume that $l(p)>0$. If $q\in[p]$, then $p\equiv q$ and $f(p)\equiv f(q)$, therefore $f(q)\in[f(p)]$. For
$$q_1=(g^{n-1}\cdot e,\cdots,g\cdot e,e), q_2=(g^{n-1}\cdot h,\cdots,g\cdot h,h)\in[p],$$
if $f(q_1)=f(q_2)$, then $f(g^{i}\cdot e)=f(g^{i}\cdot h)$ for all $0\leq i\leq n-1$. Since $L(g^{i}\cdot e)=L(g^{i}\cdot h)$ and $f$ induces bijections $L(g^{i}\cdot e)\rightarrow L'(f(g^{i}\cdot e))$ for all $0\leq i\leq n-1$, $g^{i}\cdot e=g^{i}\cdot h$ for all $0\leq i\leq n-1$. Then $q_1=q_2$ and $f:[p]\rightarrow [f(p)]$ is injective. For every $r\in[f(p)]$, since $L'(s(r))=L'(s(f(p)))$, there exists $e\in L(s(p))$ such that $f(e)=s(r)$. Let $q=(g^{n-1}\cdot e,\cdots,g\cdot e,e)$ be a standard sequence of $E$, where $n$ is the length of $r$. Then $f(q)=r\equiv f(p)$ and $P(s(q))=P(e)=P(s(p))$. By Lemma \ref{preserves parallel walk}, $p\equiv q$. Therefore $f:[p]\rightarrow [f(p)]$ is surjective.
\end{proof}

Using the above two lemmas we show that a covering of f-BCs preserves the property of being type S.

\begin{Prop}\label{covering preserves the property of being a configuration}
Let $E=(E,P,L,d)$ and $E'=(E',P',L',d')$ be two f-BCs, and let $f:E\rightarrow E'$ be a covering of f-BCs. If $E'$ is an $f_s$-BC, then so is $E$. Conversely, if $E$ is an $f_s$-BC and $f$ is surjective, then $E'$ is an $f_s$-BC.
\end{Prop}

\begin{proof}
Suppose that $E'$ is an $f_s$-BC. If $p\equiv q$ are standard sequences of $E$, then $f(p)\equiv f(q)$ and $[[\prescript{\wedge}{}{f(p)}]^{\wedge}]=[[\prescript{\wedge}{}{f(q)}]^{\wedge}]$. For $r\in[\prescript{\wedge}{}{p}]^{\wedge}$, we have $\prescript{\wedge}{}{r}\equiv\prescript{\wedge}{}{p}$. Therefore $\prescript{\wedge}{}{f(r)}=f(\prescript{\wedge}{}{r})\equiv f(\prescript{\wedge}{}{p})=\prescript{\wedge}{}{f(p)}$ and $f(r)\in [\prescript{\wedge}{}{f(p)}]^{\wedge}\subseteq [[\prescript{\wedge}{}{f(q)}]^{\wedge}]$. Then $f(r)\equiv l$ for some standard sequence $l$ of $E'$ with $\prescript{\wedge}{}{l}\equiv \prescript{\wedge}{}{f(q)}=f(\prescript{\wedge}{}{q})$. By Lemma \ref{bijection}, there exists some standard sequence $\prescript{\wedge}{}{v}$ of $E$ such that $\prescript{\wedge}{}{v}\equiv \prescript{\wedge}{}{q}$ and $f(\prescript{\wedge}{}{v})=\prescript{\wedge}{}{l}$. Since $\prescript{\wedge}{}{v}\equiv \prescript{\wedge}{}{q}$, $P(t(\prescript{\wedge}{}{v}))=P(t(\prescript{\wedge}{}{q}))$ and $P(s(v))=P(\sigma^{-1}( t(\prescript{\wedge}{}{v})))=P(\sigma^{-1}(t(\prescript{\wedge}{}{q})))=P(s(q))$, where $\sigma$ is the Nakayama automorphism of $E$. Similarly, $P(s(r))=P(s(p))$ since $\prescript{\wedge}{}{r}\equiv\prescript{\wedge}{}{p}$. Since $p\equiv q$, $P(s(p))=P(s(q))$. Therefore $P(s(v))=P(s(r))$. Since $f(v)=l\equiv f(r)$, by Lemma \ref{preserves parallel walk}, $v\equiv r$. Since $v\in [\prescript{\wedge}{}{q}]^{\wedge}$, $r\in [[\prescript{\wedge}{}{q}]^{\wedge}]$. Therefore $[\prescript{\wedge}{}{p}]^{\wedge}\subseteq[[\prescript{\wedge}{}{q}]^{\wedge}]$. Since $[[\prescript{\wedge}{}{q}]^{\wedge}]$ is closed under taking identical standard sequences, $[[\prescript{\wedge}{}{p}]^{\wedge}]\subseteq[[\prescript{\wedge}{}{q}]^{\wedge}]$. Dually, $[[\prescript{\wedge}{}{q}]^{\wedge}]\subseteq[[\prescript{\wedge}{}{p}]^{\wedge}]$. Therefore $E$ is an $f_s$-BC.

Conversely, Suppose that $E$ is an $f_s$-BC and $f$ is surjective. Let $p'\equiv q'$ be standard sequences of $E'$. Since $f$ is surjective, we can choose a standard sequence $p$ of $E$ such that $f(p)=p'$. By Lemma \ref{bijection}, there exists a standard sequence $q$ of $E$ such that $p\equiv q$ and $f(q)=q'$. Since $E$ is an $f_s$-BC, $[[\prescript{\wedge}{}{p}]^{\wedge}]=[[\prescript{\wedge}{}{q}]^{\wedge}]$. For $a'\in [\prescript{\wedge}{}{p'}]^{\wedge}$, we have $\prescript{\wedge}{}{a'}\equiv\prescript{\wedge}{}{p'}$. Since $f(\prescript{\wedge}{}{p})=\prescript{\wedge}{}{p'}$, by Lemma \ref{bijection}, there exists a standard sequence $\prescript{\wedge}{}{a}$ of $E$ such that $\prescript{\wedge}{}{a}\equiv \prescript{\wedge}{}{p}$ and $f(\prescript{\wedge}{}{a})=\prescript{\wedge}{}{a'}$. Then $a\in [\prescript{\wedge}{}{p}]^{\wedge}\subseteq [[\prescript{\wedge}{}{q}]^{\wedge}]$. Let $a\equiv b$ for standard sequence $b$ with $\prescript{\wedge}{}{b}\equiv \prescript{\wedge}{}{q}$. Then $f(\prescript{\wedge}{}{b})\equiv f(\prescript{\wedge}{}{q})=\prescript{\wedge}{}{f(q)}=\prescript{\wedge}{}{q'}$, and $f(b)=f(\prescript{\wedge}{}{b})^{\wedge}\in[\prescript{\wedge}{}{q'}]^{\wedge}$. Since $a'=f(\prescript{\wedge}{}{a})^{\wedge}=f(a)\equiv f(b)$, $a'\in[[\prescript{\wedge}{}{q'}]^{\wedge}]$. Therefore $[[\prescript{\wedge}{}{p'}]^{\wedge}]\subseteq[[\prescript{\wedge}{}{q'}]^{\wedge}]$. Dually, $[[\prescript{\wedge}{}{q'}]^{\wedge}]\subseteq[[\prescript{\wedge}{}{p'}]^{\wedge}]$. Therefore $E'$ is an $f_s$-BC.
\end{proof}

\subsection{Walks and fundamental group of f-BC}
\

Let $E$ be an f-BC. A {\it walk} $w$ of $E$ is a sequence of the form $$e_{n}\frac{\delta_{n}}{}e_{n-1}\frac{\delta_{n-1}}{}\cdots\frac{\delta_{3}}{}e_{2}\frac{\delta_{2}}{}e_{1}\frac{\delta_{1}}{}e_{0},$$
where $e_{0},e_{1},\cdots,e_{n}\in E$, $\delta_{1},\cdots,\delta_{n}\in\{g,g^{-1},\tau\}$, such that
\begin{equation*}
\begin{cases}
e_{i}=g\cdot e_{i-1}, & \text{ if } \delta_{i}=g; \\
e_{i}=g^{-1}\cdot e_{i-1}, & \text{ if } \delta_{i}=g^{-1}; \\
P(e_{i})=P(e_{i-1}), & \text{ if } \delta_{i}=\tau.
\end{cases}
\end{equation*}
We may also write $w$ as
$$(e_{n}|\delta_{n}|e_{n-1})\cdots(e_{2}|\delta_{2}|e_{1})(e_{1}|\delta_{1}|e_{0}),$$
or
$$(e_{n}|\delta_{n}\cdots\delta_{1}|e_{0})$$
for short. Define the length of the above walk $w$ to be $n$. Moreover, denote the walk of length $0$ at $e$ by $(e||e)$. For a walk $$w=e_{n}\frac{\delta_{n}}{}e_{n-1}\frac{\delta_{n-1}}{}\cdots\frac{\delta_{3}}{}e_{2}\frac{\delta_{2}}{}e_{1}\frac{\delta_{1}}{}e_{0},$$
define $s(w)=e_{0}$ and $t(w)=e_{n}$, where $s(w)$ and $t(w)$ are the source and the terminal of $w$, respectively. The composition of walks is defined in the usual way. Define the inverse of a walk
$$w=e_{n}\frac{\delta_{n}}{}e_{n-1}\frac{\delta_{n-1}}{}\cdots\frac{\delta_{3}}{}e_{2}\frac{\delta_{2}}{}e_{1}\frac{\delta_{1}}{}e_{0}$$
to be the walk $$w^{-1}=e_{0}\frac{\delta_{1}^{-1}}{}e_{1}\frac{\delta_{2}^{-1}}{}\cdots\frac{\delta_{n-2}^{-1}}{}e_{n-2}\frac{\delta_{n-1}^{-1}}{}e_{n-1}\frac{\delta_{n}^{-1}}{}e_{n},$$
where for any $1\leq i\leq n$, \begin{equation*}
\delta_i^{-1}:= \begin{cases}
g^{-1}, & \text{if } \delta_i=g; \\
g, & \text{if } \delta_i=g^{-1}; \\
\tau, & \text{if } \delta_i=\tau.
\end{cases}
\end{equation*}
$E$ is said to be {\it connected} if each two angles of $E$ can be connected by a walk.

\begin{Rem1}
(1) If $f:E\rightarrow E'$ is a morphism of f-BCs and $$w=e_{n}\frac{\delta_{n}}{}e_{n-1}\frac{\delta_{n-1}}{}\cdots\frac{\delta_{3}}{}e_{2}\frac{\delta_{2}}{}e_{1}\frac{\delta_{1}}{}e_{0}$$ is a walk of $E$, then $$f(w)=f(e_{n})\frac{\delta_{n}}{}f(e_{n-1})\frac{\delta_{n-1}}{}\cdots\frac{\delta_{3}}{}f(e_{2})\frac{\delta_{2}}{}f(e_{1})\frac{\delta_{1}}{}f(e_{0})$$ is a walk of $E'$. In particular, $w$ and $f(w)$ have the same length.

(2) If $f:E\rightarrow E'$ is a covering of f-BCs with $E$ nonempty and $E'$ connected, then $f$ is surjective.
\end{Rem1}

\begin{Def}\label{homotopy of walks}
Let $E$ be an f-BC. Define the homotopy relation $\sim$ on the set of all walks of $E$ as the smallest equivalence relation satisfying the conditions $(h1)$-$(h5)$:\\
$(h1)$ $(e|g^{-1}g|e)\sim(e|gg^{-1}|e)\sim(e|\tau|e)\sim(e||e)$. \\
$(h2)$ $(e_{2}|\tau|e_{1})(e_{1}|\tau|e_{0})\sim(e_{2}|\tau|e_{0})$. \\
$(h3)$ $(g^{d(h)}\cdot h|\tau|g^{d(e)}\cdot e)(g^{d(e)}\cdot e|g^{d(e)}|e)\sim (g^{d(h)}\cdot h|g^{d(h)}|h)(h|\tau|e)$ for $P(e)=P(h)$. \\
$(h4)$ $(g\cdot h|\tau|g\cdot e)(g\cdot e|g|e)\sim (g\cdot h|g|h)(h|\tau|e)$ for $L(e)=L(h)$. \\
$(h5)$ If $w_{1}\sim w_{2}$, then $uw_{1}\sim uw_{2}$ and $w_{1}v\sim w_{2}v$ whenever the compositions make sense.
\end{Def}

\begin{Rem1}
If two walks $u$, $v$ are homotopic, then $s(u)=s(v)$ and $t(u)=t(v)$.
\end{Rem1}

For a walk $w$, we denote by $\overline{w}$ the {\it homotopy class} of $w$. For two homotopy classes of walks $\overline{u}$, $\overline{v}$ with $t(v)=s(u)$, define {\it the composition} $\overline{u}\cdot\overline{v}:=\overline{uv}$. Since $w_{1}\sim w_{2}$ implies $uw_{1}\sim uw_{2}$ and $w_{1}v\sim w_{2}v$, this composition is well-defined.

\begin{Def} \label{definition-fundamental-group-of-f-BC}
Let $E$ be an f-BC and $e\in E$. The fundamental group $\Pi(E,e)$ of $E$ at $e$ is the group of homotopy classes of closed walks of $E$ at $e$, whose multiplication is given by the composition of homotopy classes of walks. Moreover, for a subset $A$ of $E$, the fundamental groupoid $\Pi(E,A)$ of $E$ on $A$ is the groupoid defined as follows: the object set of $\Pi(E,A)$ is $A$; for $x,y\in A$, $\Pi(E,A)(x,y)$ is the set of homotopy classes of walks of $E$ from $x$ to $y$; the composition of morphisms in $\Pi(E,A)$ is given by the composition of homotopy classes of walks.
\end{Def}

Note that the groupoid $\Pi(E,A)$ contains the information of fundamental groups $\Pi(E,e)$ for all $e\in A$ (since $\Pi(E,A)(e,e)\cong\Pi(E,e)$ for every $e\in A$).

\begin{Ex1}
Let $E$ be the Brauer graph

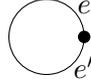
\begin{center}
\begin{tikzpicture}
\draw (0,0) circle (0.5);
\fill (0.5,0) circle (0.5ex);
\node at(0.5,0.4) {$e$};
\node at(0.5,-0.4) {$e'$};
\end{tikzpicture}
\end{center}
with a given half-edge $e$, where the f-degree of $E$ is free. Then $\Pi(E,e)$ is generated by $x=\overline{(e|g^2|e)}$ and $y=\overline{(e|\tau g|e)}$ with relation $xy=yx$ (see Lemma \ref{a calculation of fundamental group}).
\end{Ex1}

\begin{Ex1}
Let $E$ be the Brauer graph
\begin{center}
\begin{tikzpicture}
\draw (0,0)--(2,0);
\fill (0,0) circle (0.5ex);
\fill (2,0) circle (0.5ex);
\node at(0.4,0.15) {$e$};
\node at(1.6,0.2) {$e'$};
\end{tikzpicture}
\end{center}
with free f-degree, and let $A=\{e,e'\}=E$ be a subset of $E$. Let $u=(e|g|e)$, $v=(e'|\tau|e)$, $w=(e'|g|e')$ be three walks of $E$. Then $\overline{u}\in\Pi(E,A)(e,e)$, $\overline{v}\in\Pi(E,A)(e,e')$, $\overline{w}\in\Pi(E,A)(e',e')$. Moreover, $\Pi(E,A)$ is generated by the morphisms $\overline{u}$, $\overline{v}$, $\overline{w}$ with relation $\overline{v}\cdot\overline{u}=\overline{w}\cdot\overline{v}$ (see Lemma \ref{a calculation of fundamental groupoid}).
\end{Ex1}

\subsection{Covering theory for f-BCs}
\

In this subsection, we develop a covering theory for f-BCs. The main results are outlined as follows: $(1)$ We show that a morphism (resp. covering) of f-BCs induces a homomorphism (resp. an injective homomorphism) of their fundamental groups (Theorem \ref{covering-injective}). $(2)$ We define the universal cover of an f-BC and show that it is simply connected (Theorem \ref{universal-cover-description} and Proposition \ref{universal cover is simply connected}). $(3)$ We determine when there exists a morphism between two coverings $E_i\rightarrow E$ ($i=1,2$) of f-BCs (Proposition \ref{existence of morphism}) and classify coverings $E'\rightarrow E$ of connected f-BCs up to isomorphism (Proposition \ref{classification-of-coverings-by-fundamental-group}). $(4)$ We define regular covering and study its properties (Theorem \ref{regular-covering-explicit}).

\begin{Lem}\label{morphism preserves homotopic walks}
Let $f:E\rightarrow E'$ be a morphism of f-BCs, $u$, $v$ be two homotopic walks of $E$. Then $f(u)\sim f(v)$.
\end{Lem}

\begin{proof}
By definition, two walks $u$, $v$ of $E$ are homotopic if and only if there exists a sequence $u=w_0$, $w_1$, $\cdots$, $w_{n}=v$ of walks of $E$, such that for each $1\leq i\leq n$, $w_i$ is obtained from $w_{i-1}$ by replacing a subwalk to another, using one of the relations $(h1)$, $(h2)$, $(h3)$ or $(h4)$ in Definition \ref{homotopy of walks}. Since a morphism of f-BCs preserves all these relations, it maps homotopic walks to homotopic walks.
\end{proof}

\begin{Lem}\label{covering preserves homotopic walks}
Let $E=(E,P,L,d)$, $E'=(E',P',L',d')$ be two f-BCs, $f:E\rightarrow E'$ be a covering of f-BCs, $u$, $v$ be two walks of $E$ with $s(u)=s(v)$ (or $t(u)=t(v)$) and $f(u)\sim f(v)$. Then $u\sim v$.
\end{Lem}

\begin{proof}
We only prove the case $s(u)=s(v)$. Since $f$ is a covering, for each walk $w'$ of $E'$ and for each $e\in E$ such that $f(e)=s(w')$, there exists a unique walk $w$ of $E$ such that $f(w)=w'$ and $s(w)=e$. Since $f(u)\sim f(v)$, there exists a sequence $f(u)=w'_0$, $w'_1$, $\cdots$, $w'_{n}=f(v)$ of walks of $E'$, such that for each $1\leq i\leq n$, $w'_i$ is obtained from $w'_{i-1}$ by replacing a subwalk to another, using one of the relations $(h1)$, $(h2)$, $(h3)$ or $(h4)$ in Definition \ref{homotopy of walks}. For each $0\leq i\leq n$, let $w_i$ be the walk of $E$ such that $s(w_i)=s(u)$ and $f(w_i)=w'_i$. Clearly we have $w_0=u$ and $w_n=v$. To show that $u\sim v$, it suffices to show that $w_{i-1}\sim w_i$ for all $1\leq i\leq n$. So we may assume that $n=1$. Therefore $f(u)$ is obtained from $f(v)$ by replacing a subwalk to another, using one of the relations $(h1)$, $(h2)$, $(h3)$ or $(h4)$ in Definition \ref{homotopy of walks}.

When the replacement is by using relation $(h1)$ or $(h2)$ in Definition \ref{homotopy of walks}, it can be shown straightforward that $v$ is also obtained from $u$ by replacing a subwalk to another, by using the same relation. When the replacement is by using relation $(h3)$ in Definition \ref{homotopy of walks}, assume that $f(u)=w'_{2}(g^{d'(h')}\cdot h'|\tau|g^{d'(e')}\cdot e')(g^{d'(e')}\cdot e'|g^{d'(e')}|e')w'_{1}$ and $f(v)=w'_{2}(g^{d'(h')}\cdot h'|g^{d'(h')}|h')(h'|\tau|e')w'_{1}$, where $w'_{1}$, $w'_{2}$ are two walks of $E'$ and $e'$, $h'\in E'$ with $P'(e')=P'(h')$. Let $w_{1}$ be the walk of $E$ with $s(w_{1})=s(u)$ and $f(w_{1})=w'_{1}$. Assume that $t(w_{1})=e$. Since $f$ induces a bijection $P(e)\rightarrow P'(e')$, there exists a unique element $h\in P(e)$ such that $f(h)=h'$. Since $f(g^{d(h)}\cdot h)=g^{d'(h')}\cdot h'$, there exists a unique walk $w_{2}$ of $E$ with $s(w_{2})=g^{d(h)}\cdot h$ and $f(w_{2})=w'_{2}$. Therefore $\widetilde{u}=w_{2}(g^{d(h)}\cdot h|\tau|g^{d(e)}\cdot e)(g^{d(e)}\cdot e|g^{d(e)}|e)w_{1}$, $\widetilde{v}=w_{2}(g^{d(h)}\cdot h|g^{d(h)}|h)(h|\tau|e)w_{1}$ are two walks of $E$ such that $s(\widetilde{u})=s(\widetilde{v})=s(u)$, $f(\widetilde{u})=f(u)$, $f(\widetilde{v})=f(v)$. Then $\widetilde{u}=u$ and $\widetilde{v}=v$. It implies that $u\sim v$. When the replacement is by using relation $(h4)$ in Definition \ref{homotopy of walks}, it can be shown similarly that $v$ is also obtained from $u$ by replacing a subwalk to another, by using the same relation. Therefore $u$ is  homotopic to $v$.
\end{proof}

\begin{Prop}\label{homotopy lifting}
Let $f:E\rightarrow E'$ be a covering of f-BCs, $u$, $v$ be two walks of $E$ with $s(u)=s(v)$ (or $t(u)=t(v)$). Then $u\sim v$ if and only if $f(u)\sim f(v)$.
\end{Prop}

\begin{proof}
It follows from Lemma \ref{morphism preserves homotopic walks} and Lemma \ref{covering preserves homotopic walks}.
\end{proof}

\medskip
Let $f:E\rightarrow E'$ be a morphism of f-BCs and $e\in E$. By Lemma \ref{morphism preserves homotopic walks}, $f$ maps homotopic walks to homotopic walks, therefore it induces a map $f_{*}:\Pi(E,e)\rightarrow \Pi(E',f(e))$, which is a group homomorphism. Moreover, for each subset $A$ of $E$ and each subset $A'$ of $E'$ with $f(A)\subseteq A'$, $f$ induces a functor $f_{*}:\Pi(E,A)\rightarrow \Pi(E',A')$.

According to Proposition \ref{homotopy lifting}, we have

\begin{Thm}\label{covering-injective}
If $f:E\rightarrow E'$ is a covering of f-BCs, then for any $e\in E$ the map $f_{*}:\Pi(E,e)\rightarrow \Pi(E',f(e))$ is injective.
\end{Thm}

Let $E=(E,P,L,d)$ be an f-BC and $e\in E$. Define a set $\widetilde{E}$ as follows: $\widetilde{E}=\{\overline{w}\mid s(w)=e\}$, where $\overline{w}$ denotes the homotopy class of the walk $w$ in $E$; define the action of $G=\langle g\rangle$ on $\widetilde{E}$ by the formula $$g^{n}\cdot \overline{w}=\overline{(g^{n}\cdot t(w)|g^{n}|t(w))w};$$
define two partitions $\widetilde{P}$, $\widetilde{L}$ of $\widetilde{E}$ by $\widetilde{P}(\overline{w})=\{\overline{(h|\tau|t(w))w}\mid h\in P(t(w))\}$ (resp. $\widetilde{L}(\overline{w})=\{\overline{(h|\tau|t(w))w}\mid h\in L(t(w))\}$); define a function $\widetilde{d}:\widetilde{E}\rightarrow\mathbb{Z}_{>0}$ by $\widetilde{d}(\overline{w})=d(t(w))$.

\begin{Thm} \label{universal-cover-description}
If $E=(E,P,L,d)$ is an f-BC (resp. $f_s$-BC), then $\widetilde{E}=(\widetilde{E},\widetilde{P},\widetilde{L},\widetilde{d})$ is an f-BC (resp. $f_s$-BC). Moreover, the map $p:\widetilde{E}\rightarrow E$, $\overline{w}\mapsto t(w)$ is a covering.
\end{Thm}

\begin{proof}
When $E$ is an f-BC, it can be shown straightforward that $\widetilde{E}$ satisfies conditions $(f1)$ to $(f5)$ in Definition \ref{f-BC}. If $\widetilde{E}$ does not satisfy condition $(f6)$, then there exist $\overline{w}$, $\overline{v}\in\widetilde{E}$ such that $\widetilde{L}(g^{\widetilde{d}( \overline{v})-1}\cdot \overline{v})\cdots\widetilde{L}(g\cdot \overline{v})\widetilde{L}(\overline{v})$ is a proper subsequence of $\widetilde{L}(g^{\widetilde{d}( \overline{w})-1}\cdot \overline{w})\cdots\widetilde{L}(g\cdot \overline{w})\widetilde{L}(\overline{w})$. Therefore $\widetilde{d}(\overline{w})>\widetilde{d}(\overline{v})$. We may assume that $\widetilde{L}(\overline{v})=\widetilde{L}(\overline{w})$, $\widetilde{L}(g\cdot \overline{v})=\widetilde{L}(g\cdot \overline{w})$, $\cdots$, $\widetilde{L}(g^{\widetilde{d}(\overline{v})-1}\cdot \overline{v})=\widetilde{L}(g^{\widetilde{d}(\overline{v})-1}\cdot \overline{w})$. Let $t(w)=e$ and $t(v)=h$. For each $0\leq i\leq\widetilde{d}(\overline{v})-1$, since $\widetilde{L}(g^{i}\cdot\overline{v})=\widetilde{L}(g^{i}\cdot\overline{w})$, $$\overline{(g^{i}\cdot h|g^{i}|h)v}=\overline{(z_{i}|\tau|g^{i}\cdot e)(g^{i}\cdot e|g^{i}|e)w}$$
for some $z_i\in L(g^{i}\cdot e)$. Therefore $g^{i}\cdot h=z_i$ and $L(g^{i}\cdot h)=L(z_{i})=L(g^{i}\cdot e)$ for all $0\leq i\leq \widetilde{d}(\overline{v})-1$. Moreover, $d(h)=\widetilde{d}(\overline{v})$ and $d(e)=\widetilde{d}(\overline{w})$. Therefore $$L(g^{d(h)-1}\cdot h)\cdots L(g\cdot h)L(h)$$ is a proper subsequence of $$L(g^{d(e)-1}\cdot e)\cdots L(g\cdot e)L(e),$$ a contradiction. Then $\widetilde{E}$ is an f-BC.

Since two walks with different terminals are not homotopic, the map $p:\widetilde{E}\rightarrow E$, $\overline{w}\mapsto t(w)$ is well-defined. It can be shown straightforward that $p$ is a morphism of f-BCs and $p$ induces bijections $\widetilde{P}(\overline{w})\rightarrow P(t(w))$ and $\widetilde{L}(\overline{w})\rightarrow L(t(w))$. Therefore $p$ is a covering. By Proposition \ref{covering preserves the property of being a configuration}, $\widetilde{E}$ is an $f_s$-BC if $E$ is an $f_s$-BC.
\end{proof}

\medskip
We call $\widetilde{E}$ the {\it universal cover} of $E$ at $e$. If $E$ is connected, then $\widetilde{E}$ is independent to the choice of $e$ up to isomorphism (see the remarks after Corollary \ref{property of universal cover}). We first give a characterization of $\widetilde{E}$ by showing that $\widetilde{E}$ is simply connected.

\begin{Lem}\label{terminal}
Let $E=(E,P,L,d)$ be an f-BC, $\widetilde{E}=(\widetilde{E},\widetilde{P},\widetilde{L},\widetilde{d})$ be the universal cover of $E$ at $e$, $p:\widetilde{E}\rightarrow E$, $\overline{w}\mapsto t(w)$. If $\eta$ is a walk of $\widetilde{E}$ with $s(\eta)=\overline{w_1}$ and $t(\eta)=\overline{w_2}$, then $\overline{w_2}=\overline{p(\eta)w_1}$.
\end{Lem}

\begin{proof}
We shall give a proof by induction on the length $l(\eta)$ of $\eta$. When $l(\eta)=0$ the conclusion holds. Assume it holds for $l(\eta)<n$. When $l(\eta)=n$, let $\eta=(\overline{w_2}|\delta|\overline{v})\eta'$, where $\delta\in\{g,g^{-1},\tau\}$, $\eta'$ is a walk of $\widetilde{E}$ of length $n-1$. By induction, we have $\overline{v}=\overline{p(\eta')w_1}$. If $\delta=g$, then
$$\overline{w_2}=g\cdot \overline{v}=\overline{(g\cdot t(v)|g|t(v))v}=\overline{(g\cdot t(v)|g|t(v))p(\eta')w_1}=\overline{p((\overline{w_2}|g|\overline{v}))p(\eta')w_1}=\overline{p(\eta)w_1}.$$ If $\delta=g^{-1}$, the proof is similar. If $\delta=\tau$, then $P(t(w_2))=P(t(v))$ and $\overline{w_2}=\overline{(t(w_2)|\tau|t(v))v}$. Since $\overline{v}=\overline{p(\eta')w_1}$,  $\overline{w_2}=\overline{(t(w_2)|\tau|t(v))p(\eta')w_1}=\overline{p((\overline{w_2}|\tau|\overline{v}))p(\eta')w_1}=\overline{p(\eta)w_1}$.
\end{proof}

\begin{Lem}\label{universal cover is connected}
Let $E$ be an f-BC, $\widetilde{E}$ be the universal cover of $E$ at $e$. Then $\widetilde{E}$ is connected.
\end{Lem}

\begin{proof}
For $\overline{w}\in\widetilde{E}$, assume that $w=e_{n}\frac{\delta_{n}}{}e_{n-1}\frac{\delta_{n-1}}{}\cdots\frac{\delta_{3}}{}e_{2}\frac{\delta_{2}}{}e_{1}\frac{\delta_{1}}{}e_{0}$, where $\delta_{i}\in\{g,g^{-1},\tau\}$ for $1\leq i\leq n$ and $e_{0}=e$. Let $w_{i}=e_{i}\frac{\delta_{i}}{}e_{i-1}\frac{\delta_{i-1}}{}\cdots\frac{\delta_{3}}{}e_{2}\frac{\delta_{2}}{}e_{1}\frac{\delta_{1}}{}e_{0}$ for each $0\leq i\leq n$. Then
$$\overline{w_{n}}\frac{\delta_{n}}{}\overline{w_{n-1}}\frac{\delta_{n-1}}{}\cdots\frac{\delta_{3}}{}\overline{w_{2}}\frac{\delta_{2}}{}\overline{w_{1}}\frac{\delta_{1}}{}\overline{w_{0}}$$ is a walk of $\widetilde{E}$ connecting $\overline{w_{n}}=\overline{w}$ and $\overline{w_{0}}=\overline{(e||e)}$. Therefore $\widetilde{E}$ is connected.
\end{proof}

\begin{Lem}\label{independence}
If $E$ is a connected f-BC, then $\Pi(E,e)$ are isomorphic for different $e\in E$.
\end{Lem}

\begin{proof}
For $e_1$, $e_2\in E$, since $E$ is connected, there exists a walk $w$ of $E$ with $s(w)=e_1$ and $t(w)=e_2$. Then $\phi:\Pi(E,e_1)\rightarrow \Pi(E,e_2)$, $\overline{v}\mapsto\overline{wvw^{-1}}$ is a group isomorphism, whose inverse is $\psi:\Pi(E,e_2)\rightarrow \Pi(E,e_1)$, $\overline{v}\mapsto\overline{w^{-1}vw}$.
\end{proof}

\medskip
By Lemma \ref{independence}, for a connected f-BC $E$, we may write $\Pi(E,e)=\Pi(E)$, and call it the {\it fundamental group of $E$}.

\begin{Def}
A connected f-BC $E$ is said to be simply connected if $\Pi(E)=\{1\}$.
\end{Def}

\begin{Prop}\label{universal cover is simply connected}
Let $E$ be an f-BC, $\widetilde{E}$ be the universal cover of $E$ at $e$. Then $\widetilde{E}$ is simply connected.
\end{Prop}

\begin{proof}
Denote by $p$ the covering $\widetilde{E}\rightarrow E$, $\overline{w}\mapsto t(w)$. By Lemma \ref{universal cover is connected}, $\widetilde{E}$ is connected. For $\overline{\eta}\in\Pi(\widetilde{E},\overline{(e||e)})$, by Lemma \ref{terminal}, $\overline{(e||e)}=t(\eta)=\overline{p(\eta)}s(\eta)=\overline{p(\eta)}\overline{(e||e)}=\overline{p(\eta)}$. Since $\eta$, $(\overline{(e||e)}||\overline{(e||e)})$ are two walks of $\widetilde{E}$ with  $s(\eta)=s((\overline{(e||e)}||\overline{(e||e)}))=\overline{(e||e)}$ and $p(\eta)\sim(e||e)=p((\overline{(e||e)}||\overline{(e||e)}))$, by Proposition \ref{homotopy lifting}, $\eta\sim(\overline{(e||e)}||\overline{(e||e)})$. Therefore $\Pi(\widetilde{E},\overline{(e||e)})$ is trivial.
\end{proof}

The following result determines when there exists a morphism between two coverings $E_{1}\rightarrow E$ and $E_{2}\rightarrow E$.

\begin{Prop}\label{existence of morphism}
Let $E=(E,P,L,d)$, $E_{1}=(E_1,P_1,L_1,d_1)$, $E_{2}=(E_2,P_2,L_2,d_2)$ be f-BCs with $E_{1}$ connected, $f_{1}:E_{1}\rightarrow E$ a morphism of f-BCs, $f_{2}:E_{2}\rightarrow E$ a covering of f-BCs. For $e_i\in E_{i}$ $(i=1,2)$ with $f_{1}(e_1)=f_{2}(e_2)$, there exists a morphism $\phi:E_{1}\rightarrow E_{2}$ of f-BCs such that $f_{1}=f_{2}\phi$ and $\phi(e_1)=e_2$ if and only if $f_{1*}(\Pi(E_1,e_1))\subseteq f_{2*}(\Pi(E_2,e_2))$.
\end{Prop}

\begin{proof}
"$\Leftarrow$" Let $e=f_{1}(e_1)$. For each $h_{1}\in E_{1}$, since $E_1$ is connected, we choose a walk $w_1$ of $E_1$ with $s(w_1)=e_1$ and $t(w_1)=h_1$. Then $f_{1}(w_1)$ is a walk of $E$ with $s(f_{1}(w_1))=e$. Since $f_2$ is a covering, there exists a unique walk $w_2$ of $E_2$ such that $s(w_2)=e_2$ and $f_{2}(w_2)=f_{1}(w_1)$. Define $\phi(h_1)=t(w_2)$. To show $\phi$ is well-defined, we need to show that the value of $\phi(h_1)$ is independent of the choice of $w_1$: Let $v_1$ be another walk of $E_1$ from $e_1$ to $h_1$, then $\overline{(v_1)^{-1}w_1}\in\Pi(E_1,e_1)$ and $\overline{(f_{1}(v_1))^{-1}f_{1}(w_1)}\in f_{1*}(\Pi(E_1,e_1))\subseteq f_{2*}(\Pi(E_2,e_2))$. Therefore there exists a closed walk $z$ of $E_2$ at $e_2$ such that $f_{2}(z)\sim (f_{1}(v_1))^{-1}f_{1}(w_1)$. Let $w_2$ be the walk of $E_2$ such that $s(w_2)=e_2$ and $f_{2}(w_2)=f_{1}(w_1)$, $v_2$ be the walk of $E_2$ such that $t(v_2)=t(w_2)$ and $f_{2}(v_2)=f_{1}(v_1)$. Since $z$, $(v_2)^{-1}w_2$ are walks of $E_2$ with $s(z)=s((v_2)^{-1}w_2)$ and $f_{2}(z)\sim (f_{1}(v_1))^{-1}f_{1}(w_1)=f_{2}((v_2)^{-1}w_2)$, by Proposition \ref{homotopy lifting}, $z\sim (v_2)^{-1}w_2$. Therefore $(v_2)^{-1}w_2$ is a closed walk of $E_2$ at $e_2$, and $v_2$ is the walk of $E_2$ with $s(v_2)=e_2$ and $f_{2}(v_2)=f_{1}(v_1)$. Since $t(v_2)=t(w_2)$, $\phi(h_1)$ is well-defined. It can be checked that $\phi$ is a map of $G$-sets, and $\phi(P_1(h_1))\subseteq P_2(\phi(h_1))$ (resp. $\phi(L_1(h_1))\subseteq L_2(\phi(h_1))$) for each $h_1\in E_1$. Since $f_{1}=f_{2}\phi$ and both $f_1$ and $f_2$ preserve the degrees, $\phi$ also preserves the degrees. Then $\phi:E_{1}\rightarrow E_{2}$ is a morphism of f-BCs such that $f_{1}=f_{2}\phi$ and $\phi(e_1)=e_2$.

"$\Rightarrow$" Since $f_{1}=f_{2}\phi$, $f_{1*}(\Pi(E_1,e_1))=f_{2*}(\phi_{*}(\Pi(E_1,e_1)))\subseteq f_{2*}(\Pi(E_2,e_2))$.
\end{proof}

\begin{Rem1}\label{uniqueness}
If $\phi':E_{1}\rightarrow E_{2}$ is another morphism of f-BCs which satisfies the conditions $f_{1}=f_{2}\phi'$ and $\phi'(e_1)=e_2$ in Proposition \ref{existence of morphism}, then for each $h_{1}\in E_{1}$ and for each walk $w_1$ of $E_1$ from $e_1$ to $h_1$, $\phi'(w_1)$ is a walk of $E_2$ which start at $e_2$ satisfying $f_2(\phi'(w_1))=f_1(w_1)$. Then $\phi(h_1)=t(\phi'(w_1))=\phi'(h_1)$. Therefore such $\phi$ in Proposition \ref{existence of morphism} is unique.
\end{Rem1}

\begin{Cor}\label{property of universal cover}
Let $E=(E,P,L,d)$ be an f-BC, $\widetilde{E}=(\widetilde{E},\widetilde{P},\widetilde{L},\widetilde{d})$ be the universal cover of $E$ at $e$, $p:\widetilde{E}\rightarrow E$, $\overline{v}\mapsto t(v)$. Then for each covering $f:E'\rightarrow E$ and for all $\overline{w}\in\widetilde{E}$, $e'\in E'$ with $p(\overline{w})=f(e')$, there exists a unique covering $\phi:\widetilde{E}\rightarrow E'$ such that $\phi(\overline{w})=e'$ and $p=f\phi$.
\end{Cor}

\begin{proof}
By Proposition \ref{universal cover is simply connected}, $\Pi(\widetilde{E},\overline{w})$ is trivial. Therefore $p_{*}(\Pi(\widetilde{E},\overline{w}))=\{1\}\subseteq f_{*}(\Pi(E',e'))$. By Proposition \ref{existence of morphism}, there exists a morphism $\phi:\widetilde{E}\rightarrow E'$ such that $\phi(\overline{w})=e'$ and $p=f\phi$. Let $E'=(E',P',L',d')$. Since for each $\overline{v}\in\widetilde{E}$, $p:\widetilde{P}(\overline{v})\rightarrow P(t(v))$ is the composition of $\phi:\widetilde{P}(\overline{v})\rightarrow P'(\phi(\overline{v}))$ and $f:P'(\phi(\overline{v}))\rightarrow P(t(v))$, and since $p:\widetilde{P}(\overline{v})\rightarrow P(t(v))$, $f:P'(\phi(\overline{v}))\rightarrow P(t(v))$ are bijective, $\phi:\widetilde{P}(\overline{v})\rightarrow P'(\phi(\overline{v}))$ is also bijective. Similarly, $\phi:\widetilde{L}(\overline{v})\rightarrow L(\phi(\overline{v}))$ is bijective for each $\overline{v}\in\widetilde{E}$. Therefore $\phi$ is a covering of f-BCs. By Remark \ref{uniqueness}, such $\phi$ is unique.
\end{proof}

Therefore the universal covering map $p:\widetilde{E}\rightarrow E$ satisfies a universal property. Moreover, for a connected f-BC $E$, two universal covers of $E$ at different angles are isomorphic.

\begin{Def}
Let $E,E_1,E_2$ be f-BCs. Two coverings of f-BCs $f_{1}:E_{1}\rightarrow E$ and $f_{2}:E_{2}\rightarrow E$ are said to be isomorphic if there exists an isomorphism of f-BCs $\phi:E_1\rightarrow E_2$ such that $f_1=f_2 \phi$.
\end{Def}

\begin{Lem}\label{conjugacy class}
Let $f:E\rightarrow E'$ be a covering of f-BCs with $E$ connected and $e'\in E'$. If $f^{-1}(e')\neq\emptyset$ then the subgroups $f_{*}(\Pi(E,e))$ for $e\in f^{-1}(e')$ form a conjugacy class of subgroups of $\Pi(E',e')$.
\end{Lem}

\begin{proof}
For any $x,y\in f^{-1}(e')$, since $E$ is connected, we may choose a walk $w$ of $E$ from $x$ to $y$. Then $\Pi(E,y)=\overline{w}\Pi(E,x)\overline{w^{-1}}$ and $f_{*}(\Pi(E,y))=\overline{f(w)}f_{*}(\Pi(E,x))\overline{f(w)}^{-1}$. So $f_{*}(\Pi(E,x))$ and $f_{*}(\Pi(E,y))$ belong to the same conjugacy class. Conversely, for any $x\in f^{-1}(e')$ and for any $\overline{v'}\in\Pi(E',e')$, let $v$ be the walk of $E$ with $s(v)=x$ and $f(v)=v'$. Then the terminal $y$ of $v$ belongs to $f^{-1}(e')$ and $f_{*}(\Pi(E,y))=\overline{v'}f_{*}(\Pi(E,x))\overline{v'}^{-1}$.
\end{proof}

The following result shows that the isomorphism class of a covering $f:E\rightarrow E'$ can be determined by the image of the induced homomorphism $f_{*}:\Pi(E,e)\rightarrow\Pi(E',f(e))$ (up to conjugacy).

\begin{Prop}\label{classification-of-coverings-by-fundamental-group}
Let $E,E_1,E_2$ be connected f-BCs and $f_{i}:E_{i}\rightarrow E$ ($i=1,2$) be two coverings of f-BCs. Suppose that $e\in E$, then $f_1$ and $f_2$ are isomorphic if and only if the conjugacy classes $\{f_{1*}(\Pi(E_1,x))\mid x\in f_{1}^{-1}(e)\}$ and $\{f_{2*}(\Pi(E_2,y))\mid y\in f_{2}^{-1}(e)\}$ of subgroups of $\Pi(E,e)$ are equal.
\end{Prop}

\begin{proof}
Since both $E,E_1,E_2$ are connected, by Lemma \ref{conjugacy class}, $\{f_{i*}(\Pi(E_i,x))\mid x\in f_{i}^{-1}(e)\}$ forms a conjugacy class of subgroups of $\Pi(E,e)$ for $i=1,2$.

"$\Rightarrow$" Since $f_1$ and $f_2$ are isomorphic, there exists an isomorphism of f-BCs $\phi:E_1\rightarrow E_2$ such that $f_1=f_2 \phi$. Choose $x_0\in f_{1}^{-1}(e)$ and let $y_0=\phi(x_0)$. Then $f_{1*}(\Pi(E_1,x_0))=f_{2*}(\phi_{*}(\Pi(E_1,x_0)))=f_{2*}(\Pi(E_2,y_0))$ since $\phi$ is an isomorphism. Therefore the conjugacy classes $\{f_{1*}(\Pi(E_1,x))\mid x\in f_{1}^{-1}(e)\}$ and $\{f_{2*}(\Pi(E_2,y))\mid y\in f_{2}^{-1}(e)\}$ of subgroups of $\Pi(E,e)$ are equal.

"$\Leftarrow$" We may choose some $x\in f_{1}^{-1}(e)$ and $y\in f_{2}^{-1}(e)$ such that $f_{1*}(\Pi(E_1,x))=f_{2*}(\Pi(E_2,y))$. By Proposition \ref{existence of morphism}, there exist morphisms $\phi:E_1\rightarrow E_2$ and $\psi:E_2\rightarrow E_1$ of f-BCs with $\phi(x)=y$, $\psi(y)=x$ such that $f_1=f_2 \phi$ and $f_2=f_1 \psi$. By Remark \ref{uniqueness}, both $\psi\phi$ and $\phi\psi$ are identity maps, so $\phi$ is an isomorphism of f-BCs. Therefore $f_1$ and $f_2$ are isomorphic.
\end{proof}

We now define a more general notion than the universal covering, which can be described by the quotient of f-BCs.

\begin{Def}\label{regular covering}
Let $E,E'$ be connected f-BCs. A covering $f:E\rightarrow E'$ is said to be regular if $f_{*}(\Pi(E,e))$ is a normal subgroup of $\Pi(E',f(e))$ for some $e\in E$.
\end{Def}

\begin{Rem1}
If $f:E\rightarrow E'$ is a regular covering, then it can be shown that $f_{*}(\Pi(E,e))$ is a normal subgroup of $\Pi(E',f(e))$ for every $e\in E$.
\end{Rem1}

For a covering $f:E\rightarrow E'$ of f-BCs, define the {\it automorphism group} Aut$(f)$ of $f$:
$$\mathrm{Aut}(f):=\{\phi:E\rightarrow E\mid \phi \text{ is an isomorphism of f-BCs such that } f\phi=f\}.$$

The following result characterizes the automorphism group of a regular covering.

\begin{Prop}\label{automorphism-group-of-regular-covering}
Let $E,E'$ be connected f-BCs and let $f:E\rightarrow E'$ be a regular covering. Fix $e\in E$ and set $e'=f(e)$. Then there exists a group isomorphism $\mu:\mathrm{Aut}(f)\rightarrow\Pi(E',e')/f_{*}(\Pi(E,e))$ which is defined as follows: For every $\phi\in\mathrm{Aut}(f)$, choose a walk $w$ of $E$ from $\phi(e)$ to $e$, and define $\mu(\phi)=[\overline{f(w)}]$, where $[\overline{f(w)}]$ denotes the coset $\overline{f(w)}f_{*}(\Pi(E,e))$.
\end{Prop}

\begin{proof}
If $v$ is another walk of $E$ from $\phi(e)$ to $e$, then $\overline{f(w)}\cdot\overline{f(v)}^{-1}=\overline{f(wv^{-1})}\in f_{*}(\Pi(E,e))$. So $[\overline{f(w)}]=[\overline{f(v)}]$ and $\mu$ is well-defined. For $\phi,\psi\in \mathrm{Aut}(f)$, choose a walk $w$ of $E$ from $\phi(e)$ to $e$ and choose a walk $v$ of $E$ from $\psi(e)$ to $e$. Since $w\phi(v)$ is a walk of $E$ from $\phi\psi(e)$ to $e$, $\mu(\phi\psi)=[\overline{f(w\phi(v))}]=[\overline{f(w)f(\phi(v))}]=[\overline{f(w)f(v)}]=\mu(\phi)\mu(\psi)$. Thus $\mu$ is a group homomorphism.

If $\phi\in\mathrm{ker}(\mu)$, then for every walk $w$ of $E$ from $\phi(e)$ to $e$, we have $\overline{f(w)}\in f_{*}(\Pi(E,e))$. Therefore there exists a closed walk $v$ of $E$ at $e$ such that $\overline{f(v)}=\overline{f(w)}$. According to Proposition \ref{homotopy lifting}, $v$ and $w$ are homotopic walks, so they have the same source. Then $\phi(e)=e$, and by Remark \ref{uniqueness}, $\phi=id_{E}$. Therefore $\mu$ is injective.

For each closed walk $w'$ of $E'$ at $e'$, let $w$ be the walk of $E$ with $t(w)=e$ and $f(w)=w'$. Since $f$ is a regular covering, by Lemma \ref{conjugacy class}, $f_{*}(\Pi(E,e))=f_{*}(\Pi(E,s(w)))$. Therefore by Proposition \ref{existence of morphism}, there exists a morphism $\phi:E\rightarrow E$ (resp. $\psi:E\rightarrow E$) such that $f\phi=f$ (resp. $f\psi=f$) and $\phi(e)=s(w)$ (resp. $\psi(s(w))=e$). According to Remark \ref{uniqueness}, $\psi$ is the inverse of $\phi$. Therefore $\phi\in\mathrm{Aut}(f)$, and $\mu(\phi)=[\overline{w'}]$, which implies that $\mu$ is surjective.
\end{proof}

\begin{Cor}
Let $E$ be a connected f-BC and $p:\widetilde{E}\rightarrow E$, $\overline{w}\mapsto t(w)$ be the universal cover of $E$ at $e$. Then Aut$(p)$ is isomorphic to $\Pi(E,e)$.
\end{Cor}

\begin{proof}
Since the fundamental group of $\widetilde{E}$ is trivial and since $p$ is regular, the result follows from Proposition \ref{automorphism-group-of-regular-covering}.
\end{proof}

Next we define the quotient f-BC $E/\Pi$ for an admissible group of automorphisms $\Pi$ of $E$.

\begin{Def}\label{admissible}
Let $E$ be an f-BC and $\Pi$ be a group of automorphisms of $E$. $\Pi$ is said to act admissibly on $E$ if each $\Pi$-orbit of $E$ meets each polygon of $E$ in at most one angle.
\end{Def}

Let $E=(E,P,L,d)$ be an f-BC, $\Pi$ be an admissible group of automorphisms of $E$. We can define the {\it quotient} $E/\Pi=(E/\Pi,P',L',d')$: $E/\Pi$ is the $G$-set of $\Pi$-orbits of $E$, where the $G$-set structure of $E/\Pi$ is inherited from the $G$-set structure of $E$; $P'$ (resp. $L'$) is a partition of $E/\Pi$ given by $P'([e])=\{[h]\mid h\in P(e)\}$ (resp. $L'([e])=\{[h]\mid h\in L(e)\}$); $d'$ is a function on $E/\Pi$ given by $d'([e])=d(e)$. Here we denote by $[e]$ the $\Pi$-orbit of $e\in E$.

\begin{Lem}\label{being f-BC}
$E/\Pi=(E/\Pi,P',L',d')$ is an f-BC.
\end{Lem}

\begin{proof}
Since each automorphism of $E$ commutes with the action of $G=\langle g\rangle$ on $E$, the definition of $g\cdot [e]$ is independent of the choice of representative of the $\Pi$-orbit $[e]$. If $x\in [e]$, assume $x=\phi(e)$ for some $\phi\in\Pi$. Since $\phi$ induces a bijection $P(e)\rightarrow P(x)$, $$P'([x])=\{[y]\mid y\in P(x)\}=\{[\phi(h)]\mid h\in P(e)\}=\{[h]\mid h\in P(e)\}=P'([e]).$$ Therefore the definition of $P'([e])$ is independent of the choice of representative of the $\Pi$-orbit $[e]$. Similarly, the definition of $L'([e])$ is also independent of the choice of representative of the $\Pi$-orbit $[e]$. To show $P'$ (resp. $L'$) is a partition of $E/\Pi$, we need to show that if $P'([e_1])\cap P'([e_2])\neq\emptyset$ (resp. $L'([e_1])\cap L'([e_2])\neq\emptyset$), then $P'([e_1])=P'([e_2])$ (resp. $L'([e_1])=L'([e_2])$). Let $[h]\in P'([e_1])\cap P'([e_2])$. Then $[h]=[h_1]=[h_2]$ for some $h_1\in P(e_1)$ and $h_2\in P(e_2)$. So $P'([e_1])=P'([h_1])=P'([h_2])=P'([e_2])$. If $L'([e_1])\cap L'([e_2])\neq\emptyset$, then we also have $L'([e_1])=L'([e_2])$. Since each automorphism of $E$ preserves the degree, the definition of $d'([e])$ is independent of the choice of representative of the $\Pi$-orbit $[e]$.

To show that $E/\Pi$ is an f-BC, we need to check the quadruple $(E/\Pi,P',L',d')$ satisfies $(f1)$ to $(f6)$ in Definition \ref{f-BC}. $(f1)$ follows from the definition of $P'$ and $L'$. If $L'([e_1])=L'([e_2])$, then $[e_2]=[h]$ for some $h\in L(e_1)$. Therefore $g\cdot h\in P(g\cdot e_1)$ and $$P'(g\cdot [e_2])=P'(g\cdot [h])=P'([g\cdot h])=P'([g\cdot e_1])=P'(g\cdot [e_1]),$$ so $(f2)$ holds. If $[e_1]$ and $[e_2]$ belong to the same $G$-orbit, then $[e_2]=[g^{n}\cdot e_1]$ for some $n\in\mathbb{Z}$, and $$d'([e_2])=d'([g^{n}\cdot e_1])=d(g^{n}\cdot e_1)=d(e_1)=d'([e_1]).$$ Therefore $(f3)$ holds. If $P'([e_1])=P'([e_2])$, then there exists $h\in P(e_1)$ such that $[h]=[e_2]$. Since $g^{d(h)}\cdot h\in P(g^{d(e_1)}\cdot e_1)$, $P'([g^{d(e_1)}\cdot e_1])=P'([g^{d(h)}\cdot h])$. Then
$$P'(g^{d'([e_1])}\cdot [e_1])=P'([g^{d(e_1)}\cdot e_1])=P'([g^{d(h)}\cdot h])=P'(g^{d'([h])}\cdot [h])=P'(g^{d'([e_2])}\cdot [e_2]).$$
Similarly, $P'(g^{d'([e_1])}\cdot [e_1])=P'(g^{d'([e_2])}\cdot [e_2])$ implies $P'([e_1])=P'([e_2])$. Therefore $(f4)$ holds, and $(f5)$ holds for the same reason. If the quadruple $(E/\Pi,P',L',m')$ does not satisfy $(f6)$, then there exist $[e]$, $[h]\in E/\Pi$ such that $d'([e])>d'([h])$ and $L'(g^i\cdot [e])=L'(g^i\cdot [h])$ for each $0\leq i\leq d'([h])-1$. Since $L'([e])=L'([h])$, we may assume that $L(e)=L(h)$. Since $L'(g\cdot [e])=L'(g\cdot [h])$, $[g\cdot h]=[h']$ for some $h'\in L(g\cdot e)$. Since $h\in L(e)$, $g\cdot h\in P(g\cdot e)$. Since both $g\cdot h$ and $h'$ belong to $P(g\cdot e)$ and since $\Pi$ acts admissibly on $E$, $g\cdot h=h'$. Therefore $L(g\cdot h)=L(g\cdot e)$. By induction, $L(g^i\cdot e)=L(g^i\cdot h)$ for each $0\leq i\leq d(h)-1$. Since $d(e)=d'([e])>d'([h])=d(h)$, $L(g^{d(h)-1}\cdot h)\cdots L(g\cdot h)L(h)$ is a proper subsequence of $L(g^{d(e)-1}\cdot e)\cdots L(g\cdot e)L(e)$, a contradiction. Therefore $(f6)$ holds for $(E/\Pi,P',L',d')$.
\end{proof}

\begin{Lem}\label{projection is a covering}
If $E=(E,P,L,d)$ is connected, then the natural projection $p:E\rightarrow E/\Pi$, $e\mapsto [e]$ is a regular covering of f-BCs.
\end{Lem}

\begin{proof}
It is straightforward to show that $p$ commutes with the action of $G=\langle g\rangle$ and preserves degrees. For $e\in E$, since $P'([e])=\{[h]\mid h\in P(e)\}$, $p$ induces a surjective map $P(e)\rightarrow P'([e])$. If $e_1$, $e_2\in P(e)$ and $p(e_1)=p(e_2)$, then $e_1$ and $e_2$ belong to the same $\Pi$-orbit. Since $\Pi$ acts admissibly on $E$, $e_1=e_2$. Therefore the map $P(e)\rightarrow P'([e])$ is bijective. It is straightforward to show that the map $L(e)\rightarrow L'([e])$ is also bijective, so $p$ is a covering.

To show that $p$ is regular, we fix $e\in E$. For each $\overline{w}\in p_{*}(\Pi(E,e))$ and for each $\overline{v}\in\Pi(E/\Pi,[e])$, let $\overline{w}=\overline{p(w')}$ with $w'$ a closed walk of $E$ at $e$, and let $v'$ be the walk of $E$ with $p(v')=v$ and $t(v')=e$. Since $[s(v')]=p(s(v'))=s(v)=[e]$, there exists $\phi\in\Pi$ such that $\phi(e)=s(v')$. Then $v'\phi(w')(v')^{-1}$ is a closed walk of $E$ at $e$ and $p(v'\phi(w')(v')^{-1})=p(v')p\phi(w')p(v')^{-1}=p(v')p(w')p(v')^{-1}=vp(w')v^{-1}\sim vwv^{-1}$. So $\overline{v}\cdot\overline{w}\cdot\overline{v}^{-1}=\overline{p(v'\phi(w')(v')^{-1})}=p_{*}(\overline{v'\phi(w')(v')^{-1}})\in p_{*}(\Pi(E,e))$, and $p_{*}(\Pi(E,e))$ is a normal subgroup of $\Pi(E/\Pi,[e])$.
\end{proof}

Now we characterize regular covering by using the quotient map.

\begin{Thm}\label{regular-covering-explicit}
Let $E,E'$ be connected f-BCs and let $f:E\rightarrow E'$ be a regular covering. Then Aut$(f)$ acts admissibly on $E$, and there exists an isomorphism $r:E/\mathrm{Aut}(f)\xrightarrow{\sim} E'$ such that the diagram
$$\xymatrix{
		& E \ar[dr]^{f}\ar[dl]_{p} &  \\
		E/\mathrm{Aut}(f)\ar[rr]_{r}^{\sim} & & E'
	}$$
commutes, where $p:E\rightarrow E/\mathrm{Aut}(f)$ is the natural projection.
\end{Thm}

\begin{proof}
Suppose that $E=(E,P,L,d)$ and $E'=(E',P',L',d')$. If there exist $e_1,e_2\in E$ which belong to the same polygon of $E$ such that there exists some $\phi\in\mathrm{Aut}(f)$ with $\phi(e_1)=e_2$, then $f(e_1)=f\phi(e_1)=f(e_2)$. Since $f$ is a covering, it induces a bijection $P(e_1)\rightarrow P'(f(e_1))$. Therefore $e_1=e_2$ and Aut$(f)$  acts admissibly on $E$.

Since $f$ maps every two angles of $E$ belonging to the same Aut$(f)$-orbit to the same angle of $E'$, it induces a map $r:E/\mathrm{Aut}(f)\rightarrow E'$. For every two angles $e_1,e_2$ of $E$ with $f(e_1)=f(e_2)$, since $f$ is regular, by Lemma \ref{conjugacy class}, $f_{*}(\Pi(E,e_1))=f_{*}(\Pi(E,e_2))$. By Proposition \ref{existence of morphism} and Remark \ref{uniqueness}, $e_1,e_2$ belong to the same Aut$(f)$-orbit. So each fiber of $f$ is just an Aut$(f)$-orbit of $E$. Therefore $r$ is injective. Since $E'$ is connected, $f$ is surjective, and thus $r$ is also surjective. It is straightforward to show that $r$ is a covering of f-BCs. Since $r$ is bijective, its inverse is also a covering of f-BCs. Therefore $r$ is an isomorphism of f-BCs.
\end{proof}

\subsection{The universal cover of $f_{ms}$-BC}
\

In this subsection, we construct explicitly the universal cover of any $f_{ms}$-BC $E$ (Corollary \ref{universality}) and apply it to describe the homotopy classes of walks in $E$ (Proposition \ref{unique factorization}). Recall from Definition \ref{Def-fms-BC} that an f-BC $E=(E,P,L,d)$ is said to be of type MS (or $E$ is an $f_{ms}$-BC in short) if $L(e)=\{e\}$ for each $e\in E$.

\begin{Def} \label{special-walk}
Let $E=(E,P,L,d)$ be an $f_{ms}$-BC. A walk $w$ of $E$ is called special if it is of the form \begin{multline*} (g^{i_k}\cdot e_k|g^{i_k}|e_k)(e_k|\tau|g^{i_{k-1}}\cdot e_{k-1})(g^{i_{k-1}}\cdot e_{k-1}|g^{i_{k-1}}|e_{k-1})(e_{k-1}|\tau|g^{i_{k-2}}\cdot e_{k-2})\cdots \\
(e_2|\tau|g^{i_{1}}\cdot e_{1})(g^{i_1}\cdot e_1|g^{i_1}|e_1)(e_1|\tau|g^{i_{0}}\cdot e_{0})(g^{i_0}\cdot e_0|g^{i_0}|e_0),  \end{multline*}
where $0\leq i_0<d(e_0)$, $0\leq i_k<d(e_k)$, $0< i_l<d(e_l)$ for all $1\leq l\leq k-1$, and $e_j\neq g^{i_{j-1}}\cdot e_{j-1}$ for all $1\leq j\leq k$. Write $w=(g^{i_k}\cdot e_k|g^{i_k}|e_k)(e_k | \tau g^{i_{k-1}}|e_{k-1})(e_{k-1}| \tau g^{i_{k-2}}|e_{k-2})\cdots (e_2| \tau g^{i_1}|e_1)(e_1| \tau g^{i_0}|e_0)$ for short or simply $w=g^{i_k}\tau g^{i_{k-1}}\tau\cdots\tau g^{i_1}\tau g^{i_0}$ if there is no confusion.
\end{Def}

Note that if $k=0$ and $i_0=0$ in the definition above, then $w=(e_0||e_0)$ is the walk of length $0$ at $e_0$.

\begin{Rem1}
The special walk can be defined for any f-BC but we do not pursue it in this paper.
\end{Rem1}

For $e\in E$, define an $f_{ms}$-BC $B_{(E,e)}=(B_{(E,e)},P',L',d')$ as follows. $B_{(E,e)} =\{$ special walks of $E$ which start at $e\}$. The action of $G$ on $B_{(E,e)}$ is given by
\begin{equation*}
g(w)= \begin{cases}
g^{i_{k}+1}\tau g^{i_{k-1}}\tau\cdots\tau g^{i_1}\tau g^{i_0}, & \text{if } w=g^{i_k}\tau g^{i_{k-1}}\tau\cdots\tau g^{i_1}\tau g^{i_0} \text{ with } i_k <d(t(w))-1; \\
\tau g^{i_{k-1}}\tau\cdots\tau g^{i_1}\tau g^{i_0}, & \text{if } w=g^{i_k}\tau g^{i_{k-1}}\tau\cdots\tau g^{i_1}\tau g^{i_0} \text{ with } i_k =d(t(w))-1.
\end{cases}
\end{equation*}
The partition $P'$ is given as follows. If $w=(e||e)$ is the walk of length $0$ at $e$, then
$$P'(w)=\{(h|\tau|e)\mid h\in P(e)\};$$
if the length of $w$ is larger than $0$, then
$$P'(w)=\{w\}\cup \{(h|\tau|t(w))w\mid h\in P(t(w)) \text{ and } h\neq t(w)\}$$
for $w=g^{i_k}\tau g^{i_{k-1}}\tau\cdots\tau g^{i_1}\tau g^{i_0}$ with $i_k >0$, and $P'(w)=P'(w')$ for $w=g^{i_k}\tau g^{i_{k-1}}\tau\cdots\tau g^{i_1}\tau g^{i_0}$ with $i_k =0$ and $w'=g^{i_{k-1}}\tau\cdots\tau g^{i_1}\tau g^{i_0}$. The partition $L'$ is given by $L'(w)=\{w\}$ for each $w\in B_{(E,e)}$. The degree function $d'$ is defined by $d'(w)=d(t(w))$, that is, $B_{(E,e)}$ is f-degree-free.

\begin{Ex1}
Let $E=(E,P,L,d)$ be a Brauer tree with $n$ edges and an exceptional vertex $v$ of f-degree $m$. Fix a half-edge $e\in v$, $B_{(E,e)}$ can be described as follows: Let
$$B=\{(h,i)\mid h\in E, i\in\{1,2,\cdots,m\}=\mathbb{Z}/m\mathbb{Z}\}$$
be a $G$-set, where the action of $G=\langle g\rangle$ on $B$ is given by
\begin{equation*}
g\cdot (h,i)=\begin{cases}
(g\cdot h,i), &\text{ if } h\neq g^{-1}\cdot e; \\
(g\cdot h,i+1), &\text{ if } h=g^{-1}\cdot e.
\end{cases}
\end{equation*}
Let $P_B$ be a partition of $B$ given by $P_B(h,i)=\{(h,i),(h',i)\}$ for any $(h,i)\in B$, where $P(h)=\{h,h'\}$, and let $L_B$ be a partition of $B$ given by $L_B(h,i)=\{(h,i)\}$ for any $(h,i)\in B$. Let $d_B:B\rightarrow\mathbb{Z}_{>0}$ be a function given by $d_B(h,i)=d(h)$ for any $(h,i)\in B$. Then $B=(B,P_B,L_B,d_B)$ becomes a Brauer tree with free f-degree, and $B\cong B_{(E,e)}$. Note that $B$ is obtained by gluing $m$ copies of $E$ at their common exceptional vertices. For example, if $E$ is a Brauer tree given by the diagram
$$\begin{tikzpicture}
\draw (-1,0)--(1,0);
\fill (0,0) circle (0.5ex);
\fill (-1,0) circle (0.5ex);
\fill (1,0) circle (0.5ex);
\draw (-0.2,0.2) rectangle (0.2,0.6);
\node at(0,0.4) {$2$};
\node at(1.5,0) {,};
\end{tikzpicture}$$
where the f-degree of the middle vertex is $2$. Then $B$ is an f-degree-free Brauer tree given by the diagram
$$\begin{tikzpicture}
\draw (-1,0)--(1,0);
\draw (0,-1)--(0,1);
\fill (0,0) circle (0.5ex);
\fill (-1,0) circle (0.5ex);
\fill (1,0) circle (0.5ex);
\fill (0,-1) circle (0.5ex);
\fill (0,1) circle (0.5ex);
\node at(1.5,0) {.};
\end{tikzpicture}$$
\end{Ex1}

\begin{Ex1}\label{example-Brauer-graph}
Let $E$ be the Brauer graph
$$
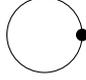
\begin{tikzpicture}
\draw (0,0) circle (0.5);
\fill (0.5,0) circle (0.5ex);
\end{tikzpicture}$$
with free f-degree. Fix a half-edge $e$ of $E$, $B_{(E,e)}$ is an infinite tree with free f-degree, which is given by the diagram
$$\begin{tikzpicture}
\fill (0,0) circle (0.5ex);
\draw (0,0)--(1,0);
\fill (1,0) circle (0.5ex);
\draw (1,0)--(2,0);
\fill (2,0) circle (0.5ex);
\draw (2,0)--(3,0);
\fill (3,0) circle (0.5ex);
\node at(-0.4,0) {$\cdot$};
\node at(-0.6,0) {$\cdot$};
\node at(-0.8,0) {$\cdot$};
\node at(3.4,0) {$\cdot$};
\node at(3.6,0) {$\cdot$};
\node at(3.8,0) {$\cdot$};
\end{tikzpicture}$$
\end{Ex1}

We remark that in general the underlying diagram of $B_{(E,e)}$ is not a tree if the $f_{ms}$-BC $E$ is not an $f_{ms}$-BG (that is, an $f_{ms}$-BC with each polygon containing exactly two elements).

It is straightforward to show that $B_{(E,e)}$ is a connected $f_{ms}$-BC. Denote by $\overline{\sigma}$ the element \\ $\overline{((e||e)|g^{d(e)}|(e||e))}$ of $\Pi(B_{(E,e)},(e||e))$.

\begin{Prop}\label{fundamental group of Be}
$\Pi(B_{(E,e)},(e||e))$ is an infinite cyclic group generated by $\overline{\sigma}$.
\end{Prop}

\begin{Lem}\label{terminal of special walk}
For a special walk
$$\eta=w_{n}\frac{\delta_{n}}{}w_{n-1}\frac{\delta_{n-1}}{}\cdots\frac{\delta_{3}}{}w_{2}\frac{\delta_{2}}{}w_{1}\frac{\delta_{1}}{}w_{0}$$
of $B_{(E,e)}$ which starts at $(e||e)$, where $\delta_{i}\in \{g,\tau\}$ for $1\leq i\leq n$ and $w_{i}\in B_{(E,e)}$ for $0\leq i\leq n$, define a sequence $$\phi(\eta)=t(w_{n})\frac{\delta_{n}}{}t(w_{n-1})\frac{\delta_{n-1}}{}\cdots\frac{\delta_{3}}{}t(w_{2})\frac{\delta_{2}}{}t(w_{1})
\frac{\delta_{1}}{}t(w_{0}).$$
Then $\phi(\eta)$ is a special walk of $E$ and $\phi(\eta)=w_{n}$.
\end{Lem}

\begin{proof}
We shall give a proof by induction on the length $l(\eta)$ of $\eta$. If $l(\eta)=0$, then $\eta=((e||e)||(e||e))$ and $\phi(\eta)=(e||e)$. Assume it is true for $l(\eta)<n$. When $l(\eta)=n$, assume that $\eta=w_{n}\frac{\delta_{n}}{}w_{n-1}\frac{\delta_{n-1}}{}\cdots\frac{\delta_{3}}{}w_{2}\frac{\delta_{2}}{}w_{1}\frac{\delta_{1}}{}w_{0}$.

When $\delta_n =g$, let $k$ be the minimal number such that $\delta_i =g$ for all $k+1\leq i\leq n$. Then $$\eta'=w_{k}\frac{\delta_{k}}{}w_{k-1}\frac{\delta_{k-1}}{}\cdots\frac{\delta_{3}}{}w_{2}\frac{\delta_{2}}{}w_{1}\frac{\delta_{1}}{}w_{0}$$
is a special walk of $B_{(E,e)}$ starting at $(e||e)$, and $l(\eta')<n$. By induction, $$\phi(\eta')=t(w_{k})\frac{\delta_{k}}{}t(w_{k-1})\frac{\delta_{k-1}}{}\cdots\frac{\delta_{3}}{}t(w_{2})\frac{\delta_{2}}{}t(w_{1})
\frac{\delta_{1}}{}t(w_{0})$$
is a special walk of $E$ and $\phi(\eta')=w_{k}$. If $k=0$, then $\eta=(g^{n}\cdot(e||e)|g^n|(e||e))$. Since $\eta$ is a special walk of $B_{(E,e)}$, $n<d'((e||e))=d(e)$. Therefore $\phi(\eta)=(g^{n}\cdot e|g^{n}|e)$ is a special walk of $E$ and $\phi(\eta)=g^{n}\cdot(e||e)=w_n$. If $k>0$, then $\delta_k=\tau$. Since $\eta$ is a special walk of $B_{(E,e)}$, $n-k<d'(w_k)=d(t(w_k))$. By induction,
$$w_k=t(w_{k})\frac{\delta_{k}}{}t(w_{k-1})\frac{\delta_{k-1}}{}\cdots\frac{\delta_{3}}{}t(w_{2})\frac{\delta_{2}}{}t(w_{1})
\frac{\delta_{1}}{}t(w_{0}).$$
Since $\delta_k=\tau$ and $n-k<d(t(w_k))$, $w_{k+i}=g^{i}\cdot w_k=(g^{i}\cdot t(w_k)|g^i|t(w_k))w_k$ for $0\leq i\leq n-k$. Then
$$\phi(\eta)=(g^{n-k}\cdot t(w_k)|g^{n-k}|t(w_k))w_k=g^{n-k}\cdot w_k=w_n$$
is a special walk of $E$.

When $\delta_n=\tau$, let
$$\eta'=w_{n-1}\frac{\delta_{n-1}}{}\cdots\frac{\delta_{3}}{}w_{2}\frac{\delta_{2}}{}w_{1}\frac{\delta_{1}}{}w_{0}.$$
By induction,
$$\phi(\eta')=t(w_{n-1})\frac{\delta_{n-1}}{}\cdots\frac{\delta_{3}}{}t(w_{2})\frac{\delta_{2}}{}t(w_{1})
\frac{\delta_{1}}{}t(w_{0})$$
is a special walk of $E$ and $\phi(\eta')=w_{n-1}$. Since $P'(w_n)=P'(w_{n-1})$ and $w_n \neq w_{n-1}$, $P(t(w_n))=P(t(w_{n-1}))$ and $t(w_n) \neq t(w_{n-1})$. Since $\eta$ is a special walk of $B_{(E,e)}$ and $\delta_n=\tau$, we have $\delta_{n-1}=g$. Since
$$\phi(\eta')=t(w_{n-1})\frac{\delta_{n-1}}{}\cdots\frac{\delta_{3}}{}t(w_{2})\frac{\delta_{2}}{}t(w_{1})
\frac{\delta_{1}}{}t(w_{0})$$
is a special walk of $E$,
$$\phi(\eta)=t(w_{n})\frac{\delta_{n}}{}t(w_{n-1})\frac{\delta_{n-1}}{}\cdots\frac{\delta_{3}}{}t(w_{2})\frac{\delta_{2}}{}t(w_{1})
\frac{\delta_{1}}{}t(w_{0})$$
is also a special walk of $E$. Since $P'(w_n)=P'(w_{n-1})$, $t(w_n) \neq t(w_{n-1})$, and $$w_{n-1}=t(w_{n-1})\frac{\delta_{n-1}}{}\cdots\frac{\delta_{3}}{}t(w_{2})\frac{\delta_{2}}{}t(w_{1})
\frac{\delta_{1}}{}t(w_{0})$$
with $\delta_{n-1}=g$, we infer that $w_n=(t(w_n)|\tau|t(w_{n-1}))w_{n-1}$.
Therefore
$$\phi(\eta)=(t(w_n)|\delta_n|t(w_{n-1}))\phi(\eta')=(t(w_n)|\tau|t(w_{n-1}))w_{n-1}=w_n.$$
\end{proof}

\begin{proof}[{\bf Proof of Proposition \ref{fundamental group of Be}}]
It is straightforward to show that each walk $w$ of an $f_{ms}$-BC $\overline{E}=(\overline{E},\overline{P},\overline{L},\overline{d})$ is homotopic to a walk of the form $(t(w)|g^{n\overline{d}(t(v))}|t(v))v$, where $v$ is a special walk of $\overline{E}$ and $n$ is an integer. Let $\xi$ be a closed walk of $B_{(E,e)}$ at $(e||e)$. Then $\xi$ is homotopic to a walk of the form $((e||e)|g^{n d(e)}|(e||e))\eta$, where $\eta$ is a special walk of $B_{(E,e)}$ starting at $(e||e)$ and $n$ is an integer. By Lemma \ref{terminal of special walk}, $\phi(\eta)=t(\eta)=(e||e)$. Therefore $\eta$ is trivial and $\xi$ is homotopic to $((e||e)|g^{n d(e)}|(e||e))$. Then $\overline{\xi}=\overline{\sigma}^{n}$ and $\Pi(B_{(E,e)},(e||e))$ is generated by $\overline{\sigma}$.

To show the order of $\overline{\sigma}$ is infinite, we need to define a number $\alpha(w)$ for each walk $w$ of an $f_{ms}$-BC $\overline{E}=(\overline{E},\overline{P},\overline{L},\overline{d})$ by induction on the length of $w$. Define $\alpha(w)=0$ if $w$ is trivial. For a walk $w$ of length $\geq 1$, write $w=(e_2|\delta|e_1)v$, where $\delta\in \{g,g^{-1},\tau\}$, define
\begin{equation}\label{alpha-definition}
\alpha(w)= \begin{cases}
\alpha(v)+1/\overline{d}(t(v)), & \text{if } \delta=g; \\
\alpha(v)-1/\overline{d}(t(v)), & \text{if } \delta=g^{-1}; \\
\alpha(v), & \text{if } \delta=\tau.
\end{cases}
\end{equation}
The values of $\alpha$ on homotopic walks are equal. For walk $w=((e||e)|g^{n d(e)}|(e||e))$ of $B_{(E,e)}$, $\alpha(w)=n$, so $((e||e)|g^{n d(e)}|(e||e))$ is not homotopic to the trivial walk of $B_{(E,e)}$ at $(e||e)$ when $n\neq 0$. Therefore $\overline{\sigma}^{n}\neq 1$ for $n\neq 0$.
\end{proof}

Let $E=(E,P,L,d)$ be an $f_{ms}$-BC and $e\in E$, we define an $f_{ms}$-BC $\mathbb{Z}B_{(E,e)}=(\mathbb{Z}B_{(E,e)},\widetilde{P},\widetilde{L},\widetilde{d})$ as follows: $\mathbb{Z}B_{(E,e)}=\{(w,n)\mid w\in B_{(E,e)}$ and $n\in\mathbb{Z}\}$. For each special walk \\ $w=g^{i_k}\tau g^{i_{k-1}}\tau\cdots\tau g^{i_1}\tau g^{i_0}$ of $E$ starting at $e$ and for each $n\in \mathbb{Z}$, define the action of $G=\langle g\rangle$ on $(w,n)\in\mathbb{Z}B_{(E,e)}$ by
\begin{equation*}
g\cdot(w,n)= \begin{cases}
(g\cdot w,n), & \text{if } i_k <d(t(w))-1; \\
(g\cdot w,n+1), & \text{if } i_k =d(t(w))-1.
\end{cases}
\end{equation*}
Define the partition $\widetilde{P}$ of $\mathbb{Z}B_{(E,e)}$ by $\widetilde{P}(w,n)=\{(w',n)\mid w'\in P'(w)\}$, and define the partition $\widetilde{L}$ of $\mathbb{Z}B_{(E,e)}$ by $\widetilde{L}(w,n)=\{(w,n)\}$. Define the degree function $\widetilde{d}$ of $\mathbb{Z}B_{(E,e)}$ by $\widetilde{d}(w,n)=d'(w)=d(t(w))$.

For example, if $E$ is the Brauer graph in Example \ref{example-Brauer-graph}, then $\mathbb{Z}B_{(E,e)}$ is an infinite $f_{ms}$-BG, which is given by the diagram
\begin{center}
\tikzset{every picture/.style={line width=0.75pt}}
\begin{tikzpicture}[x=15pt,y=15pt,yscale=1,xscale=1]
\fill (-2.5,-5) circle (0.5ex);
\fill (2.5,-5) circle (0.5ex);
\fill (7.5,-5) circle (0.5ex);
\draw    (-2.5,-5) .. controls (-1,-2) and (6,-2.5) .. (2.5,-5) ;
\draw    (-2.5,-5) .. controls (0,-7) and (7,-6) .. (2.5,-5) ;
\draw    (-2.5,-5) .. controls (-1.5,-10) and (5,-10) .. (2.5,-5) ;
\draw    (2.5,-5) .. controls (3,-2) and (9,-1) .. (7.5,-5) ;
\draw    (2.5,-5) .. controls (5,-4) and (11,-4) .. (7.5,-5) ;
\draw    (2.5,-5) .. controls (4,-7) and (10,-7) .. (7.5,-5) ;
\draw    (7.5,-5) .. controls (8.2,-4) and (9,-3) .. (10,-3) ;
\draw    (7.5,-5) .. controls (8.5,-5.5) and (9,-5.5) .. (10,-5.5) ;
\draw    (7.5,-5) .. controls (8,-7) and (9,-7) .. (10,-7) ;
\draw    (-2.5,-5) .. controls (-2.1,-3.5) and (-2,-3) .. (-4,-3) ;
\draw    (-2.5,-5) .. controls (0,-4.5) and (-2,-4.3) .. (-4,-4.3) ;
\draw    (-2.5,-5) .. controls (-1.25,-6.5) and (-2,-7) .. (-4,-7) ;
\node at(-5,-5) {$\cdot$};
\node at(-5.2,-5) {$\cdot$};
\node at(-4.8,-5) {$\cdot$};
\node at(-2.5,-3.9) {$\cdot$};
\node at(-2.65,-3.85) {$\cdot$};
\node at(-2.8,-3.8) {$\cdot$};
\node at(-2.5,-6.1) {$\cdot$};
\node at(-2.65,-6.15) {$\cdot$};
\node at(-2.8,-6.2) {$\cdot$};
\node at(2.6,-3.9) {$\cdot$};
\node at(2.45,-3.85) {$\cdot$};
\node at(2.3,-3.8) {$\cdot$};
\node at(2.7,-6.5) {$\cdot$};
\node at(2.55,-6.55) {$\cdot$};
\node at(2.4,-6.6) {$\cdot$};
\node at(7.5,-3.9) {$\cdot$};
\node at(7.35,-3.85) {$\cdot$};
\node at(7.2,-3.8) {$\cdot$};
\node at(8,-6.7) {$\cdot$};
\node at(7.85,-6.75) {$\cdot$};
\node at(7.7,-6.8) {$\cdot$};
\node at(11,-5) {$\cdot$};
\node at(11.2,-5) {$\cdot$};
\node at(10.8,-5) {$\cdot$};
\node at(12,-9) {,};
\end{tikzpicture}
\end{center}
where the degree function of $\mathbb{Z}B_{(E,e)}$ is identically equal to $2$.

\begin{Prop}\label{ZB is simply connected}
$\mathbb{Z}B_{(E,e)}$ is simply connected.
\end{Prop}

\begin{proof}
Define $f:\mathbb{Z}B_{(E,e)}\rightarrow B_{(E,e)}$, $(w,n)\mapsto w$. It can be shown that $f$ is a covering of f-BCs. For $(w,n)$, $(v,n')\in\mathbb{Z}B_{(E,e)}$, since $B_{(E,e)}$ is connected, we can choose a walk $\xi$ of $B_{(E,e)}$ from $w$ to $v$. Since $f$ is a covering, $\xi$ can be lifted to a walk $\Xi$ of $\mathbb{Z}B_{(E,e)}$ starting at $(w,n)$. Since $f(t(\Xi))=v$, $t(\Xi)=(v,l)$ for some $l\in\mathbb{Z}$. Since $(v,n')$ and $(v,l)$ are in the same $G$-orbit, they can be connected by a walk. Therefore $(w,n)$ and $(v,n')$ can be connected by a walk, which implies that $\mathbb{Z}B_{(E,e)}$ is connected. By Theorem \ref{covering-injective}, $f$ induces an injective group homomorphism $f_{*}:\Pi(\mathbb{Z}B_{(E,e)},((e||e),0))\rightarrow\Pi(B_{(E,e)},(e||e))$. For each closed walk $\Xi$ of $\mathbb{Z}B_{(E,e)}$ at $((e||e),0)$, by Proposition \ref{fundamental group of Be}, $\overline{f(\Xi)}=\overline{\sigma}^{n}$ for some $n\in\mathbb{Z}$. Therefore $f(\Xi)\sim ((e||e)|g^{n d(e)}|(e||e))$. By Proposition \ref{homotopy lifting}, $\Xi\sim (((e||e),n)|g^{n d(e)}|((e||e),0))$. We imply that $(((e||e),n)|g^{n d(e)}|((e||e),0))$ is a closed walk of $\mathbb{Z}B_{(E,e)}$, and $n=0$. Therefore $f_{*}(\Pi(\mathbb{Z}B_{(E,e)},((e||e),0)))$ is trivial, and $\Pi(\mathbb{Z}B_{(E,e)},((e||e),0))$ is also trivial.
\end{proof}

By the proposition above, we show that the f-BC $\mathbb{Z}B_{(E,e)}$ is the universal cover of $E$ at $e$.

\begin{Cor}\label{universality}
Let $E=(E,P,L,d)$ be an $f_{ms}$-BC and $e\in E$, and define $q:\mathbb{Z}B_{(E,e)}\rightarrow E$, $(w,n)\mapsto g^{n d(t(w))}\cdot t(w)$. Then $q$ is a covering of f-BCs. Let $\widetilde{E}$ be the universal cover of $E$ at $e$ and $p:\widetilde{E}\rightarrow E$, $\overline{v}\mapsto t(v)$. Then there exists a unique isomorphism of f-BCs $\phi:\mathbb{Z}B_{(E,e)}\rightarrow\widetilde{E}$ such that $q=p\phi$ and $\phi(((e||e),0))=\overline{(e||e)}$.
\end{Cor}

\begin{proof}
It is straightforward to show that $q$ is a covering of f-BCs. The fact that there is a unique isomorphism with the desired property follows from Proposition \ref{existence of morphism}, Remark \ref{uniqueness} and Proposition \ref{ZB is simply connected}.
\end{proof}

The following result describes the homotopy classes of walks in an $f_{ms}$-BC $E$, and it is useful when we calculate the fundamental groups of Brauer configurations in last section.

\begin{Prop}\label{unique factorization}
Let $E=(E,P,L,d)$ be an $f_{ms}$-BC, $w$ be a walk of $E$. Then there exist a unique special walk $v$ and a unique integer $n$ such that $w\sim (t(w)|g^{n d(t(w))}|t(v))v$.
\end{Prop}

\begin{proof}
Since each walk $w$ of $E$ is homotopic to a walk of this form, we only need to show that such $v$ and $n$ are unique. Let
$$w\sim (t(w)|g^{n d(t(w))}|t(v))v \sim (t(w)|g^{n' d(t(w))}|t(v'))v',$$
where $v$, $v'$ are special walks and $n$, $n'\in\mathbb{Z}$. Let $e=s(w)$ and $q:\mathbb{Z}B_{(E,e)}\rightarrow E$, $(u,k)\mapsto g^{k d(t(u))}\cdot t(u)$ be the covering in Corollary \ref{universality}. Let
$$v=e_{N}\frac{\delta_{N}}{}e_{N-1}\frac{\delta_{N-1}}{}\cdots\frac{\delta_{3}}{}e_{2}\frac{\delta_{2}}{}e_{1}\frac{\delta_{1}}{}e_{0},$$
where $e_{0}=e,e_{1},\cdots,e_{N}\in E$, $\delta_{1},\cdots,\delta_{N}\in\{g,\tau\}$. Define $$v_i=e_{i}\frac{\delta_{i}}{}e_{i-1}\frac{\delta_{i-1}}{}\cdots\frac{\delta_{2}}{}e_{1}\frac{\delta_{1}}{}e_{0}$$
for each $0\leq i\leq N$. Since $v$ is a special walk of $E$, each $v_i$ is a special walk of $E$. It can be shown that $v$ lifts to a special walk $$V=(v_{N},0)\frac{\delta_{N}}{}(v_{N-1},0)\frac{\delta_{N-1}}{}\cdots\frac{\delta_{3}}{}(v_{2},0)\frac{\delta_{2}}{}(v_{1},0)
\frac{\delta_{1}}{}(v_{0},0)$$
of $\mathbb{Z}B_{(E,e)}$ starting at $((e||e),0)$ with $t(V)=(v,0)$. Similarly, $v'$ lifts to a special walk $V'$ of $\mathbb{Z}B_{(E,e)}$ starting at $((e||e),0)$ with $t(V')=(v',0)$. Therefore $(t(w)|g^{n d(t(w))}|t(v))v$ \\ (resp. $(t(w)|g^{n' d(t(w))}|t(v'))v'$) lifts to a walk $((v,n)|g^{n d(t(w))}|(v,0))V$ \\ (resp. $((v',n')|g^{n' d(t(w))}|(v',0))V'$) of $\mathbb{Z}B_{(E,e)}$ starting at $((e||e),0)$. Since
$$(t(w)|g^{n d(t(w))}|t(v))v \sim (t(w)|g^{n' d(t(w))}|t(v'))v',$$
by Proposition \ref{homotopy lifting},
$$((v,n)|g^{n d(t(w))}|(v,0))V\sim ((v',n')|g^{n' d(t(w))}|(v',0))V.$$
Therefore $(v,n)=(v',n')$.
\end{proof}

\section{The fundamental group of an $f_s$-BC $E$ is isomorphic to the fundamental group of the quiver with relations $(Q_E,I_E)$}\label{sec:isomorphic-fundamental-groups}

In this section, we first recall from Martinez-Villa and de la Pe\~na \cite{MV-DLP} (see also Bongartz-Gabriel \cite{BG}) the definition of the fundamental group $\Pi(Q,I)$ for a quiver with relations $(Q,I)$, and then show that for a connected $f_s$-BC $E$ and $e\in E$, there is a natural isomorphism between the fundamental group $\Pi(Q_E,I_E)$ (where $(Q_E,I_E)$ is defined in Definition \ref{f-BC algebra}) and the fundamental group $\Pi(E,e)$ of $E$ which is defined in Definition \ref{definition-fundamental-group-of-f-BC}. Moreover, we show that the fundamental group $\Pi(Q_E,I_E)$ is isomorphic to $\Pi(Q'_E,I'_E)$ where $(Q'_E,I'_E)$ is an admissible ideal presentation of $\Lambda_E$ (see Definition \ref{reduced-arrow}).

\begin{Def} {\rm(\cite[Definition 1.3]{MV-DLP})} \label{minimal-relation}
Let $Q$ be a locally finite quiver and $I$ be an ideal of the path category $kQ$ (not necessarily admissible). A relation $\rho=\sum_{i=1}^{n}\lambda_i u_i\in I(x,y)$ with $\lambda_i\in k^{*}$ and $u_1,u_2,\ldots,u_n$ pairwise distinct paths from $x$ to $y$, is a minimal relation if $n\geq 2$ and for every non-empty proper subset $K$ of $\{1,\cdots,n\}$, $\sum_{i\in K}\lambda_i u_i\notin I(x,y)$.
\end{Def}

Denote by $m(I)$ the set of minimal relations of $I$.

For a quiver $Q$, we denote by $Q_0$ and $Q_1$ its vertex set and arrow set respectively, and denote by $Q_1^{-1}$ the set of all formal inverses of arrows in $Q_1$. A {\it walk} of $Q$ is a sequence of the form $\alpha_n\cdots\alpha_2\alpha_1$, where each $\alpha_i\in Q_1\cup Q_1^{-1}$, such that the source of $\alpha_i$ is equal to the terminal of $\alpha_{i-1}$ for each $i$. Define an equivalence relation $\sim$ on the set of walks of $Q$: $\sim$ is generated by
\begin{itemize}
	\item $\alpha^{-1}\alpha\sim 1_{s(\alpha)}$ and $\alpha\alpha^{-1}\sim 1_{t(\alpha)}$ for every arrow $\alpha$ of $Q$;
    \item If $w_1\sim w_2$, then $uw_1 v\sim uw_2 v$ for every two walks $u,v$ of $Q$ such that the compositions make sense.
\end{itemize}
For a walk $w$ of $Q$, we denote by $[w]$ the equivalence class of $w$.

For a vertex $x$ of $Q$, denote by $\Pi(Q,x)$ the {\it fundamental group of $Q$ at $x$}, which is the group of equivalence classes of closed walks of $Q$ at $x$. Similarly, for a subset $A$ of $Q_0$, we denote by $\Pi(Q,A)$ the {\it fundamental groupoid of $Q$ on $A$}. When $A=Q_0$, we denote $\Pi(Q,A)$ as $\Pi(Q)$, which is called the {\it fundamental groupoid of the quiver $Q$}.

\begin{Def} {\rm(\cite[Section 1]{MV-DLP})} \label{fundamental-group-of-quiver-with-relations}
Let $Q$ be a connected locally finite quiver and $I$ be an ideal of the path category $kQ$ (not necessarily admissible). Define $\Pi(Q,I):=\Pi(Q,x)/N(Q,m(I),x)$ the fundamental group of $(Q,I)$, where $x$ is a vertex of $Q$ and $N(Q,m(I),x)$ is the normal subgroup of $\Pi(Q,x)$ generated by elements of the form $[\gamma^{-1}u^{-1}v\gamma]$, where $\gamma$ is a walk of $Q$ from $x$ to $y$ and $u,v$ are paths of $Q$ from $y$ to $z$ such that there exists some $\rho=\sum^{n}_{i=1}\lambda_i p_i\in m(I)$ with $u=p_1$ and $v=p_2$.
\end{Def}

Since $Q$ is connected, the isomorphism class of $\Pi(Q,I)$ does not depend on the choice of the vertex $x$ of $Q$. We denote by $\overline{w}$ the image of $[w]\in\Pi(Q,x)$ in $\Pi(Q,I)$.

Let $E$ be an f-BC and
$$w=e_{n}\frac{\delta_{n}}{}e_{n-1}\frac{\delta_{n-1}}{}\cdots\frac{\delta_{3}}{}e_{2}\frac{\delta_{2}}{}e_{1}\frac{\delta_{1}}{}e_{0}$$
be a walk of $E$, where $e_{0},e_{1},\cdots,e_{n}\in E$, $\delta_{1},\cdots,\delta_{n}\in\{g,g^{-1},\tau\}$. We define a walk $f(w)$ of $Q_E$: $f(w)=f(e_n|\delta_n|e_{n-1})\cdots f(e_2|\delta_2|e_1)f(e_1|\delta_1|e_0)$, where
\begin{equation*}
f(e_i|\delta_i|e_{i-1})= \begin{cases}
L(e_{i-1}), &\text{ if } \delta_i=g; \\
L(e_i)^{-1}, &\text{ if } \delta_i=g^{-1}; \\
1_{P(e_i)}, &\text{ if } \delta_i=\tau.
\end{cases}
\end{equation*}

\begin{Thm}\label{fundamental-group-of-fs-BC-and-fundamental-group-of-algebra-and-isomorphic}
Let $E$ be a connected $f_s$-BC and $e\in E$. Then there is an isomorphism $\Pi(E,e)\rightarrow \Pi(Q_E,I_E)$, $\overline{w}\mapsto\overline{f(w)}$.
\end{Thm}

We need some preparations before we prove Theorem \ref{fundamental-group-of-fs-BC-and-fundamental-group-of-algebra-and-isomorphic}. First we recall an equivalent definition of fundamental group of a quiver with relations.

\begin{Def}\label{homotopy-relation-with-respect-to-ideal}
Let $Q$ be a locally finite quiver and $I$ be an ideal of the path category $kQ$, let $\sim_{I}$ be the equivalence relation on the set of walks of $Q$ generated by
\begin{itemize}
	\item[$(E1)$] $\alpha^{-1}\alpha\sim_{I} 1_{s(\alpha)}$ and $\alpha\alpha^{-1}\sim_{I} 1_{t(\alpha)}$ for every arrow $\alpha$ of $Q$;
    \item[$(E2)$] $w_1\sim_{I} w_2$ if there exists some $\rho=\sum^{n}_{i=1}\lambda_i p_i\in m(I)$ with $w_1=p_1$ and $w_2=p_2$;
    \item[$(E3)$] If $w_1\sim_{I} w_2$, then $uw_1 v\sim_{I} uw_2 v$ for every two walks $u,v$ of $Q$ such that the compositions make sense.
\end{itemize}
\end{Def}

For a connected quiver and for a vertex $x$ of $Q$, it is straightforward to show that the fundamental group $\Pi(Q,I)$ is isomorphic to the group of equivalence classes of closed walks of $Q$ at $x$ under the equivalence relation $\sim_{I}$. We shall identify these two definitions for $\Pi(Q,I)$.

\begin{Lem}\label{injectivity}
Let $E=(E,P,L,d)$ be an $f_s$-BC and let $w$ be a walk of $E$. If $\gamma$ is a walk of $Q_E$ with $\gamma\sim_{I_E}f(w)$, then there exists some walk $w'$ of $E$ such that $w\sim w'$ and $f(w')=\gamma$.
\end{Lem}

\begin{proof}
Since $\gamma\sim_{I_E}f(w)$, there exists a sequence $\gamma_0=f(w)$, $\gamma_1$, $\cdots$, $\gamma_{k-1}$, $\gamma_{k}=\gamma$ of walks of $Q_E$ such that for each $1\leq i\leq k$, $\gamma_i$ is obtained from $\gamma_{i-1}$ by replacing a subwalk to another, using the relation $(E1)$ or the relation $(E2)$ in Definition \ref{homotopy-relation-with-respect-to-ideal}. By induction, we may assume that $k=1$. So $\gamma$ is obtained from $f(w)$ by replacing a subwalk to another, using the relation $(E1)$ or the relation $(E2)$ in Definition \ref{homotopy-relation-with-respect-to-ideal}.

If $\gamma$ is obtained from $f(w)$ by replacing a subwalk to another using relation $(E1)$, we first assume that $f(w)=\beta_1\alpha^{-1}\alpha\beta_0$ and $\gamma=\beta_1\beta_0$, where $\beta_0,\beta_1$ are two walks of $Q_E$ and $\alpha$ is an arrow of $Q_E$. Let $$w=e_{n}\frac{\delta_{n}}{}e_{n-1}\frac{\delta_{n-1}}{}\cdots\frac{\delta_{3}}{}e_{2}\frac{\delta_{2}}{}e_{1}\frac{\delta_{1}}{}e_{0},$$
where $e_{0},e_{1},\cdots,e_{n}\in E$, $\delta_{1},\cdots,\delta_{n}\in\{g,g^{-1},\tau\}$. Then we may assume that $\beta_0=f(w_0)$, $\beta_1=f(w_1)$, $\alpha=f(e_{r+1}|\delta_{r+1}|e_r)$, $\alpha^{-1}=f(e_{l}|\delta_{l}|e_{l-1})$, where
$$w_0=e_{r}\frac{\delta_{r}}{}e_{r-1}\frac{\delta_{r-1}}{}\cdots\frac{\delta_{2}}{}e_{1}\frac{\delta_{1}}{}e_{0},$$
$$w_1=e_{n}\frac{\delta_{n}}{}e_{n-1}\frac{\delta_{n-1}}{}\cdots\frac{\delta_{l+1}}{}e_{l},$$
$l\geq r+2$, and $\delta_i=\tau$ for every integer $r+1<i<l$. Therefore $\delta_{r+1}=g$, $\delta_l=g^{-1}$, and $e_{r+1}=g\cdot e_r$, $e_{l-1}=g\cdot e_{l}$. Moreover, we have $\alpha=L(e_r)=L(e_l)$. Since $\delta_i=\tau$ for every integer $r+1<i<l$, $P(e_{r+1})=P(e_{l-1})$. So we have
\begin{multline*}
(e_l|\delta_l|e_{l-1})\cdots(e_{r+2}|\delta_{r+2}|e_{r+1})(e_{r+1}|\delta_{r+1}|e_r)\sim(e_l|g^{-1}|e_{l-1})(e_{l-1}|\tau|e_{r+1})(e_{r+1}|g|e_r) \\
\sim(e_l|g^{-1}|e_{l-1})(e_{l-1}|g|e_l)(e_l|\tau|e_r)\sim(e_l|\tau|e_r).
\end{multline*}
Let $w'=w_1(e_l|\tau|e_r)w_0$. Then $w'\sim w$ and $f(w')=f(w_1)f(w_0)=\beta_1\beta_0=\gamma$. Dually, it can be shown that the result is also true when $f(w)=\beta_1\beta_0$ and $\gamma=\beta_1\alpha^{-1}\alpha\beta_0$, where $\alpha$ is an arrow of $Q_E$. When $\gamma$ is obtained from $f(w)$ by replacing a subwalk to another using the relation $\alpha\alpha^{-1}\sim_{I} 1_{t(\alpha)}$, the proof is similar.

If $\gamma$ is obtained from $f(w)$ by replacing a subwalk to another using relation $(E2)$, then we may assume that $f(w)=\beta_1 p_1\beta_0$ and $\gamma=\beta_1 p_2\beta_0$, where $\beta_0,\beta_1$ are two walks of $Q_E$ and $p_1,p_2$ are two paths of $Q_E$ such that there exists some $\rho=\sum^{n}_{i=1}\lambda_i q_i\in m(I_E)$ with $q_1=p_1$, $q_2=p_2$. Since $\rho$ is a minimal relation of $I_E$, according to \cite[Lemma 4.16]{LL}, each $q_i$ is a path in $\mathscr{E}$ and $q_i R q_j$ for every $1\leq i,j\leq n$ (for the definition of the set $\mathscr{E}$ and the equivalence relation $R$, see \cite[Definition 4.9]{LL}). Then there exist some standard sequences $l_1,l_2$ of $E$ such that $p_1=L(l_1)$, $p_2=L(l_2)$ and $\prescript{\wedge}{}{l_1}\equiv \prescript{\wedge}{}{l_2}$. Since $f(w)=\beta_1 p_1\beta_0$, we may assume that $w=w_1 u w_0$, such that $f(w_0)=\beta_0$, $f(w_1)=\beta_1$ and $f(u)=p_1$. Suppose that $l_1=(g^{r-1}\cdot e,\cdots,g\cdot e,e)$, where $0\leq r\leq d(e)$. Since $f(u)=p_1=L(l_1)$, the walk $u$ of $E$ is of the form
\begin{multline}\label{walk-u}
(e_{r+1,m_{r+1}}|\tau|e_{r+1,m_{r+1}-1})\cdots(e_{r+1,2}|\tau|e_{r+1,1})(g\cdot h_r|g|h_r)(e_{r,m_r}|\tau|e_{r,m_r-1})\cdots(e_{r2}|\tau|e_{r1})\cdots \\
(g\cdot h_2|g|h_2)(e_{2,m_2}|\tau|e_{2,m_2-1})\cdots(e_{22}|\tau|e_{21})(g\cdot h_1|g|h_1)(e_{1,m_1}|\tau|e_{1,m_1-1})\cdots(e_{12}|\tau|e_{11}),
\end{multline}
such that $P(e_{11})=P(e)$, $P(e_{r+1,m_{r+1}})=P(g^{r}\cdot e)$, and $L(h_i)=L(g^{i-1}\cdot e)$ for each $1\leq i\leq r$. Then up to homotopy, we may assume that $u$ is a walk of the form (\ref{walk-u}), where $P(e_{11})=P(e)$, $P(e_{r+1,m_{r+1}})=P(g^{r}\cdot e)$, and $h_i=g^{i-1}\cdot e$ for each $1\leq i\leq r$. Therefore $u$ is homotopic to $v$, where
$$v=(e_{r+1,m_{r+1}}|\tau|g^{r}\cdot e)(g^{r}\cdot e|g|g^{r-1}\cdot e)\cdots (g^{2}\cdot e|g|g\cdot e)(g\cdot e|g|e)(e|\tau|e_{11}).$$

Suppose that $l_2=(g^{s-1}\cdot h,\cdots,g\cdot h,h)$, where $0\leq s\leq d(h)$. Since $\prescript{\wedge}{}{l_1}\equiv \prescript{\wedge}{}{l_2}$, the walk
$$ (g^{s}\cdot h|\tau|g^{r}\cdot e)(g^{r}\cdot e|g|g^{r-1}\cdot e)\cdots (g^{2}\cdot e|g|g\cdot e)(g\cdot e|g|e) $$
is homotopic to the walk
$$ (g^{s}\cdot h|g|g^{s-1}\cdot h)\cdots (g^{2}\cdot h|g|g\cdot h)(g\cdot h|g|h)(h|\tau|e). $$
Therefore $v$ is homotopic to $v'$, where
$$v'=(e_{r+1,m_{r+1}}|\tau|g^{s}\cdot h)(g^{s}\cdot h|g|g^{s-1}\cdot h)\cdots (g^{2}\cdot h|g|g\cdot h)(g\cdot h|g|h)(h|\tau|e_{11}).$$
Let $w'=w_1 v' w_0$. Since $u\sim v'$, $w=w_1 u w_0\sim w_1 v' w_0=w'$. We have $f(w')=f(w_1)f(v')f(w_0)=\beta_1 p_2 \beta_0=\gamma$.
\end{proof}

\begin{proof}[{\bf Proof of Theorem \ref{fundamental-group-of-fs-BC-and-fundamental-group-of-algebra-and-isomorphic}}]
First we need to show that the map $\overline{w}\mapsto\overline{f(w)}$ is well-defined. For two walks $u,v$ of $E$ with $u\sim v$, there exists a sequence $u=w_0$, $w_1$, $\cdots$, $w_{n}=v$ of walks of $E$, such that for each $1\leq i\leq n$, $w_i$ is obtained from $w_{i-1}$ by replacing a subwalk to another, using one of the relations $(h1)$, $(h2)$, $(h3)$ or $(h4)$ in Definition \ref{homotopy of walks}. To show that $f(u)\sim_{I_E} f(v)$, we may assume that $n=1$, so $v$ is obtained from $u$ by replacing a subwalk to another, using one of the relations $(h1)$, $(h2)$, $(h3)$ or $(h4)$ in Definition \ref{homotopy of walks}.

If $v$ is obtained from $u$ by replacing a subwalk to another using one of the relations $(h1)$, $(h2)$ or $(h4)$ in Definition \ref{homotopy of walks}, then it is straightforward to show that $f(u)\sim_{I_E} f(v)$. If $v$ is obtained from $u$ by replacing a subwalk to another using relation $(h3)$ in Definition \ref{homotopy of walks}, then $f(v)$ is obtained from $f(u)$ by replacing a subwalk of the form $L(p)$ to a subwalk of the form $L(q)$, where $p,q$ are full sequences of $E$ with $L(p)RL(q)$. According to \cite[Lemma 4.16]{LL}, $L(p)-L(q)$ is a minimal relation of $I_E$, then we have $f(u)\sim_{I_E} f(v)$.

For any walk $\gamma=\alpha_r\cdots\alpha_2\alpha_1$ of $Q_E$, where $\alpha_i\in (Q_E)_1\cup (Q_E)_{1}^{-1}$, we define a walk $w_i$ of $E$ for each $1\leq i\leq r$:
\begin{equation*}
w_i=\begin{cases}
(g\cdot e_i|g|e_i), &\text{ if } \alpha_i=L(e_i) \text{ for some } e_i\in E; \\
(e_i|g^{-1}|g\cdot e_i), &\text{ if } \alpha_i=L(e_i)^{-1} \text{ for some } e_i\in E.
\end{cases}
\end{equation*}
Since the source of $\alpha_{i+1}$ is equal to the terminal of $\alpha_i$ for every $1\leq i\leq r-1$, the source of $w_{i+1}$ and the terminal of $w_i$ belong to the same polygon of $E$ for every $1\leq i\leq r-1$. Denote by $x_i$ (resp. $y_i$) the source (resp. terminal) of $w_i$, then
$$w=w_r(x_r|\tau|y_{r-1})w_{r-1}(x_{r-1}|\tau|y_{r-2})\cdots (x_3|\tau|y_2)w_2(x_2|\tau|y_1)w_1$$
is a walk of $E$ such that $f(w)=\gamma$. Therefore the map $\Pi(E,e)\rightarrow \Pi(Q_E,I_E)$, $\overline{w}\mapsto\overline{f(w)}$ is surjective. To show the injectivity, let $\overline{w}\in\Pi(E,e)$ with $\overline{f(w)}=1$. Then $f(w)\sim_{I_E} 1_{P(e)}$, and by Lemma \ref{injectivity}, there exists some walk $w'$ of $E$ such that $w\sim w'$ and $f(w')=1_{P(e)}$. So $w'$ is of the form
$$(e_k|\tau|e_{k-1})\cdots (e_3|\tau|e_2)(e_2|\tau|e_1),$$
where $e_1,e_2,\cdots,e_k\in P(e)$. Since $w\sim w'$, the source (resp. terminal) of $w$ is equal to the source (resp. terminal) of $w'$, and we have $e_1=e_k=e$. Therefore $w\sim w'\sim (e||e)$.
\end{proof}

Let $E$ be an $f_s$-BC. Recall that we have defined a set of paths $\mathscr{E}$ of $Q_E$ and an equivalence relation $R$ on $\mathscr{E}$ in Definition \ref{relation R}. Moreover, we have constructed a quiver with admissible relations $(Q'_E,I'_E)$ of $\Lambda_E$ in Definition \ref{reduced-arrow}, where $Q'_E$ is the subquiver of $Q_E$ with $(Q'_E)_0=(Q_E)_0$ and $(Q'_E)_1$ is a complete set of representatives of non-reduced arrows of $Q_E$ under the equivalence relation $R$.

\begin{Lem}\label{replace-reduced-arrow}
Let $E$ be an $f_s$-BC. Then for every reduced arrow $\alpha$ of $Q_E$ there exists a path $v\in\mathscr{E}'$ such that $\alpha R v$, where $\mathscr{E}'$ denotes the set of paths of $Q'_E$ which belong to $\mathscr{E}$.
\end{Lem}

\begin{proof}
Let $\alpha$ be a reduced arrow of $Q_E$. Since $\Lambda_E$ is locally bounded, there exists a positive integer $N$ such that each path $\delta$ of $Q_E$ with $s(\delta)=s(\alpha)$ of length $>N$ is equal to zero in $\Lambda_E$. Since every path in $\mathscr{E}$ is nonzero in $\Lambda_E$, there exists a longest path $u\in\mathscr{E}$ such that $\alpha R u$. Suppose that $u$ contains a reduced arrow $\beta$, then $\beta R u'$ for some path $u'\in\mathscr{E}$ of length $\geq 2$. Let $w$ be the path obtained from $u$ by replacing the arrow $\beta$ with the path $u'$. Since $\beta R u'$, $\beta=u'$ in $\Lambda_E$, and therefore $w=u$ in $\Lambda_E$. According to \cite[Lemma 4.16]{LL}, $w\in\mathscr{E}$ and $w R u$. Then $\alpha R w$, which contradicts the fact that $u$ is the longest path in $\mathscr{E}$ such that $\alpha R u$.

Therefore every arrow of $u$ is non-reduced. By replacing each arrow $\beta$ of $u$ with the arrow $\beta'$ of $Q'_E$ with $\beta R \beta'$, we obtain a path $v$ of $Q'_E$ such that $v=u$ in $\Lambda_E$. According to \cite[Lemma 4.16]{LL}, $v\in\mathscr{E}$ and $v R u$. Then we have $\alpha R v$ and $v\in\mathscr{E}'$.
\end{proof}

\begin{Prop}\label{iso-of-fundamental-gp-between-two-quiver-with-relation}
Let $E$ be a connected $f_s$-BC. Then $\Pi(Q_E,I_E)\cong\Pi(Q'_E,I'_E)$.
\end{Prop}

\begin{proof}
For each arrow $\alpha\in (Q_E)_0-(Q'_E)_0$, we define a path $v_{\alpha}$ of $Q'_E$: if $\alpha$ is reduced, let $v_{\alpha}$ be a path in $\mathscr{E}'$ such that $\alpha R v_{\alpha}$ (such $v_{\alpha}$ exists according to Lemma \ref{replace-reduced-arrow}); if $\alpha$ is non-reduced, let $v_{\alpha}$ be the arrow of $Q'_E$ such that $\alpha R v_{\alpha}$.

Fix a vertex $x$ of $Q_E$ and set $\Pi(Q_E,I_E)=\Pi(Q_E,x)/N(Q_E,m(I_E),x)$, \\ $\Pi(Q'_E,I'_E)=\Pi(Q'_E,x)/N(Q'_E,m(I'_E),x)$. For every walk $w$ of $Q_E$, denote by $\mu(w)$ the walk of $Q'_E$ obtained from $w$ by replacing each arrow $\alpha$ (resp. inverse arrow $\alpha^{-1}$) with $\alpha\in (Q_E)_0-(Q'_E)_0$ with the path $v_{\alpha}$ (resp. the inverse path $v_{\alpha}^{-1}$). We need to show that the map $\overline{\mu}:\Pi(Q_E,I_E)\rightarrow\Pi(Q'_E,I'_E)$, $\overline{w}\mapsto\overline{\mu(w)}$ is an isomorphism.

$\overline{\mu}$ is well-defined: Let $w,w'$ be two closed walks of $Q_E$ at $x$ such that $w'$ is obtained from $w$ by replacing a subwalk to another using the relation $(E1)$ or the relation $(E2)$ in Definition \ref{homotopy-relation-with-respect-to-ideal}. We need to show that $\mu(w)\sim_{I'_E}\mu(w')$. If $w'$ is obtained from $w$ by replacing a subwalk to another using the relation $(E1)$, then clearly $\mu(w)\sim_{I'_E}\mu(w')$. If $w'$ is obtained from $w$ by replacing a subwalk to another using the relation $(E2)$, then we may assume that $w=w_1 u w_0$ and $w'=w_1 u' w_0$, such that there exists some $\rho=\sum_{i=1}^{n}\lambda_i p_i\in m(I_E)$ with $p_1=u$, $p_2=u'$. According to \cite[Lemma 4.16]{LL}, $u,u'\in\mathscr{E}$ and $u R u'$. By the definition of $\mu$, we have $\mu(v)=v$ in $\Lambda_E$ for every path $v$ of $Q_E$. Then $\mu(u)=u=u'=\mu(u')$ in $\Lambda_E$. Since $u\in\mathscr{E}$, we have $\mu(u)=\mu(u')\neq 0$ in $\Lambda_E$. According to \cite[Lemma 6.3]{LL}, $\mu(u),\mu(u')\in\mathscr{E}'$ and $\mu(u) R \mu(u')$. Therefore $\mu(u)-\mu(u')\in m(I'_E)$ and $\mu(w)=\mu(w_1)\mu(u)\mu(w_0)\sim_{I'_E}\mu(w_1)\mu(u')\mu(w_0)=\mu(w')$.

$\overline{\mu}$ is an isomorphism of groups: Clearly $\overline{\mu}$ is a group homomorphism. To show that $\overline{\mu}$ is bijective, we need to construct its inverse. According to \cite[Lemma 4.16]{LL} and \cite[Lemma 6.3]{LL}, every minimal relation of $I'_E$ is also a minimal relation of $I_E$. So for every two walks $u,v$ of $Q'_E$, $u\sim_{I'_E}v$ implies $u\sim_{I_E}v$. Then there is a well-defined map $\overline{\nu}:\Pi(Q'_E,I'_E)\rightarrow\Pi(Q_E,I_E)$, $\overline{u}\mapsto\overline{u}$. Clearly $\overline{\mu}\circ\overline{\nu}=id$. Note that for every two paths $u,v\in\mathscr{E}$ with $uRv$, we have $u\sim_{I_E}v$ and $u^{-1}\sim_{I_E}v^{-1}$. Then for every walk $w$ of $Q_E$, we have $w\sim_{I_E}\mu(w)$. Therefore $\overline{\nu}\circ\overline{\mu}(\overline{w})=\overline{\mu(w)}=\overline{w}$ and $\overline{\nu}\circ\overline{\mu}=id$.
\end{proof}

Combining Theorem \ref{fundamental-group-of-fs-BC-and-fundamental-group-of-algebra-and-isomorphic} and Proposition \ref{iso-of-fundamental-gp-between-two-quiver-with-relation}, we have

\begin{Cor}\label{iso-of-fundamental-gp-of-f-BC-and-fundamental-gp-of-quiver-with-aadmissible-relation}
Let $E$ be a connected $f_s$-BC. Then the fundamental group $\Pi(E)$ of $E$ is isomorphic to $\Pi(Q'_E,I'_E)$.
\end{Cor}

\section{A covering of fractional Brauer configurations induces a covering of fractional Brauer configuration categories}

In this section, we first recall (by a slight modifying) from \cite{MV-DLP} the definition of a covering of quivers with relations, and then show that a (regular) covering $\phi: E\rightarrow E'$ of fractional Brauer configurations of type S induces a (Galois) covering $f: (Q_E,I_E)\rightarrow (Q_{E'},I_{E'})$ of quivers with relations. These results are the contents of Theorem \ref{covering of f-BCs induces covering functor} and Proposition \ref{induce-Galois-covering-of-categories}. We also obtain some general results about covering theory, such as a criterion for a locally bounded category to be simply connected, which is a beneficial complement of classical results. As an application, we show that every fractional Brauer configuration category of type MS is standard in the sense of Skowro\'nski \cite{Sk}.

\subsection{A covering of f-BCs induces a covering of f-BCCs}

\begin{Def} {\rm(\cite[Section I.10.1]{E})}
Let $Q$, $Q'$ be (locally finite) quivers and $f:Q\rightarrow Q'$ be a quiver morphism. $f$ is said to be a covering if for each $x\in Q_0$, $f$ induces a bijection between the set of arrows of $Q$ starting (ending) at $x$ and the set of arrows of $Q'$ starting (ending) at $f(x)$.
\end{Def}

By this definition, if $f: Q\rightarrow Q'$ is a covering map such that $Q$, $Q'$ have no double arrows, then for each $x\in Q_0$, $f$ induces a bijection between the set $x^+$ (resp. $x^-$) of successors (resp. predecessors) of $x$ and the set $f(x)^+$ (resp. $f(x)^-$) of successors (resp. predecessors) of $f(x)$.

It is important that a covering map $f$ has the unique lifting of paths property, that is, for a path $p'$ in $Q'$ and for a vertex $x\in Q_0$ with $f(x)$ the starting (resp. the ending) of $p'$, there exists a unique path $p$ starting (resp. ending) at $x$ in $Q$ such that $f(p)=p'$. Moreover, by the unique lifting of paths property, it is easy to see that if $f: Q\rightarrow Q'$ is a covering map such that $Q'$ is connected, then $f$ is surjective both on vertices and on arrows.

Recall from \cite{BG} that a $k$-linear functor $F: M\rightarrow N$ between two $k$-categories is said to be a covering functor if for every two objects $a,b$ of $N$, $F$ induces isomorphisms $\oplus_{z/a}M(z,y)\rightarrow N(a,b)$ and $\oplus_{z/b}M(x,z)\rightarrow N(a,b)$, where $x,y$ are objects of $M$ with $a=F(x)$, $b=F(y)$. The following result is well-known.

\begin{Lem}\label{covering of quivers induces covering functor}
If $f:Q\rightarrow Q'$ is a covering of quivers, then $f$ induces a covering functor $kf:kQ\rightarrow kQ'$ between the associated path categories.
\end{Lem}

\begin{proof}
Clearly, any morphism $f:Q\rightarrow Q'$ induces a functor $kQ\rightarrow kQ'$ of path categories which we denote by $kf$. Since $f$ is a covering of quivers, for each $x\in Q_0$ and $a\in Q'_0$ with $f(x)=a$, $f$ maps paths of $Q$ which start at $x$ bijectively to paths of $Q'$ which start at $a$. Therefore for each $b\in Q'_0$, $f$ induces a bijection between $\mathscr{B}=\{$paths $p$ of $Q$ with $s(p)=x$ and $t(p)\in f^{-1}(b)\}$ and $\mathscr{B'}=\{$paths $p'$ of $Q'$ with $s(p')=a$ and $t(p')=b\}$. Since $\mathscr{B}$ is a basis of $\bigoplus_{y/b}kQ(x,y)$ and $\mathscr{B'}$ is a basis of $kQ'(a,b)$, $kf$ induces a bijection $\bigoplus_{y/b}kQ(x,y)\rightarrow kQ'(a,b)$. Similarly, for each $y\in Q_0$ and $a,b\in Q'_0$ with $f(y)=b$, $kf$ induces a bijection $\bigoplus_{x/a}kQ(x,y)\rightarrow kQ'(a,b)$.
\end{proof}

We now recall from \cite{MV-DLP} the definition of a covering of quivers with relations. Our definition here is slightly more general than the original one. Indeed, we only request that the relations are contained in the ideal generated by arrows and we do not assume that the covering is induced by an admissible group of automorphisms. Note that the original definition in \cite{MV-DLP} is precisely our definition of a Galois covering of quivers with (admissible) relations and admissible group (see Definition \ref{galois-covering}).

\begin{Def} \label{morphism-and-covering-of-quivers-with-relations}
Let $(Q,I)$ and $(Q',I')$ be quivers with relations, where $I$ (resp. $I'$) is contained in the ideal of $kQ$ (resp. $kQ'$) generated by arrows.
\begin{itemize}
\item A morphism of quivers with relations $f:(Q,I)\rightarrow (Q',I')$ is a morphism of quivers $f:Q\rightarrow Q'$ such that the $k$-linear functor $kf:kQ\rightarrow kQ'$ induced by $f$ maps the morphisms in $I$ to $I'$;
\item if, moreover, $f$ is a covering of quivers, such that for any $\rho'\in I'(a',b')$ with $\rho'$ a minimal relation or a zero relation of $I'$ and for any $a\in f^{-1}(a')$ (resp. $b\in f^{-1}(b')$), there exist some $b\in f^{-1}(b')$ (resp. $a\in f^{-1}(a')$) and some $\rho\in I(a,b)$ such that $(kf)(\rho)=\rho'$, then $f$ is called a covering of quivers with relations.
\end{itemize}
\end{Def}

\begin{Lem} \label{preserve-minimal-relation}
Let $f:(Q,I)\rightarrow (Q',I')$ be a covering of quivers with relations. Let $\rho\in I(a,b)$ be a relation of $I$ and let $\rho'=(kf)(\rho)\in I'(f(a),f(b))$. Then $\rho$ is a minimal relation of $I$ if and only if $\rho'$ is a minimal relation of $I'$. In particular, for a covering of quivers with relations, we have also the unique lifting of minimal relations property.
\end{Lem}

\begin{proof}
Assume that the relation $\rho'$ of $I'$ is minimal. Suppose that $\rho=\sum_{i=1}^{n}\lambda_i p_i$, where $\lambda_i\in k^{*}$ and $p_i$'s are distinct paths of $Q$ from $a$ to $b$, is not minimal, then there exists a non-empty proper subset $K$ of $\{1,\cdots ,n\}$ such that $\rho_1=\sum_{i\in K}\lambda_i p_i\in I$. Then $\rho'=\sum_{i=1}^{n}\lambda_i f(p_i)\in I'$, where $f(p_i)$'s are distinct paths of $Q'$, and $\sum_{i\in K}\lambda_i f(p_i)\in I'$. This contradicts the minimality of $\rho'$.

Conversely, let $\rho=\sum_{i=1}^{n}\lambda_i p_i$, where $\lambda_i\in k^{*}$ and $p_i$'s are distinct paths of $Q$ from $a$ to $b$; we call it the standard form of $\rho$. Since $f:Q\rightarrow Q'$ is a covering of quivers, $f(p_i)$'s are distinct paths of $Q'$, and $\rho'=\sum_{i=1}^{n}\lambda_i f(p_i)$ is the standard form of $\rho'$. Suppose that $\rho'$ is not minimal, then there exists a non-empty proper subset $K$ of $\{1,\cdots ,n\}$ such that $\rho''=\sum_{i\in K}\lambda_i f(p_i)\in I'$. We may assume that $\rho''$ is a minimal relation or a zero relation of $I'$, so there exists some $c\in Q_0$ with $f(c)=f(b)$ and some $r\in I(a,c)$ such that $(kf)(r)=\rho''$. If $r=\sum_{j=1}^{m}\mu_j q_j$ is the standard form of $r$, then $\sum_{j=1}^{m}\mu_j f(q_j)$ is the standard form of $\rho''$. So there exists a bijection $\sigma:\{1,\cdots,m\}\rightarrow K$ such that $\mu_j=\lambda_{\sigma(j)}$ and $f(q_j)=f(p_{\sigma(j)})$ for each $j\in\{1,\cdots,m\}$. Since $f:Q\rightarrow Q'$ is a covering of quivers and since $s(q_j)=s(p_{\sigma(j)})$ for each $j\in\{1,\cdots,m\}$, we have $q_j=p_{\sigma(j)}$ for each $j\in\{1,\cdots,m\}$. Then $r=\sum_{i\in K}\lambda_i p_i\in I$, which contradicts the minimality of $\rho$.
\end{proof}

The following result gives a criterion for whether the functor induced by a covering functor on the quotient categories is again a covering functor. Note that there is a similar but different result which was stated without proof in \cite[Section 3.1]{BG}.

\begin{Prop}\label{covering of quotient categories}
Let $M$, $N$ be $k$-categories and $F:M\rightarrow N$ be a covering functor. Let $R$ (resp. $R'$) be a class of morphisms of $M$ (resp. $N$), $I_M$ (resp. $I_N$) be the ideal of $M$ (resp. $N$) generated by $R$ (resp. $R'$). If the following conditions hold:
\begin{enumerate}
\item for every $f\in R$, $F(f)\in R'$;
\item for every $f'\in N(a,b)\cap R'$ and for each $x\in F^{-1}(a)$, there exist $y\in F^{-1}(b)$ and $f\in M(x,y)\cap R$ such that $F(f)=f'$;
\item for every $f'\in N(a,b)\cap R'$ and for each $y\in F^{-1}(b)$, there exist $x\in F^{-1}(a)$ and $f\in M(x,y)\cap R$ such that $F(f)=f'$,
\end{enumerate}
then $F$ induces a covering functor $\widetilde{F}:M/I_M\rightarrow N/I_N$.
\end{Prop}

\begin{proof}
By condition (1), $F$ maps morphisms in $I_M$ to morphisms in $I_N$, therefore it induces a functor $\widetilde{F}:M/I_M\rightarrow N/I_N$. To show $\widetilde{F}$ is a covering functor, it suffices to show that $F$ induces bijections $\bigoplus_{t/b}I_M(x,t)\rightarrow I_N(Fx,b)$ and $\bigoplus_{z/a}I_M(z,y)\rightarrow I_N(a,Fy)$ for every objects $a$, $b$ of $N$ and for every objects $x$, $y$ of $M$. Since $F$ is a covering functor, the above maps are injective.

For every objects $a$, $b$ of $N$ and for every object $x$ of $M$ such that $Fx=a$, let $g'\in I_N(a,b)$ be a morphism of the form $u'f'v'$, where $u'\in N(d,b)$, $v'\in N(a,c)$ and $f'\in N(c,d)\cap R'$. Since $F$ induces a bijection $\bigoplus_{z/c}M(x,z)\rightarrow N(a,c)$, there exists $(v_z)_{z/c}\in\bigoplus_{z/c}M(x,z)$ such that $\sum_{z/c}F(v_z)=v'$. By condition (2), for each $z\in F^{-1}(c)$, there exist $t_z\in F^{-1}(d)$ and $f_z\in M(z,t_z)\cap R$ such that $F(f_z)=f'$. For each $z\in F^{-1}(c)$, since $F$ induces a bijection $\bigoplus_{y/b}M(t_z,y)\rightarrow N(d,b)$, there exists $(u_{zy})_{y/b}\in\bigoplus_{y/b}M(t_z,y)$ such that $\sum_{y/b}F(u_{zy})=u'$. For each $y\in F^{-1}(b)$, let $g_y=\sum_{z/c}u_{zy}f_{z}v_{z}$. Then $g_y\in I_M(x,y)$ and \begin{multline*}\sum_{y/b}F(g_y)=\sum_{y/b}\sum_{z/c}F(u_{zy})F(f_z)F(v_z)=\sum_{z/c}(\sum_{y/b}F(u_{zy}))F(f_z)F(v_z)\\ =\sum_{z/c}u'F(f_z)F(v_z)=\sum_{z/c}u'f'F(v_z)=u'f'v'=g'.\end{multline*}
Since each morphism in $I_N(a,b)$ is a sum of morphisms of the form $u'f'v'$, where $f'\in R'$, we imply that the map $\bigoplus_{y/b}I_M(x,y)\rightarrow I_N(a,b)$ is surjective. For every objects $a$, $b$ of $N$ and for every object $y$ of $M$ such that $Fy=b$, it can be shown similarly that the map $\bigoplus_{x/a}I_M(x,y)\rightarrow I_N(a,b)$ is also surjective.
\end{proof}

\begin{Cor}\label{Cor}
If $f:(Q,I)\rightarrow (Q',I')$ is a covering of quivers with relations, then $f$ induces a covering functor $kQ/I\rightarrow kQ'/I'$.
\end{Cor}

\begin{proof}
By Lemma \ref{covering of quivers induces covering functor}, $f$ induces a covering functor $kf:kQ\rightarrow kQ'$ of associated path categories. Let $R$ (resp. $R'$) be the set of minimal relations together with zero relations in $I$ (resp. $I'$). Then $R$ (resp. $R'$) generated $I$ (resp. $I'$). According to Lemma \ref{preserve-minimal-relation}, for every $\rho\in R$, we have $(kf)(\rho)\in R'$, so condition $(1)$ in Proposition \ref{covering of quotient categories} holds. Since $f:(Q,I)\rightarrow (Q',I')$ is a covering of quivers with relations, condition $(2)$ and condition $(3)$ in Proposition \ref{covering of quotient categories} also hold. Therefore $f$ induces a covering functor $kQ/I\rightarrow kQ'/I'$.
\end{proof}

For an f-BC $E=(E,P,L,d)$, recall from Definition \ref{relation R} that we have defined a set of paths $\mathscr{E}$ of $Q_E$ and a relation $R$ on $\mathscr{E}$. When $E$ is an $f_s$-BC, it follows from \cite[Lemma 4.14]{LL} that $R$ is an equivalence relation on $\mathscr{E}$, and according to \cite[Lemma 4.16]{LL}, each zero relation of $I_E$ is a relation of type $(fR2)$ or $(fR3)$ in Definition \ref{f-BC algebra}, and each minimal relation of $I_E$ is a relation of the form $\sum_{i=1}^{n}\lambda_i u_i$, where $\lambda_i\in k^{*}$ with $\sum_{i=1}^{n}\lambda_i=0$ and $\sum_{i\in K}\lambda_i\neq 0$ for any proper subset $K$ of $\{1,\cdots,n\}$, and $u_i$'s are pairwise distinct paths of $Q_E$ in $\mathscr{E}$ with $u_i R u_j$ for all $1\leq i,j\leq n$.

\begin{Lem}\label{L(e)=L(h)}
Let $E=(E,P,L,d)$, $E'=(E',P',L',d')$ be f-BCs, $\phi:E\rightarrow E'$ be a covering. For $e$, $h\in E$, if $P(g\cdot e)=P(g\cdot h)$ and $L'(\phi(e))=L'(\phi(h))$, then $L(e)=L(h)$.
\end{Lem}

\begin{proof}
For each $e'\in L(e)$ we have $P(g\cdot e')=P(g\cdot e)$. Therefore $g\cdot L(e)\subseteq P(g\cdot e)$ (resp. $g\cdot L(h)\subseteq P(g\cdot h)$). Since $\phi$ is a covering, it maps $L(e)$ bijectively onto $L'(\phi(e))$. Then $\phi(g\cdot L(e))=g\cdot\phi(L(e))=g\cdot L'(\phi(e))$ (resp. $\phi(g\cdot L(h))=g\cdot\phi(L(h))=g\cdot L'(\phi(h))$). Since $L'(\phi(e))=L'(\phi(h))$, we have $\phi(g\cdot L(e))=g\cdot L'(\phi(e))=g\cdot L'(\phi(h))=\phi(g\cdot L(h))$. Since $g\cdot L(e)$, $g\cdot L(h)$ are subsets of $P(g\cdot e)=P(g\cdot h)$, and since $\phi$ maps $P(g\cdot e)$ bijectively onto $P'(\phi(g\cdot e))$, we have $g\cdot L(e)=g\cdot L(h)$. Therefore $L(e)=L(h)$.
\end{proof}

\begin{Thm}\label{covering of f-BCs induces covering functor}
Let $E=(E,P,L,d)$, $E'=(E',P',L',d')$ be connected $f_s$-BCs, $\phi:E\rightarrow E'$ be a covering. Then $\phi$ induces a covering $f:(Q_E,I_E)\rightarrow (Q_{E'},I_{E'})$ of quivers with relations. Especially, $\phi$ induces a covering functor $F:\Lambda_{E}\rightarrow\Lambda_{E'}$.
\end{Thm}

\begin{proof}
Let $f:Q_E \rightarrow Q_{E'}$ be the quiver morphism which maps each vertex $P(e)$ of $Q_E$ to the vertex $P'(\phi(e))$ of $Q_{E'}$ and maps each arrow $L(e)$ of $Q_E$ to the arrow $L'(\phi(e))$ of $Q_{E'}$. We need to show that $f$ is a covering of quivers.

For each vertex $P(e)$ of $Q_E$, the set of arrows of $Q_E$ starting at $P(e)$ is $\{L(h)\mid h\in P(e)\}$, and the set of arrows of $Q_{E'}$ starting at $P'(\phi(e))$ is $\{L'(h')\mid h'\in P'(\phi(e))\}$. Since $\phi$ is a covering of f-BCs, it induces bijections $P(h)\rightarrow P'(\phi(h))$ and $L(h)\rightarrow L'(\phi(h))$ for all $h\in E$. Therefore $f$ maps the set of arrows starting at $P(e)$ bijectively onto the set of arrows starting at $f(P(e))=P'(\phi(e))$. Let $L(h_1)$, $L(h_2)$ be two arrows of $Q_E$ ending at $P(e)$ such that $L'(\phi(h_1))=L'(\phi(h_2))$. Since $g\cdot h_1$, $g\cdot h_2\in P(e)$, by Lemma \ref{L(e)=L(h)}, $L(h_1)=L(h_2)$. For each arrow $L'(h')$ of $Q_{E'}$ ending at $P'(\phi(e))$, we have $g\cdot h'\in P'(\phi(e))$. Since $\phi$ is a covering of f-BCs, there exists $\widetilde{h}\in P(e)$ such that $\phi(\widetilde{h})=g\cdot h'$. Let $h=g^{-1}\cdot\widetilde{h}$. Then $L(h)$ is an arrow of $Q_E$ ending at $P(e)$ and $f(L(h))=L'(\phi(h))=L'(h')$. Therefore $f$ maps the set of arrows of $Q_E$ ending at $P(e)$ bijectively onto the set of arrows of $Q_{E'}$ ending at $f(P(e))=P'(\phi(e))$.

It is straightforward to show that the functor $kf:kQ_E\rightarrow kQ_{E'}$ maps each morphism in $I_E$ to a morphism in $I_{E'}$. Let $C$ (resp. $C'$) be the set of minimal relations together with zero relations of $I_E$ (resp. $I_{E'})$. To show that $f:(Q_E,I_E)\rightarrow (Q_{E'},I_{E'})$ is a covering of quivers with relations, it suffices to show that the following conditions hold: for each $\rho'\in C'$ starting (resp. ending) at $a\in (Q_{E'})_0$ and for each $x\in (Q_{E})_0$ with $f(x)=a$, there exists $\rho\in C$ starting (resp. ending) at $x$ such that $(kf)(\rho)=\rho'$.

Let $\rho'=L'(e'_n)\cdots L'(e'_2)L'(e'_1)\in C'$ be a zero relation of $I_{E'}$ starting at $P'(e')\in (Q_{E'})_0$, where $P'(g\cdot e'_i)=P'(e'_{i+1})$ for each $1\leq i\leq n-1$. For any $P(e)\in (Q_{E})_0$ such that $(kf)(P(e))=P'(e')$, since $f$ is a covering of quivers, there exists a unique path $\rho=L(e_n)\cdots L(e_2)L(e_1)$ of $Q_E$ starting at $P(e)$, such that
$$kf(L(e_n)\cdots L(e_2)L(e_1))=L'(e'_n)\cdots L(e'_2)L(e'_1).$$
If $\rho'$ is a relation of $I_{E'}$ of type $(fR3)$, it is straightforward to show that $\rho$ is a relation of $I_{E}$ of type $(fR3)$. If $\rho'$ is a relation of $I_{E'}$ of type $(fR2)$, then $\bigcap_{i=1}^n g^{n-i}\cdot L'(e'_i)=\emptyset$. Suppose that $\rho$ is not a relation of $I_E$ of type $(fR2)$, then there exists $h\in E$ such that
$$L(e_n)\cdots L(e_2)L(e_1)=L(g^{n-1}\cdot h)\cdots L(g\cdot h)L(h).$$
We have
$$L'(g^{i-1}\cdot\phi(h))=L'(\phi(g^{i-1}\cdot h))=f(L(g^{i-1}\cdot h))=f(L(e_i))=L'(e'_i)$$
for each $1\leq i\leq n$. Then $g^{n-1}\cdot\phi(h)\in\bigcap_{i=1}^n g^{n-i}\cdot L'(e'_i)$, which contradicts $\bigcap_{i=1}^n g^{n-i}\cdot L'(e'_i)=\emptyset$. Therefore $\rho$ is a relation of $I_E$ of type $(fR2)$. If $\rho'=L'(e'_n)\cdots L(e'_2)L(e'_1)\in C'$ is a zero relation of $I_{E'}$ ending at $P'(e')\in (Q_{E'})_0$, then for each $P(e)\in (Q_{E})_0$ such that $(kf)(P(e))=P'(e')$, it can be shown similarly that there exists a zero relation $\rho=L(e_n)\cdots L(e_2)L(e_1)\in C$ ending at $P(e)$ such that
$(kf)(\rho)=\rho'$.

Denote
$$\mathscr{E}=\{L(p)\mid p\mbox{ is a standard sequence of }E\}$$
and
$$\mathscr{E}'=\{L'(p')\mid p'\mbox{ is a standard sequence of }E'\},$$
and let $R$ (resp. $R'$) be the equivalence relation on $\mathscr{E}$ (resp. $\mathscr{E}'$) given in Definition \ref{relation R}. Let $\rho'=\sum_{i=1}^{n}\lambda_i u'_i\in C'$ be a minimal relation of $I_{E'}$ starting at $P'(x')\in (Q_{E'})_0$, where $\lambda_i\in k^{*}$ with $\sum_{i=1}^{n}\lambda_i=0$ and $\sum_{i\in K}\lambda_i\neq 0$ for any proper subset $K$ of $\{1,\cdots,n\}$ and $u'_i$'s are pairwise distinct paths of $Q_{E'}$ in $\mathscr{E}'$ with $u'_i R' u'_j$ for all $1\leq i,j\leq n$. Let $P(x)$ be a vertex of $Q_{E}$ such that $f(P(x))=P'(x')$, and let $u_i$ be the unique path of $Q_E$ starting at $P(x)$ such that $f(u_i)=u'_i$. We need to show that $\rho=\sum_{i=1}^{n}\lambda_i u_i$ is a minimal relation of $I_{E}$ starting at $P(x)$, that is, to show that each $u_i$ is a path of $Q_E$ in $\mathscr{E}$ and $u_p R u_q$ for all $1\leq p,q\leq n$.

For any $1\leq p,q\leq n$, since $u'_p R' u'_q$, $u'_p-u'_q$ is a relation of $I_{E'}$ of type $(fR1)$. Assume that
$$u'_p=L'(g^{d'(e')-1-k}\cdot e')\cdots L'(g\cdot e')L'(e')$$
and
$$u'_q=L'(g^{d'(h')-1-k}\cdot h')\cdots L'(g\cdot h')L'(h'),$$
where $P'(e')=P'(h')=P'(x')$, $0\leq k<\mathrm{min}\{d'(e'),d'(h')\}$ and $L'(g^{d'(e')-i}\cdot e')=L'(g^{d'(h')-i}\cdot h')$ for each $1\leq i\leq k$. Since $\phi$ induces a bijection $P(x)\rightarrow P'(x')$, there exist $e,h\in P(x)$ such that $\phi(e)=e'$ and $\phi(h)=h'$. We have
$$u_p=L(g^{d(e)-1-k}\cdot e)\cdots L(g\cdot e)L(e)$$
and
$$u_q=L(g^{d(h)-1-k}\cdot h)\cdots L(g\cdot h)L(h).$$
So $u_p,u_q$ are paths of $Q_E$ in $\mathscr{E}$. Since $P(e)=P(h)$, $P(g^{d(e)}\cdot e)=P(g^{d(h)}\cdot h)$. Since \begin{multline*}L'(\phi(g^{d(e)-1}\cdot e))=L'(g^{d'(\phi(e))-1}\cdot\phi(e))=L'(g^{d'(e')-1}\cdot e') \\ =L'(g^{d'(h')-1}\cdot h')=L'(g^{d'(\phi(h))-1}\cdot\phi(h))=L'(\phi(g^{d(h)-1}\cdot h)),\end{multline*}
by Lemma \ref{L(e)=L(h)}, $L(g^{d(e)-1}\cdot e)=L(g^{d(h)-1}\cdot h)$. By induction, $L(g^{d(e)-i}\cdot e)=L(g^{d(h)-i}\cdot h)$ for all $1\leq i\leq k$. Therefore $u_p-u_q$ is a relation of $I_E$ of type $(fR1)$, and $u_p R u_q$.

Let $\rho'\in C'$ be a minimal relation of $I_{E'}$ ending at $P'(x')\in (Q_{E'})_0$ and let $P(x)$ be a vertex of $Q_{E}$ with $f(P(x))=P'(x')$. It can be shown similarly that there exists a minimal relation $\rho\in C$ of $I_{E}$ ending at $P(x)$ such that $(kf)(\rho)=\rho'$.
\end{proof}

Recall from Definition \ref{admissible} that a group $\Pi$ of automorphisms of an f-BC $E$ is said to act admissibly on $E$ if each $\Pi$-orbit of $E$ meets each polygon of $E$ in at most one angle.

\begin{Lem}\label{act-freely-on-quiver}
Let $E=(E,P,L,d)$ be a connected f-BC and $\Pi$ be an admissible group of automorphisms of $E$. For any $\phi\in\Pi$, if the automorphism $f$ of $Q_E$ induced by $\phi$ satisfies $f(a)=a$ for some $a\in (Q_E)_0$, then $\phi=id_{E}$.
\end{Lem}

\begin{proof}
Suppose that $a=P(e)$ for some $e\in E$. Since $f(a)=a$, $\phi(e)\in P(e)$. Since $\Pi$ is admissible, we have $\phi(e)=e$. For every $h\in E$, since $E$ is connected, there exists a walk $w$ of $E$ from $e$ to $h$, we will show that $\phi(h)=h$ by induction on the length $l(w)$ of $w$. If $l(w)=0$, then $w$ is trivial and $\phi(h)=h$. If $l(w)>0$, let $w=(h|\delta|h')w'$, where $w'$ is a walk of $E$ from $e$ to $h'$ and $\delta\in\{g,g^{-1},\tau\}$. By induction we have $\phi(h')=h'$. If $\delta=g$ or $g^{-1}$, $\phi(h)=\phi(\delta(h'))=\delta(\phi(h'))=\delta(h')=h$. If $\delta=\tau$, then $h\in P(h')$. Since $\phi$ is an automorphism of $E$, $\phi(h)\in P(\phi(h'))=P(h')$. Since $\Pi$ is admissible, $\phi(h)=h$. Therefore $\phi=id_{E}$.
\end{proof}

Recall that a group $\Pi$ of $(Q,I)$-automorphisms is said to act freely on $Q$ if for every $f\in\Pi$ and for every $a\in Q_0$, $f(a)=a$ implies $f=id_{Q}$.

\begin{Lem}\label{admissible-free}
Let $E$ be a connected f-BC and let $\Pi$ be an admissible group of automorphisms of $E$. Then $\Pi$ induces a group $\Pi'$ of $(Q_E,I_E)$-automorphisms which acts freely on $Q_E$.
\end{Lem}

\begin{proof}
It is obvious that every automorphism of $E$ induces a $(Q_E,I_E)$-automorphism. For $\phi\in\Pi$, denote by $f\in\Pi'$ the $(Q_E,I_E)$-automorphism induced by $\phi$. If $f(a)=a$ for some $a\in (Q_E)_0$, by Lemma \ref{act-freely-on-quiver}, $\phi=id_{E}$. Then $f=id_{Q_E}$.
\end{proof}

Let $\Pi$ be a group of $(Q,I)$-automorphisms which acts freely on $Q$. Then we can define the {\it orbit quiver} $\overline{Q}:=Q/\Pi$. For any vertex $x$ of $Q$, we denote by $\overline{x}$ the $\Pi$-orbit of $x$. Moreover, let $\overline{I}$ be the image of $I$ under the natural projection $kQ\rightarrow k\overline{Q}$. That is, for every two vertices $\overline{x},\overline{y}$ of $\overline{Q}$, $\overline{I}(\overline{x},\overline{y})$ is the $k$-vector space generated by morphisms $\overline{\rho}$, where $\rho\in I(a,b)$ with $a\in\overline{x}$, $b\in\overline{y}$ and $\overline{\rho}$ denotes the image of $\rho$ under the natural projection $kQ\rightarrow k\overline{Q}$.

\begin{Def} {\rm(\cite[Section 1]{MV-DLP})} \label{galois-covering}
A morphism $f:(Q,I)\rightarrow (Q',I')$ of quivers with relations is said to be a Galois covering if there exist a group $\Pi$ of $(Q,I)$-automorphisms which acts freely on $Q$ and an isomorphism $\nu:(Q/\Pi,\overline{I})\xrightarrow{\sim}(Q',I')$ of quivers with relations such that the diagram
$$\xymatrix{
		& (Q,I) \ar[dr]^{f}\ar[dl]_{\pi} &  \\
		(Q/\Pi,\overline{I})\ar[rr]_{\nu}^{\sim} & & (Q',I')
	}$$
commutes, where $\pi:(Q,I)\rightarrow (Q/\Pi,\overline{I})$ is the natural projection.
\end{Def}

According to \cite[Proposition 1.4]{MV-DLP}, a Galois covering of quivers with relations is a covering of quivers with relations in the sense of Definition \ref{morphism-and-covering-of-quivers-with-relations}.

\begin{Lem}\label{covering-of-quivers-with-relations}
Let $E,E'$ be connected f-BCs and let $\phi:E\rightarrow E'$ be a regular covering. Then $\phi$ induces a Galois covering $(Q_E,I_E)\rightarrow (Q_{E'},I_{E'})$ of quivers with relations with group $\mathrm{Aut}(\phi)$.
\end{Lem}

\begin{proof}
According to Theorem \ref{regular-covering-explicit}, we may assume that $\phi$ is the natural projection $E\rightarrow E/\Pi$, where $\Pi$ is an admissible group of automorphisms of $E$. According to Lemma \ref{admissible-free}, $\Pi$ induces a group $\Pi'$ of $(Q_E,I_E)$-automorphisms which acts freely on $Q_E$. It is straightforward to show that the natural projection $(Q_E,I_E)\rightarrow (Q_{E/\Pi},I_{E/\Pi})$ induces an isomorphism $(Q_{E}/\Pi',\overline{I_E})\xrightarrow{\sim} (Q_{E/\Pi},I_{E/\Pi})$. Then $(Q_E,I_E)\rightarrow (Q_{E'},I_{E'})$ is a Galois covering of quivers with relations with group $\Pi'$, where $\Pi'\cong\Pi\cong\mathrm{Aut}(\phi)$.
\end{proof}

\begin{Def} {\rm(\cite[Section 3.1]{G})} \label{galois-covering-functor}
Let $M,N$ be locally bounded categories. A covering functor $E:M\rightarrow N$ is called a Galois covering if $E$ is surjective on objects and there exists a group $G$ of $k$-linear automorphisms of $M$ which acts freely on $M$, such that $E\circ g=E$ for any $g\in G$ and $G$ acts transitively on $E^{-1}(a)$ for each object $a$ of $N$.
\end{Def}

The following result is well-known.

\begin{Lem}\label{induce-Galois-covering}
Let $(Q,I)$, $(Q',I')$ be quivers with relations such that $kQ/I$, $kQ'/I'$ are locally bounded categories. A Galois covering of quivers with relations $f:(Q,I)\rightarrow (Q',I')$ with group $\Pi$ induces a Galois covering $kQ/I\rightarrow kQ'/I'$ of associated categories with the isomorphic group.
\end{Lem}

\begin{proof}
We may assume that $f$ is the natural projection $(Q,I)\rightarrow (\overline{Q},\overline{I})$, where $\overline{Q}=Q/\Pi$ with $\Pi$ a group of $(Q,I)$-automorphisms which acts freely on $Q$. Since a Galois covering of quivers with relations is a covering of quivers with relations, by Corollary \ref{Cor}, $f$ induces a covering functor $E:kQ/I\rightarrow k\overline{Q}/\overline{I}$ of associated categories.

Denote by $\widetilde{\Pi}$ the group of automorphisms of $kQ/I$ induced by $\Pi$. Since $\Pi$ acts freely on $Q$, we see that $\widetilde{\Pi}$ acts freely on $kQ/I$ and is isomorphic to $\Pi$. It is obvious that $E\circ g=E$ for any $g\in \widetilde{\Pi}$ and $\widetilde{\Pi}$ acts transitively on $E^{-1}(a)$ for each object $a$ of $k\overline{Q}/\overline{I}$. Therefore $E:kQ/I\rightarrow k\overline{Q}/\overline{I}$ is a Galois covering with group $\widetilde{\Pi}\cong\Pi$.
\end{proof}

Combining Lemma \ref{covering-of-quivers-with-relations} and Lemma \ref{induce-Galois-covering}, we have

\begin{Prop}\label{induce-Galois-covering-of-categories}
If $\phi: E\rightarrow E'$ is a regular covering of connected f-BCs, then $\phi$ induces a Galois covering $\Lambda_E\rightarrow \Lambda_{E'}$ of locally bounded categories with group Aut$(\phi)$.
\end{Prop}

\subsection{A universal cover of $f_s$-BCs induces a universal cover of the Gabriel quiver with admissible relations of the associated $f_s$-BCCs}
\

In the subsection, let $\phi: E\rightarrow E'$ be a covering of f-BCs and $f: Q_E\rightarrow Q_{E'}$ be the covering of quivers induced by $\phi$. Denote by $\mathscr{E}$ (resp. $\mathscr{E}'$) the set of paths in $Q_E$ (resp. $Q_{E'}$) which are given by standard sequences, and denote by $R$ (resp. $R'$) the relation on $\mathscr{E}$ (resp. $\mathscr{E}'$) given in Definition \ref{relation R}.

\begin{Lem}\label{lift-the-relation-R}
If $u$ is a path of $Q_E$ in $\mathscr{E}$ and $v'$ is a path of $Q_{E'}$ in $\mathscr{E}'$ with $f(u)R'v'$, then there exists a path $v$ of $Q_E$ which belongs to $\mathscr{E}$ such that $uRv$ and $f(v)=v'$.
\end{Lem}

\begin{proof}
Write $E=(E,P,L,d)$ and $E'=(E',P',L',d')$. Let $p',q'$ be two standard sequences of $E'$ with $f(u)=L'(p')$, $v'=L'(q')$ and $\prescript{\wedge}{}{p'}\equiv\prescript{\wedge}{}{q'}$. Let $h'\in E'$ (resp. $k'\in E'$) be the source of $p'$ (resp. $q'$). Since $\prescript{\wedge}{}{p'}\equiv\prescript{\wedge}{}{q'}$, $P'(h')=P'(k')$. Assume that $s(u)=P(e)$ with $e\in E$, then $P'(\phi(e))=f(P(e))=s(f(u))=s(L'(p'))=P'(h')$. Since $\phi$ is a covering of f-BCs, there exist $h,k\in P(e)$ such that $\phi(h)=h'$ and $\phi(k)=k'$. Let $p$ (resp. $q$) be the standard sequence of $E$ with source $h$ (resp. $k$) such that $\phi(p)=p'$ (resp. $\phi(q)=q'$). Since $f(L(p))=L'(\phi(p))=L'(p')=f(u)$ and since the paths $L(p)$ and $u$ have the same source, $L(p)=u$.

Let $v=L(q)$, and we have $f(v)=f(L(q))=L'(\phi(q))=L'(q')=v'$. To show that $uRv$, we need to show that $\prescript{\wedge}{}{p}\equiv\prescript{\wedge}{}{q}$. Since $\phi(\prescript{\wedge}{}{p})=\prescript{\wedge}{}{\phi(p)}=\prescript{\wedge}{}{p'}\equiv \prescript{\wedge}{}{q'}$, by Lemma \ref{bijection}, there exists a standard sequence $l$ of $E$ such that $\prescript{\wedge}{}{p}\equiv l$ and $\phi(l)=\prescript{\wedge}{}{q'}$. Then $\phi(l^{\wedge})=q'$. Let $a\in E$ be the source of $l^{\wedge}$. Since $\prescript{\wedge}{}{p}\equiv l$, the terminals of $\prescript{\wedge}{}{p}$ and $l$ belong to the same polygon of $E$. Therefore the sources of $p$ and $l^{\wedge}$ also belong to the same polygon of $E$, and $P(k)=P(h)=P(a)$. Since $\phi(l^{\wedge})=q'=\phi(q)$, we have $\phi(a)=\phi(k)$. Since $\phi$ is a covering of f-BCs, $a=k$. Since the standard sequences $l^{\wedge}$ and $q$ have the same source and the same length, they are equal. Then we have $l=\prescript{\wedge}{}{q}$ and $\prescript{\wedge}{}{p}\equiv \prescript{\wedge}{}{q}$.
\end{proof}

In the remaining of this section, we assume that $E,E'$ are connected $f_s$-BCs. Let $\mathscr{N}'$ be a complete set representatives of non-reduced arrows of $Q_{E'}$ under the equivalence relation $R'$ (for the definition of non-reduced arrows, see Definition \ref{reduced-arrow}), and let $\mathscr{N}$ be the preimage of $\mathscr{N}'$ under the map $f:(Q_E)_1\rightarrow (Q_{E'})_1$.

\begin{Lem}
$\mathscr{N}$ is a complete set representatives of non-reduced arrows of $Q_E$ under the equivalence relation $R$.
\end{Lem}

\begin{proof}
It is straightforward to show that if $u,v$ are two paths of $Q_E$ which belong to $\mathscr{E}$ with $uRv$, then $f(u),f(v)$ are two paths of $Q_{E'}$ which belong to $\mathscr{E}'$ and $f(u)R'f(v)$.

For each $\alpha\in\mathscr{N}$, if $\alpha$ is reduced, then $\alpha R p$ for some path $p$ of $Q_E$ in $\mathscr{E}$ with $l(p)\geq 2$, so we have $f(\alpha) R' f(p)$, which contradicts the fact that $f(\alpha)$ is a non-reduced arrows of $Q_E'$. Then $\mathscr{N}$ is a set of non-reduced arrows of $Q_E$. Note that the arrows in $\mathscr{N}$ are pairwise non-equivalent: For every two arrows $\alpha,\beta\in\mathscr{N}$ with $\alpha R \beta$, we have $f(\alpha) R' f(\beta)$. Since the arrows in $\mathscr{N}$ are pairwise non-equivalent, $f(\alpha)=f(\beta)$. Since $\alpha R \beta$, the arrows $\alpha,\beta$ have the same source. Since $f$ is a covering of quivers, we have $\alpha=\beta$.

If there exists a non-reduced arrow $\alpha$ of $Q_E$ such that $f(\alpha)$ is a reduced arrow of $Q_{E'}$, then $f(\alpha) R' p'$ for some path $p'$ of $Q_{E'}$ in $\mathscr{E}'$ with $l(p')\geq 2$. By Lemma \ref{lift-the-relation-R}, there exists some path $p$ of $Q_E$ in $\mathscr{E}$ with $\alpha R p$ and $f(p)=p'$. Then $l(p)\geq 2$ and $\alpha$ is reduced, a contradiction. Therefore the image of every non-reduced arrow of $Q_E$ under $f$ is also a non-reduced arrow of $Q_{E'}$.

For every non-reduced arrow $\alpha$ of $Q_E$, let $\beta'$ be the unique arrow in $\mathscr{N}'$ such that $f(\alpha) R' \beta'$. By Lemma \ref{lift-the-relation-R}, there exists some path $p$ of $Q_E$ in $\mathscr{E}$ with $\alpha R p$ and $f(p)=\beta'$. Then $l(p)=1$ and $p\in\mathscr{N}$. So $\mathscr{N}$ contains a representative for each equivalence class of non-reduced arrows of $Q_E$ under the equivalence relation $R$.
\end{proof}

Let $Q'_E$ (resp. $Q'_{E'}$) be the subquiver of $Q_E$ (resp. $Q_{E'}$) given by $(Q'_E)_0=(Q_E)_0$ and $(Q'_E)_1=\mathscr{N}$ (resp. $(Q'_{E'})_0=(Q_{E'})_0$ and $(Q'_{E'})_1=\mathscr{N}'$), and let $I'_E$ (resp. $I'_{E'}$) be the kernel of the natural $k$-linear functor $\rho:kQ'_E\rightarrow \Lambda_{E}=kQ_E/I_E$ (resp. $\rho':kQ'_{E'}\rightarrow \Lambda_{E'}=kQ_{E'}/I_{E'}$). According to \cite[Lemma 6.2]{LL}, $\rho$ (resp. $\rho'$) induces an isomorphism between $kQ'_E/I'_E$ (resp. $kQ'_{E'}/I'_{E'}$) and the associate fractional Brauer configuration category $\Lambda_E=kQ_E/I_E$ (resp. $\Lambda_{E'}=kQ_{E'}/I_{E'}$) of $E$ (resp. $E'$).

\begin{Lem}\label{covering-of-admissible-quivers-with-relations}
If the covering $\phi:E\rightarrow E'$ is regular, then it induces a Galois covering of quivers with relations $(Q'_E,I'_E)\rightarrow (Q'_{E'},I'_{E'})$.
\end{Lem}

\begin{proof}
According to Theorem \ref{regular-covering-explicit}, we may assume that $\phi$ is the natural projection $E\rightarrow E/\Pi$, where $\Pi$ is an admissible group of automorphisms of $E$. By Lemma \ref{admissible-free}, $\Pi$ induces a group $\Pi'$ of $(Q_E,I_E)$-automorphisms which acts freely on $Q_E$. By the definition of $\mathscr{N}$ and $\mathscr{N}'$, each automorphism $g'\in\Pi'$ induces a $(Q'_E,I'_E)$-automorphism $g''$. Denote by $\Pi''$ the group of $(Q'_E,I'_E)$-automorphisms induced from $\Pi'$. Since $\Pi'$ acts freely on $Q_E$, $\Pi''$ also acts freely on $Q'_E$. According to Lemma \ref{covering-of-quivers-with-relations}, $\phi$ induces a Galois covering of quivers with relations $f:(Q_E,I_E)\rightarrow (Q_{E/\Pi},I_{E/\Pi})$, and there is an isomorphism $\nu:(Q_E/\Pi',\overline{I_E})\xrightarrow{\sim}(Q_{E/\Pi},I_{E/\Pi})$ of quivers with relations such that the diagram
$$\xymatrix{
		& (Q_E,I_E) \ar[dr]^{f}\ar[dl]_{\pi} &  \\
		(Q_E/\Pi',\overline{I_E})\ar[rr]_{\nu}^{\sim} & & (Q_{E/\Pi},I_{E/\Pi})
	}$$
commutes, where $\pi:(Q_E,I_E)\rightarrow (Q_E/\Pi',\overline{I_E})$ is the natural projection.

Let $f':(Q'_E,I'_E)\rightarrow (Q'_{E/\Pi},I'_{E/\Pi})$ be the morphism of quivers with relations induced by $f$, and let $\pi':(Q'_E,I'_E)\rightarrow (Q'_E/\Pi'',\overline{I'_E})$ be the natural projection. We may consider $Q'_E/\Pi''$ as a subquiver of $Q_E/\Pi'$. It is straightforward to show that $\overline{I'_E}(\overline{a},\overline{b})=\overline{I_E}(\overline{a},\overline{b})\cap k(Q'_E/\Pi'')(\overline{a},\overline{b})$ for each vertices $\overline{a},\overline{b}$ of $Q'_E/\Pi''$. So the quiver isomorphism $\nu':Q'_E/\Pi''\xrightarrow{\sim}Q'_{E/\Pi}$ induced by $\nu$ gives an isomorphism of quivers of relations $(Q'_E/\Pi'',\overline{I'_E})\xrightarrow{\sim}(Q'_{E/\Pi},I'_{E/\Pi})$, and we have a commutative diagram
$$\xymatrix{
		& (Q'_E,I'_E) \ar[dr]^{f'}\ar[dl]_{\pi'} &  \\
		(Q'_E/\Pi'',\overline{I'_E})\ar[rr]_{\nu'}^{\sim} & & (Q'_{E/\Pi},I'_{E/\Pi}).
	}$$
\end{proof}

The following result shows that a universal covering map $\widetilde{E}\rightarrow E$ of an $f_s$-BCs $E$ induces a universal cover $(Q'_{\widetilde{E}},I'_{\widetilde{E}})\rightarrow (Q'_E,I'_E)$ of associated quivers with admissible relations.

\begin{Prop}
Let $E$ be a connected $f_s$-BC and let $p:\widetilde{E}\rightarrow E$ be a universal cover of $E$. Then $p$ induces a universal cover $\rho:(Q'_{\widetilde{E}},I'_{\widetilde{E}})\rightarrow (Q'_E,I'_E)$ of $(Q'_E,I'_E)$, that is, for any covering of quivers with relations $\mu:(Q,I)\rightarrow (Q'_E,I'_E)$ and for every $x\in(Q'_{\widetilde{E}})_0$, $y\in Q_0$ with $\rho(x)=\mu(y)$, there exists a unique covering of quivers with relations $\nu:(Q'_{\widetilde{E}},I'_{\widetilde{E}})\rightarrow (Q,I)$ such that $\mu\nu=\rho$ and $\nu(x)=y$.
\end{Prop}

\begin{proof}
Since the fundamental group of $\widetilde{E}$ is trivial, $p$ is regular. By Lemma \ref{covering-of-admissible-quivers-with-relations}, $\rho$ is a Galois covering of quivers with relations. According to Corollary \ref{iso-of-fundamental-gp-of-f-BC-and-fundamental-gp-of-quiver-with-aadmissible-relation}, the fundamental group of $(Q'_{\widetilde{E}},I'_{\widetilde{E}})$ is trivial. The rest of the proof is similar to that of Corollary \ref{property of universal cover}.
\end{proof}

\subsection{Fractional Brauer configuration categories of type MS are standard}
\

As an application of above discussions, we show that for a connected $f_{ms}$-BC $E$, the associated fractional Brauer configuration category $\Lambda_E$ is standard. Note that we have proved in \cite{LL} that any representation-finite fractional Brauer configuration category of type S is also standard.

\begin{Def}  {\rm(\cite[Section 1]{Sk})} \label{simply-connected-and-standard}
Let $\Lambda$ be a connected locally bounded category with Gabriel quiver $Q$.
\begin{itemize}
\item $\Lambda$ is called simply connected if $Q$ contains no oriented cycles and for any admissible ideal $I$ of $kQ$ with $\Lambda\cong kQ/I$, $\Pi(Q,I)=\{1\}$;
\item $\Lambda$ is called standard if there exists some simply connected locally bounded category $\widetilde{\Lambda}$ such that there exists a Galois covering $\widetilde{\Lambda}\rightarrow\Lambda$.
\end{itemize}
\end{Def}

Note that if $\Lambda$ is a locally representation-finite category, then $\Lambda$ is standard if and only if $\mathrm{ind}\Lambda\cong k(\Gamma_\Lambda)$ where $k(\Gamma_\Lambda)$ denotes the mesh category of the Auslander-Reiten quiver $\Gamma_\Lambda$ of $\Lambda$ (cf. \cite{BG,Br-G}). However, for general case we have to define the standardness by using covering language as above.

\begin{Lem}\label{criterion-simply-connected}
Let $\Lambda=kQ/I$ be a connected locally bounded category, where $Q$ is the Gabriel quiver of $\Lambda$ and $I$ is an admissible ideal in $kQ$, such that $\Pi(Q,I)=\{1\}$ and $Q$ contains no oriented cycles. If for each arrow $x\xrightarrow{\alpha}y$ of $Q$, $\alpha$ is the only path of $Q$ from $x$ to $y$, then $\Lambda$ is simply connected.
\end{Lem}

\begin{proof}
Suppose that $\Lambda\cong kQ/J$ for some admissible ideal $J$ of $kQ$. Let $E:kQ/J\rightarrow kQ/I$ be an isomorphism functor and let $F:kQ\rightarrow kQ/I$ be the functor induced by $E$. For each arrow $x\xrightarrow{\alpha}y$ of $Q$, since $F(\alpha)\in\mathscr{R}(F(x),F(y))-\mathscr{R}^2(F(x),F(y))$, where $\mathscr{R}$ denotes the radical of the category $kQ/I$, there exists some arrow $\mu_{\alpha}$ of $Q$ from $F(x)$ to $F(y)$. Then $\mu_{\alpha}$ is the only path of $Q$ from $F(x)$ to $F(y)$, and we have $F(\alpha)=\lambda_{\alpha}\overline{\mu_{\alpha}}$ for some $\lambda_{\alpha}\in k^{*}$, where $\overline{\mu_{\alpha}}$ denotes the residue class $\mu_{\alpha}+I(F(x),F(y))$. Denote by $f:Q\rightarrow Q$ the quiver automorphism of $Q$ given by $f(x)=F(x)$ for every $x\in Q_0$ and $f(\alpha)=\mu_{\alpha}$ for every $\alpha\in Q_1$. For every path $p=\alpha_r\cdots\alpha_2\alpha_1$ of $Q$ with $\alpha_1,\cdots,\alpha_r\in Q_1$, let $\lambda_p=\lambda_{\alpha_r}\cdots\lambda_{\alpha_1}$. Then $F(p)=\lambda_p \overline{f(p)}$ for each path $p$ of $Q$. For every $x,y\in Q_0$ and for every $\sum_{i=1}^{n}c_i p_i\in kQ(x,y)$ with $c_i\in k^{*}$ and $p_i$'s pairwise distinct paths of $Q$ from $x$ to $y$, we have $\sum_{i=1}^{n}c_i p_i\in J(x,y)$ if and only if $\sum_{i=1}^{n}c_i\lambda_{p_i}f(p_i)\in I(f(x),f(y))$. In particular, $\sum_{i=1}^{n}c_i p_i$ is a minimal relation of $J$ if and only if $\sum_{i=1}^{n}c_i\lambda_{p_i}f(p_i)$ is a minimal relation of $I$. Therefore two paths $p,q$ of $Q$ belong to the same minimal relation of $J$ if and only if $f(p),f(q)$ belong to the same minimal relation of $I$, and $f$ induces a group isomorphism $\Pi(Q,J)\rightarrow\Pi(Q,I)$. So $\Pi(Q,J)=\{1\}$ and $\Lambda$ is simply connected.
\end{proof}

\begin{Prop}\label{f-BCC-of-simply-connected-f-BC-is-simply-connected}
If $E=(E,P,L,d)$ is a simply connected $f_{ms}$-BC, then $\Lambda_{E}$ is simply connected.
\end{Prop}

\begin{proof}
If there exists a polygon $P(e)$ of $E$ such that $d(h)=1$ for all $h\in P(e)$, then it can be shown that $d(h)=1$ for all $h\in E$ since $E$ is connected. Therefore as a set, $E$ is a union of polygons $\sigma^i(P(e))$ with $i\in\mathbb{Z}$, where $\sigma$ denotes the Nakayama automorphism of $E$. If there exists some positive integer $n$ such that $\sigma^n(P(e))=P(e)$, then $\sigma^n$ induces a permutation on $P(e)$, and therefore $\sigma^{mn}(e)=e$ for some positive integer $m$. According to Proposition \ref{unique factorization}, $\overline{(e|g^{mn}|e)}\neq 1$, contradicts the assumption that $E$ is simply connected. Then as a set, $E$ is a disjoint union of the polygons $\sigma^i(P(e))$, $i\in\mathbb{Z}$. It is straightforward to show that $\Lambda_E$ is isomorphic to $kQ/I$, where $Q$ is the quiver
$$\begin{tikzpicture}
\fill (0,0) circle (0.5ex);
\draw[->] (0.2,0)--(0.8,0);
\fill (1,0) circle (0.5ex);
\draw[->] (1.2,0)--(1.8,0);
\fill (2,0) circle (0.5ex);
\draw[->] (2.2,0)--(2.8,0);
\fill (3,0) circle (0.5ex);
\node at(-0.2,0) {$\cdot$};
\node at(-0.4,0) {$\cdot$};
\node at(-0.6,0) {$\cdot$};
\node at(3.2,0) {$\cdot$};
\node at(3.4,0) {$\cdot$};
\node at(3.6,0) {$\cdot$};
\end{tikzpicture}$$
and $I$ is generated by all paths of $Q$ of length $2$. By Lemma \ref{criterion-simply-connected}, $\Lambda_E$ is simply connected. Therefore we may assume that each polygon of $E$ contains an angle $e$ with $d(e)>1$.

Let $Q'_E$ be the subquiver of $Q_E$ with $(Q'_E)_0=(Q_E)_0$ and
$$(Q'_E)_1=\{L(e)\mid e\in E\text{ with } d(e)>1\},$$
and let $I'_E$ be the kernel of the natural functor $\rho:kQ'_E\rightarrow kQ_E/I_E$. According to \cite[Lemma 6.2]{LL}, $Q'_E$ is the Gabriel quiver of $\Lambda_E$ and $I'_E$ is an admissible ideal of $kQ'_E$ with $\Lambda_E\cong kQ'_E/I'_E$. Note also that according to Corollary \ref{iso-of-fundamental-gp-of-f-BC-and-fundamental-gp-of-quiver-with-aadmissible-relation}, $\Pi(Q'_E,I'_E)\cong\Pi(E)=\{1\}$.

If $Q_E$ contains an oriented cycle $p$, we may assume that
$$p=L(g^{n_r-1}\cdot e_r)\cdots L(g\cdot e_r)L(e_r)\cdots L(g^{n_1-1}\cdot e_1)\cdots L(g\cdot e_1)L(e_1),$$
where $e_1,\cdots,e_r\in E$, $n_1,\cdots,n_r\in\mathbb{Z}_{>0}$, and $P(e_{i+1})=P(g^{n_i}\cdot e_{i})$ with $e_{i+1}\neq g^{n_i}\cdot e_{i}$ for each $1\leq i\leq r-1$. Since $p$ is an oriented cycle, we have $P(e_1)=P(g^{n_r}\cdot e_r)$. Since $E$ is simply connected, the walk
$$w=(g^{n_r}\cdot e_r|g^{n_r}|e_r)(e_r|\tau|g^{n_{r-1}}\cdot e_{r-1})\cdots(e_3|\tau|g^{n_2}\cdot e_2)(g^{n_2}\cdot e_2|g^{n_2}|e_2)(e_2|\tau|g^{n_1}\cdot e_1)(g^{n_1}\cdot e_1|g^{n_1}|e_1)$$
of $E$ is homotopic to the walk $v=(g^{n_r}\cdot e_r|\tau|e_1)$ of $E$. Let $\alpha$ be the function on the set of walks of $E$ defined by the formula (\ref{alpha-definition}) in the proof of Proposition \ref{fundamental group of Be}. Since $w\sim v$, we have $\alpha(w)=\alpha(v)=0$. So $r=0$ and $p$ is a trivial path of $Q_E$, a contradiction. Since $Q'_E$ is a subquiver of $Q_E$, $Q'_E$ also contains no oriented cycles.

For each arrow $P(e)\xrightarrow{L(e)}P(g\cdot e)$ of $Q'_E$, if $p$ is a path of $Q'_E$ from $P(e)$ to $P(g\cdot e)$, we may assume that
$$p=L(g^{n_r-1}\cdot e_r)\cdots L(g\cdot e_r)L(e_r)\cdots L(g^{n_1-1}\cdot e_1)\cdots L(g\cdot e_1)L(e_1),$$
where $e_1,\cdots,e_r\in E$, $n_1,\cdots,n_r\in\mathbb{Z}_{>0}$, and $P(e_{i+1})=P(g^{n_i}\cdot e_{i})$ with $e_{i+1}\neq g^{n_i}\cdot e_{i}$ for each $1\leq i\leq r-1$. Since $s(p)=P(e)$ and $t(p)=P(g\cdot e)$, we have $P(e)=P(e_1)$ and $P(g\cdot e)=P(g^{n_r}\cdot e_r)$. Since $E$ is simply connected, the walk
$$w=(g^{n_r}\cdot e_r|g^{n_r}|e_r)(e_r|\tau|g^{n_{r-1}}\cdot e_{r-1})\cdots(e_3|\tau|g^{n_2}\cdot e_2)(g^{n_2}\cdot e_2|g^{n_2}|e_2)(e_2|\tau|g^{n_1}\cdot e_1)(g^{n_1}\cdot e_1|g^{n_1}|e_1)$$
of $E$ is homotopic to the walk $v=(g^{n_r}\cdot e_r|\tau|g\cdot e)(g\cdot e|g|e)(e|\tau|e_1)$ of $E$. Since $L(e)\in (Q'_E)_1$, $d(e)>1$. So $\alpha(v)=\frac{1}{d(e)}<1$. Since $w\sim v$, we have $\alpha(w)=\alpha(v)$. Since $\alpha(w)=\frac{n_1}{d(e_1)}+\cdots+\frac{n_r}{d(e_r)}$, we imply that $n_i<d(e_i)$ for all $1\leq i\leq r$. Then $w$ is a special walk of $E$. According to Proposition \ref{unique factorization}, we have $e=e_1$, $g^{n_r}\cdot e_r=g\cdot e$, $r=1$ and $n_1=1$. So $p=L(e_1)=L(e)$, and $L(e)$ is the only path of $Q'_E$ from $P(e)$ to $P(g\cdot e)$. According to Lemma \ref{criterion-simply-connected}, $\Lambda_E\cong kQ'_E/I'_E$ is simply connected.
\end{proof}

\begin{Cor} \label{type-ms-is-standard}
If $E$ is a connected $f_{ms}$-BC and $\widetilde{E}$ is its universal cover, then the induced universal cover $\Lambda_{\widetilde{E}}$ of $\Lambda_E$ is simply connected. Especially, $\Lambda_E$ is standard.
\end{Cor}

\begin{proof}
It follows directly from Proposition \ref{f-BCC-of-simply-connected-f-BC-is-simply-connected} and Proposition \ref{induce-Galois-covering-of-categories}.
\end{proof}

From the discussion above, it is interesting to know whether every fractional Brauer configuration category of type S is standard.

\section{The fundamental group of a Brauer configuration}

In this section, we calculate the fundamental group of any connected Brauer configuration. Recall that a Brauer configuration (abbr. BC) is a finite $f_{ms}$-BC $E=(E,P,L,d)$ with integral f-degree and containing no $1$-gons. We first define the notion of sub-f-BC and prove an analogy of the Van Kampen theorem for f-BC (Proposition \ref{Van-Kampen}). Then we use this proposition to reduce the calculation of the fundamental group of a BC to the calculation of the fundamental group of a Brauer graph (abbr. BG), that is, a BC with each polygon containing two elements (Proposition \ref{iso. of direct systems} and Corollary \ref{fundamental group of BC and BG}). The main result of this section is Theorem \ref{fundamental group of BC}.

We first recall the definition of simplicial complex (cf. \cite[Definition 1.1]{Sc}). Let $X$ be a finite set. A {\it simplicial complex} $K$ on $X$ is a collection of nonempty  subsets of $X$ such that
\begin{itemize}
\item $\{x\}\in K$ for each $x\in X$;
\item each nonempty subset of each element of $K$ also belongs to $K$,
\end{itemize}
where the elements of $K$ are called simplices. A walk of $K$ is a sequence $(x_0,x_1,\cdots,x_n)$ of elements of $X$ ($n\geq 0$) such that for each $1\leq i\leq n$, there exists some simplex $\sigma$ of $K$ such that $x_{i-1},x_i\in \sigma$. The homotopy relation on the set of all walks of $K$ is an equivalence relation generated by
\begin{itemize}
\item[(i)] $(x_0$, $x_1$, $x_2)$ is homotopic to $(x_0$, $x_2)$, if there exists a simplex $\sigma$ of $K$ with $x_0$, $x_1$, $x_2\in\sigma$;
\item[(ii)] $(x$, $x)$ is homotopic to the trivial walk $(x)$ for each element $x$ of $X$;
\item[(iii)] $uw_{1}v$ is homotopic to $uw_{2}v$ if $w_1$ homotopic to $w_2$, where $w_1$, $w_2$, $u$, $v$ are walks of $K$ such that the compositions make sense.
\end{itemize}
For each $x\in X$, define the fundamental group $\Pi(K,x)$ of $K$ at $x$ as the group of homotopy classes of closed walks of $K$ at $x$, where the multiplication is given by the composition of walks. $K$ is said to be connected if every two elements of $X$ can be connected by a walk of $K$, and $K$ is said to be simply connected if $K$ is connected and the fundamental group of $K$ is trivial.

\begin{Def}\label{sub-pre-configuration}
Let $E$ be an f-BC. A sub-f-BC $F$ of $E$ is a $G$-invariant subset of $E$ together with an f-BC structure, such that the inclusion map $F\hookrightarrow E$ is a morphism of f-BCs.
\end{Def}

\medskip
Equivalently, a sub-f-BC $F$ of an f-BC $E=(E,P,L,d)$ is a subset of $E$ which is a union of some $G$-orbits, together with two partitions $P_{F}$, $L_{F}$ of $F$, such that \\
(1) $P_{F}(e)\subseteq P(e)$ and $L_{F}(e)\subseteq L(e)$ for all $e\in F$. \\
(2) $L_{F}(e)\subseteq P_{F}(e)$ for all $e\in F$. \\
(3) $L_{F}(e_1)=L_{F}(e_2)$ implies $P_{F}(g\cdot e_1)=P_{F}(g\cdot e_2)$ for $e_1$, $e_2\in F$. \\
(4) $P_{F}(e_1)=P_{F}(e_2)$ if and only if $P_{F}(g^{d(e_1)}\cdot e_1)=P_{F}(g^{d(e_2)}\cdot e_2)$. \\
(5) $L_{F}(e_1)=L_{F}(e_2)$ if and only if $L_{F}(g^{d(e_1)}\cdot e_1)=L_{F}(g^{d(e_2)}\cdot e_2)$. \\
The $G$-set structure and the degree function on $F$ are inherited from $E$.

\medskip
Let $\{E_{\alpha}=(E_{\alpha},P_{\alpha},L_{\alpha},d)\}_{\alpha\in I}$ be a family of sub-f-BCs of $E=(E,P,L,d)$. For $\alpha\in I$ and $e$, $h\in E_{\alpha}$, define $e\sim_{\alpha}h$ if $P_{\alpha}(e)=P_{\alpha}(h)$. For $e$, $h\in \bigcup_{\alpha\in I}E_{\alpha}$, define $e\sim h$ if there exists a sequence of elements $e=e_0$, $e_1$, $\cdots$, $e_{n-1}$, $e_n=h$ in $\bigcup_{\alpha\in I}E_{\alpha}$, such that for each $1\leq i\leq n$, there exists some $\alpha\in I$ with $e_{i-1}$, $e_i\in E_{\alpha}$ and $e_{i-1}\sim_{\alpha}e_i$. Then $\sim$ is an equivalence relation on $\bigcup_{\alpha\in I}E_{\alpha}$. Define a partition $P'$ of $\bigcup_{\alpha\in I}E_{\alpha}$ as $P'(e)=\{h\in\bigcup_{\alpha\in I}E_{\alpha}\mid e\sim h\}$ for each $e\in\bigcup_{\alpha\in I}E_{\alpha}$. Similarly, one can define a partition $L'$ of $\bigcup_{\alpha\in I}E_{\alpha}$ using the partitions $L_{\alpha}$'s. Then $\bigcup_{\alpha\in I}E_{\alpha}$ together with partitions $P'$, $L'$ forms a sub-f-BC of $E$, which is said to be the {\it union} of sub-f-BCs $\{E_{\alpha}\}_{\alpha\in I}$ of $E$. Moreover, let $P''$ (resp. $L''$) be the partition of $\bigcap_{\alpha\in I}E_{\alpha}$ given by $P''(e)=\bigcap_{\alpha\in I}P_{\alpha}(e)$ (resp. $L''(e)=\bigcap_{\alpha\in I}L_{\alpha}(e)$) for each $e\in\bigcap_{\alpha\in I}E_{\alpha}$. Then $(\bigcap_{\alpha\in I}E_{\alpha},P'',L'',d)$ is also a sub-f-BC of $E$, which is said to be the {\it intersection} of sub-f-BCs $\{E_{\alpha}\}_{\alpha\in I}$ of $E$.

Let $E=(E,P,L,d)$ be an f-BC, which is a union of the family of sub-f-BCs \\ $\{E_{\alpha}=(E_{\alpha},P_{\alpha},L_{\alpha},d)\}_{\alpha\in I}$ which is closed under finite intersections. For each $e\in E$, let $K_e$ be the simplicial complex on $P(e)$ generated by simplices of the form $P_{\alpha}(h)$, where $\alpha\in I$ and $h\in P(e)\cap E_{\alpha}$ (that is, $K_e$ consists of simplices of the form $P_{\alpha}(h)$ together with their nonempty subsets). We call $E$ an {\it admissible union} of sub-f-BCs $\{E_{\alpha}\}_{\alpha\in I}$ if the simplicial complex $K_e$ is simply connected for each $e\in E$.

The following proposition is an analogy of the Van Kampen theorem (cf. \cite{M}). Note that the assumption that each simplicial complex $K_e$ is simply connected is necessary in the construction of the functor $F$ in defining the direct limit, and this assumption is clearly satisfied when each polygon of $E$ contains at most two angles.

\begin{Prop}\label{Van-Kampen}
Let $E$ be an f-BC, which is an admissible union of a family of sub-f-BCs $\{E_{\alpha}\}_{\alpha\in I}$. Let $A$ be a subset of $\bigcap_{\alpha\in I}E_{\alpha}$ such that for each $\alpha\in I$, $A$ meets each connected component of $E_{\alpha}$. Then the groupoid $\Pi(E,A)$ is the direct limit of the groupoids $\Pi(E_{\alpha},A)$.
\end{Prop}

\begin{proof}
Denote by $E=(E,P,L,d)$ and $E_{\alpha}=(E_{\alpha},P_{\alpha},L_{\alpha},d)$ for each $\alpha\in I$. We need to show that for each groupoid $\mathscr{G}$ together with a functor $F_{\alpha}:\Pi(E_{\alpha},A)\rightarrow \mathscr{G}$ for each $\alpha\in I$, such that $F_{\alpha}=F_{\beta}j_{\beta\alpha *}$ for all $\alpha$, $\beta\in I$ with $E_{\alpha}$ a sub-f-BC of $E_{\beta}$, where $j_{\beta\alpha}:E_{\alpha}\rightarrow E_{\beta}$ is the inclusion map, there exists a unique functor $F:\Pi(E,A)\rightarrow \mathscr{G}$ such that $F_{\alpha}=Fj_{\alpha *}$ for all $\alpha\in I$, where $j_{\alpha}:E_{\alpha}\rightarrow E$ is the inclusion map. For each $a\in A$, define $F(a)=F_{\alpha}(a)$ for some $\alpha\in I$. Since the family of sub-f-BCs $\{E_{\alpha}\}_{\alpha\in I}$ of $E$ is closed under finite intersections, the value of $F(a)$ does not depend on the choice of $\alpha\in I$. To define the value of $F$ on morphisms of $\Pi(E,A)$, and to show $F$ is the unique functor which satisfies the conditions above, we will divide our proof into seven steps.

A walk $w$ of $E$ is called divisible if it can be written as $w=w_n \cdots w_2 w_1$, where $w_i$ is a walk of some $E_{\alpha_i}$. Call such a tuple $(w_1,\cdots,w_n,E_{\alpha_1},\cdots,E_{\alpha_n})$ a division of $w$. For a divisible walk $w$ of $E$ with $s(w)$, $t(w)\in A$, let $(w_1,\cdots,w_n,E_{\alpha_1},\cdots,E_{\alpha_n})$ be a division of $w$, and let $e_i=t(w_i)$ for each $1\leq i\leq n-1$. Since $\{E_{\alpha}\}_{\alpha\in I}$ is closed under finite intersections, for each $1\leq i\leq n-1$, there exists some $\beta_i\in I$ such that $E_{\beta_i}=E_{\alpha_i}\cap E_{\alpha_{i+1}}$. Then $e_i\in E_{\beta_i}$ for each $1\leq i\leq n-1$. Since $A$ meets each connected component of $E_{\beta_i}$, there exists some $a_i\in A$ such that there exists a walk $u_i$ of $E_{\beta_i}$ from $a_i$ to $e_i$. Since $F_{\alpha_{i}}(a_i)=F_{\beta_i}(a_i)=F_{\alpha_{i+1}}(a_i)$ for $1\leq i\leq n-1$, we can define a morphism $F(w)$ of $\mathscr{G}$ as $F(w)=F_{\alpha_n}(\overline{w_n u_{n-1}})F_{\alpha_{n-1}}(\overline{u_{n-1}^{-1}w_{n-1}u_{n-2}})\cdots F_{\alpha_{2}}(\overline{u_{2}^{-1}w_{2}u_{1}})F_{\alpha_{1}}(\overline{u_{1}^{-1}w_{1}})$.

\medskip
{\it Step 1: To show that the value of $F(w)$ does not depend on the choices of $a_i$ together with the walks $u_i$ for $1\leq i\leq n-1$.}

Let $v_i$ be another walk of $E_{\beta_i}=E_{\alpha_i}\cap E_{\alpha_{i+1}}$ from $b_i$ to $e_i$ for each $1\leq i\leq n-1$, where $b_i\in A$. Then
\begin{multline*}
F_{\alpha_n}(\overline{w_n v_{n-1}})F_{\alpha_{n-1}}(\overline{v_{n-1}^{-1}w_{n-1}v_{n-2}})\cdots F_{\alpha_{2}}(\overline{v_{2}^{-1}w_{2}v_{1}})F_{\alpha_{1}}(\overline{v_{1}^{-1}w_{1}}) \\
=F_{\alpha_n}(\overline{w_n u_{n-1}})F_{\alpha_n}(\overline{u_{n-1}^{-1}v_{n-1}})
F_{\alpha_{n-1}}(\overline{v_{n-1}^{-1}u_{n-1}})F_{\alpha_{n-1}}(\overline{u_{n-1}^{-1}w_{n-1}u_{n-2}})F_{\alpha_{n-1}}(\overline{u_{n-2}^{-1}v_{n-2}})
\cdots \\
F_{\alpha_{2}}(\overline{v_{2}^{-1}u_{2}})F_{\alpha_{2}}(\overline{u_{2}^{-1}w_{2}u_{1}})F_{\alpha_{2}}(\overline{u_{1}^{-1}v_{1}})
F_{\alpha_{1}}(\overline{v_{1}^{-1}u_{1}})F_{\alpha_{1}}(\overline{u_{1}^{-1}w_{1}}) \\
=F_{\alpha_n}(\overline{w_n u_{n-1}})F_{\beta_{n-1}}(\overline{u_{n-1}^{-1}v_{n-1}})
F_{\beta_{n-1}}(\overline{v_{n-1}^{-1}u_{n-1}})F_{\alpha_{n-1}}(\overline{u_{n-1}^{-1}w_{n-1}u_{n-2}})F_{\beta_{n-2}}(\overline{u_{n-2}^{-1}v_{n-2}})
\cdots \\
F_{\beta_{2}}(\overline{v_{2}^{-1}u_{2}})F_{\alpha_{2}}(\overline{u_{2}^{-1}w_{2}u_{1}})F_{\beta_{1}}(\overline{u_{1}^{-1}v_{1}})
F_{\beta_{1}}(\overline{v_{1}^{-1}u_{1}})F_{\alpha_{1}}(\overline{u_{1}^{-1}w_{1}}) \\
=F_{\alpha_n}(\overline{w_n u_{n-1}})F_{\alpha_{n-1}}(\overline{u_{n-1}^{-1}w_{n-1}u_{n-2}})\cdots F_{\alpha_{2}}(\overline{u_{2}^{-1}w_{2}u_{1}})F_{\alpha_{1}}(\overline{u_{1}^{-1}w_{1}}).
\end{multline*}

\medskip
{\it Step 2: To show that the value of $F(w)$ does not depend on the choice of the division \\ $(w_1,\cdots,w_n,E_{\alpha_1},\cdots,E_{\alpha_n})$ of $w$.}

Define a refinement of a division $(w_1,\cdots,w_n,E_{\alpha_1},\cdots,E_{\alpha_n})$ of $w$ as a division \\ $(w_{1,1},\cdots,w_{1,k_{1}},\cdots,w_{n,1},\cdots,w_{n,k_{n}}, E_{\alpha_{1,1}},\cdots,E_{\alpha_{1,k_{1}}},\cdots,E_{\alpha_{n,1}}
,\cdots,E_{\alpha_{n,k_{n}}})$ of $w$, such that $w_i =w_{i,k_{i}}\cdots w_{i,1}$ for each $1\leq i\leq n$ and $E_{\alpha_{i,1}}$, $\cdots$, $E_{\alpha_{i,k_{i}}}$ are sub-f-BCs of $E_{\alpha_{i}}$ for each $1\leq i\leq n$. Since $\{E_{\alpha}\}_{\alpha\in I}$ is closed under finite intersections, for two divisions $(u_1,\cdots,u_n,E_{\alpha_1},\cdots,E_{\alpha_n})$ and $(v_1,\cdots,v_m,E_{\beta_1},\cdots,E_{\beta_m})$ of $w$, there exists a common refinement $(w_1,\cdots,w_k,E_{\gamma_1},\cdots,E_{\gamma_k})$. Therefore it suffices to show the value of $F(w)$ given by a refinement \\ $(w_{1,1},\cdots,w_{1,k_{1}},\cdots,w_{n,1},\cdots,w_{n,k_{n}}, E_{\alpha_{1,1}},\cdots,E_{\alpha_{1,k_{1}}},\cdots,E_{\alpha_{n,1}}
,\cdots,E_{\alpha_{n,k_{n}}})$ \\ of the division $(w_1,\cdots,w_n,E_{\alpha_1},\cdots,E_{\alpha_n})$ of $w$ is the same as the value of $F(w)$ given by the division $(w_1,\cdots,w_n,E_{\alpha_1},\cdots,E_{\alpha_n})$ of $w$.

For $1\leq i\leq n$ and $1\leq j< k_i$, let $u_{i,j}$ be a walk of $E_{\alpha_{i,j}}\cap E_{\alpha_{i,j+1}}$ from $a_{i,j}$ to $t(w_{i,j})$ where $a_{i,j}\in A$. For $1\leq i\leq n-1$, let $u_{i}$ be a walk of $E_{\alpha_{i,k_i}}\cap E_{\alpha_{i+1,1}}$ from $a_i$ to $t(w_i)$ where $a_i\in A$. Then $u_i$ is also a walk of $E_{\alpha_{i}}\cap E_{\alpha_{i+1}}$. The value of $F(w)$ given by the division \\ $(w_{1,1},\cdots,w_{1,k_{1}},\cdots,w_{n,1},\cdots,w_{n,k_{n}}, E_{\alpha_{1,1}},\cdots,E_{\alpha_{1,k_{1}}},\cdots,E_{\alpha_{n,1}}
,\cdots,E_{\alpha_{n,k_{n}}})$ is

\begin{multline*}
F_{\alpha_{n,k_n}}(\overline{w_{n,k_n}u_{n,k_{n}-1}})\cdots F_{\alpha_{n,2}}(\overline{u_{n,2}^{-1}w_{n,2}u_{n,1}})F_{\alpha_{n,1}}(\overline{u_{n,1}^{-1}w_{n,1}u_{n-1}})\cdots \\
F_{\alpha_{2,k_2}}(\overline{u_{2}^{-1}w_{2,k_2}u_{2,k_{2}-1}})\cdots F_{\alpha_{2,2}}(\overline{u_{2,2}^{-1}w_{2,2}u_{2,1}})F_{\alpha_{2,1}}(\overline{u_{2,1}^{-1}w_{2,1}u_1}) \\
F_{\alpha_{1,k_1}}(\overline{u_{1}^{-1}w_{1,k_1}u_{1,k_{1}-1}})\cdots F_{\alpha_{1,2}}(\overline{u_{1,2}^{-1}w_{1,2}u_{1,1}})F_{\alpha_{1,1}}(\overline{u_{1,1}^{-1}w_{1,1}}) \\
=F_{\alpha_{n}}(\overline{w_{n,k_n}u_{n,k_{n}-1}})\cdots F_{\alpha_{n}}(\overline{u_{n,2}^{-1}w_{n,2}u_{n,1}})F_{\alpha_{n}}(\overline{u_{n,1}^{-1}w_{n,1}u_{n-1}})\cdots \\
F_{\alpha_{2}}(\overline{u_{2}^{-1}w_{2,k_2}u_{2,k_{2}-1}})\cdots F_{\alpha_{2}}(\overline{u_{2,2}^{-1}w_{2,2}u_{2,1}})F_{\alpha_{2}}(\overline{u_{2,1}^{-1}w_{2,1}u_1}) \\
F_{\alpha_{1}}(\overline{u_{1}^{-1}w_{1,k_1}u_{1,k_{1}-1}})\cdots F_{\alpha_{1}}(\overline{u_{1,2}^{-1}w_{1,2}u_{1,1}})F_{\alpha_{1}}(\overline{u_{1,1}^{-1}w_{1,1}}) \\
=F_{\alpha_{n}}(\overline{w_{n,k_n}u_{n,k_{n}-1}\cdots u_{n,2}^{-1}w_{n,2}u_{n,1}u_{n,1}^{-1}w_{n,1}u_{n-1}})\cdots \\
F_{\alpha_{2}}(\overline{u_{2}^{-1}w_{2,k_2}u_{2,k_{2}-1}\cdots u_{2,2}^{-1}w_{2,2}u_{2,1}u_{2,1}^{-1}w_{2,1}u_1}) \\
F_{\alpha_{1}}(\overline{u_{1}^{-1}w_{1,k_1}u_{1,k_{1}-1}\cdots u_{1,2}^{-1}w_{1,2}u_{1,1}u_{1,1}^{-1}w_{1,1}}) \\
=F_{\alpha_n}(\overline{w_n u_{n-1}})\cdots F_{\alpha_{2}}(\overline{u_{2}^{-1}w_{2}u_{1}})F_{\alpha_{1}}(\overline{u_{1}^{-1}w_{1}}),
\end{multline*}
which is equal to the value of $F(w)$ given by the division $(w_1,\cdots,w_n,E_{\alpha_1},\cdots,E_{\alpha_n})$.

As a summery, for each $a_1$, $a_2\in A$, we have defined a map $F:\{$divisible walks $w$ of $E$ with $s(w)=a_1$ and $t(w)=a_2\}\rightarrow \mathscr{G}(F(a_1),F(a_2))$.

\medskip
{\it Step 3: To extend the domain of $F$ to all walks $w$ of $E$ with $s(w)$, $t(w)\in A$.}

Let $v=(y|\delta|x)$ be a walk of $E$ of length $1$, where $\delta\in \{g,g^{-1},\tau\}$. When $\delta=g$ or $\delta=g^{-1}$, $v$ is also a walk of some $E_{\alpha}$. When $\delta=\tau$, we have $P(x)=P(y)$, so there exists a sequence of elements $e_0 =x$, $e_1$, $\cdots$, $e_n =y$ of $P(x)$ such that for each $1\leq i\leq n$, $P_{\alpha_i}(e_{i-1})=P_{\alpha_i}(e_{i})$ for some $\alpha_i\in I$. Call a sequence of this form a {\it connecting sequence} from $x$ to $y$. Note that such a connecting sequence defines a walk $(e_0,e_1,\cdots,e_n)$ of the simplicial complex $K_x$.

Let $w$ be a walk of $E$ with $s(w)$, $t(w)\in A$. For each subwalk $v=(y|\tau|x)$ of $w$ of length $1$, choose a connecting sequence $e_0 =x$, $e_1$, $\cdots$, $e_n =y$ from $x$ to $y$ and replace the subwalk $v=(y|\tau|x)$ of $w$ by $(e_n|\tau|e_{n-1})\cdots(e_2|\tau|e_1)(e_1|\tau|e_0)$. Then we obtain a divisible walk $w'$ of $E$ with $s(w')$, $t(w')\in A$, which is homotopic to $w$. Define $F(w):=F(w')$.

\medskip
{\it Step 4: To show that for each walk $w$ of $E$ with $s(w)$, $t(w)\in A$, the value of $F(w)$ is independent of the choices of connecting sequences $e_0 =x$, $e_1$, $\cdots$, $e_n =y$ for each subwalk $v=(y|\tau|x)$ of $w$.}

It suffices to show the following fact: Let $w'=u(e_n|\tau|e_{n-1})\cdots(e_2|\tau|e_1)(e_1|\tau|e_0)v$ and  $w''=u(h_m|\tau|h_{m-1})\cdots(h_2|\tau|h_1)(h_1|\tau|h_0)v$ be two divisible walks with $s(v),t(u)\in A$, where $e_0 =x$, $e_1$, $\cdots$, $e_n =y$ and $h_0 =x$, $h_1$, $\cdots$, $h_m =y$ are two connecting sequences from $x$ to $y$. Then $F(w')=F(w'')$.

Since $K_x$ is simply connected, $(e_{0}$, $e_{1}$, $\cdots$, $e_{n})$ and $(h_{0}$, $h_{1}$, $\cdots$, $h_{m})$ are homotopic walks of $K_x$. Then we can assume that the walk $(h_{0}$, $h_{1}$, $\cdots$, $h_{m})$ is obtained from $(e_{0}$, $e_{1}$, $\cdots$, $e_{n})$ by replacing a subwalk with another, using relation (i) or relation (ii) in the definition of homotopy relations of walks of a simplicial complex (see the beginning of Section 5). When $(h_{0}$, $h_{1}$, $\cdots$, $h_{m})$ is obtained from $(e_{0}$, $e_{1}$, $\cdots$, $e_{n})$ by replacing a subwalk with another using relation (i), then $m=n-1$. We may assume that $h_{0}=e_{0}$ and $h_{k}=e_{k+1}$ for $1\leq k\leq m$ with $e_{0}$, $e_{1}$, $e_{2}$ belong to a simplex $P_{\alpha}(z)$ of $K_x$ for some $\alpha\in I$ and $z\in P(x)\cap E_{\alpha}$. Then
\begin{multline*}
F(w'')=F(u(h_m|\tau|h_{m-1})\cdots(h_2|\tau|h_1)(h_1|\tau|h_0)v) \\
=F(u(e_n|\tau|e_{n-1})\cdots(e_3|\tau|e_2)(e_2|\tau|e_0)v) \\
=F(u(e_n|\tau|e_{n-1})\cdots(e_3|\tau|e_2)(e_2|\tau|e_1)(e_1|\tau|e_0)v)=F(w'),
\end{multline*}
where the third identity follows from the fact that $e_0$, $e_1$, $e_2\in P_{\alpha}(z)$. When $(h_{0}$, $h_{1}$, $\cdots$, $h_{m})$ is obtained from $(e_{0}$, $e_{1}$, $\cdots$, $e_{n})$ by replacing a subwalk with another using relation (ii), the proof of $F(w')=F(w'')$ is similar.

\medskip
{\it Step 5: To show that $F(w)=F(\widetilde{w})$ for homotopic walks $w$, $\widetilde{w}$ of $E$ with $s(w)$, $t(w)\in A$.}

We may assume that $\widetilde{w}$ is obtained from $w$ by replacing a subwalk with another, using one of the relations $(h1)$, $(h2)$, $(h3)$ or $(h4)$ in Definition \ref{homotopy of walks}.

\medskip
{\it Case 1: $\widetilde{w}$ is obtained from $w$ by replacing a subwalk with another, using relation $(h1)$ in Definition \ref{homotopy of walks}.}

Let $w=u(h|g^{-1}g|h)v$ and $\widetilde{w}=uv$ for some walks $u$, $v$ of $E$ with $t(u)$, $s(v)\in A$. As in Step 3, we may replace each subwalk of $u$ and $v$ of the form $(y|\tau|x)$ by a walk of the form $(e_n|\tau|e_{n-1})\cdots(e_2|\tau|e_1)(e_1|\tau|e_0)$, where $e_0 =x$, $e_1$, $\cdots$, $e_n =y$ is a connecting sequence from $x$ to $y$. Then we obtain divisible walks $u'$ and $v'$. Therefore $F(w)=F(u'(h|g^{-1}g|h)v')$ and $F(\widetilde{w})=F(u'v')$, where $u'(h|g^{-1}g|h)v'$ and $u'v'$ are divisible walks. Choose a division $(u_1,\cdots,u_n,E_{\alpha_1},\cdots,E_{\alpha_n})$ of $u'$ and a division $(v_1,\cdots,v_m,E_{\beta_1},\cdots,E_{\beta_m})$ of $v'$. Since $h\in E_{\alpha_1}$, the walk $u_1(h|g^{-1}g|h)$ of $E$ is also a walk of $E_{\alpha_1}$. Therefore $$(v_1,\cdots,v_m,u_1(h|g^{-1}g|h),u_2,\cdots,u_n,E_{\beta_1},\cdots,E_{\beta_m},E_{\alpha_1},E_{\alpha_2},\cdots,E_{\alpha_n})$$
is a division of $u'(h|g^{-1}g|h)v'$ and
$$(v_1,\cdots,v_m,u_1,\cdots,u_n,E_{\beta_1},\cdots,E_{\beta_m},E_{\alpha_1},\cdots,E_{\alpha_n})$$
is a division of $u'v'$. Let $r_i$ be a walk of $E_{\beta_i}\cap E_{\beta_{i+1}}$ from $a_i\in A$ to $t(v_i)$ for $1\leq i\leq m-1$, $r_m$ be a walk of $E_{\beta_m}\cap E_{\alpha_1}$ from $a_m\in A$ to $h$, and $r_{m+j}$ be a walk of $E_{\alpha_j}\cap E_{\alpha_{j+1}}$ from $a_{m+j}\in A$ to $t(u_j)$ for $1\leq j\leq n-1$. Then
\begin{multline*}
F(w)=F(u'(h|g^{-1}g|h)v') \\
=F_{\alpha_n}(\overline{u_n r_{m+n-1}})\cdots F_{\alpha_{2}}(\overline{r_{m+2}^{-1}u_{2}r_{m+1}})F_{\alpha_{1}}(\overline{r_{m+1}^{-1}u_{1}(h|g^{-1}g|h)r_{m}})F_{\beta_m}(\overline{r_{m}^{-1}v_m r_{m-1}})\cdots \\
F_{\beta_{2}}(\overline{r_{2}^{-1}v_{2}r_{1}})F_{\beta_{1}}(\overline{r_{1}^{-1}v_{1}}) \\
=F_{\alpha_n}(\overline{u_n r_{m+n-1}})\cdots F_{\alpha_{2}}(\overline{r_{m+2}^{-1}u_{2}r_{m+1}})F_{\alpha_{1}}(\overline{r_{m+1}^{-1}u_{1}r_{m}})F_{\beta_m}(\overline{r_{m}^{-1}v_m r_{m-1}})\cdots \\
F_{\beta_{2}}(\overline{r_{2}^{-1}v_{2}r_{1}})F_{\beta_{1}}(\overline{r_{1}^{-1}v_{1}}) \\
=F(u'v')=F(\widetilde{w}).
\end{multline*}

If $w=u(h|gg^{-1}|h)v$ or $w=u(h|\tau|h)v$, similarly one can show that $F(w)=F(\widetilde{w})$, where $\widetilde{w}=uv$.

\medskip
{\it Case 2: $\widetilde{w}$ is obtained from $w$ by replacing a subwalk with another, using relation $(h2)$ in Definition \ref{homotopy of walks}.}

Let $w=u(e_2|\tau|e_1)(e_1|\tau|e_0)v$ and $\widetilde{w}=u(e_2|\tau|e_0)v$ for some walks $u$, $v$ of $E$ with $t(u)$, $s(v)\in A$ and $P(e_0)=P(e_1)=P(e_2)$. As in Step 3, we may replace $u$, $v$ by divisible walks $u'$, $v'$ respectively. Choose a connecting sequence $h_0 =e_0$, $h_1$, $\cdots$, $h_n =e_1$ from $e_0$ to $e_1$ and a connecting sequence $h_n =e_1$, $h_{n+1}$, $\cdots$, $h_{n+m} =e_2$ from $e_1$ to $e_2$. Then $h_0 =e_0$, $h_1$, $\cdots$, $h_n =e_1$, $h_{n+1}$, $\cdots$, $h_{n+m} =e_2$ is a connecting sequence from $e_0$ to $e_2$. It follows that \\ $r=u'(h_{n+m}|\tau|h_{n+m-1})\cdots(h_{n+2}|\tau|h_{n+1})(h_{n+1}|\tau|h_n)(h_n|\tau|h_{n-1})\cdots(h_2|\tau|h_1)(h_1|\tau|h_0)v'$ is divisible and $F(w)=F(r)=F(\widetilde{w})$.

\medskip
{\it Case 3: $\widetilde{w}$ is obtained from $w$ by replacing a subwalk with another, using relation $(h3)$ in Definition \ref{homotopy of walks}.}

Denote by $\sigma$ the Nakayama automorphism of $E$. Let $w=u(\sigma(e_2)|\tau|\sigma(e_1))(\sigma(e_1)|g^{d(e_1)}|e_1)v$ and $\widetilde{w}=u(\sigma(e_2)|g^{d(e_2)}|e_2)(e_2|\tau|e_1)v$ for some walks $u$, $v$ of $E$ with $t(u)$, $s(v)\in A$ and $P(e_1)=P(e_2)$. As in Step 3, we may replace $u$, $v$ by divisible walks $u'$, $v'$ respectively. Choose a connecting sequence $h_0 =e_1$, $h_1$, $\cdots$, $h_n =e_2$ from $e_1$ to $e_2$. Then $\sigma(h_0) =\sigma(e_1)$, $\sigma(h_1)$, $\cdots$, $\sigma(h_n) =\sigma(e_2)$ is a connecting sequence from $\sigma(e_1)$ to $\sigma(e_2)$. For each $0\leq i\leq n$, let
\begin{equation*}
 r_i=u'(\sigma(h_{n})|\tau|\sigma(h_{n-1}))\cdots(\sigma(h_{i+1})|\tau|\sigma(h_i))(\sigma(h_i)|g^{d(h_i)}|h_i)(h_i|\tau|h_{i-1})\cdots(h_1|\tau|h_0)v'.
\end{equation*}
Each $r_i$ is divisible with $F(w)=F(r_0)$ and $F(\widetilde{w})=F(r_n)$. To show $F(w)=F(\widetilde{w})$, it suffices to show $F(r_{i-1})=F(r_i)$ for each $1\leq i\leq n$.
Choose a division $(u_1,\cdots,u_k,E_{\alpha_1},\cdots,E_{\alpha_k})$ of $u'(\sigma(h_{n})|\tau|\sigma(h_{n-1}))\cdots(\sigma(h_{i+1})|\tau|\sigma(h_i))$ and a division $(v_1,\cdots,v_m,E_{\beta_1},\cdots,E_{\beta_m})$ of \\ $(h_{i-1}|\tau|h_{i-2})\cdots(h_1|\tau|h_0)v'$. Choose $\gamma\in I$ such that $h_{i-1},h_i\in E_{\gamma}$ and $P_{\gamma}(h_{i-1})=P_{\gamma}(h_i)$. Then both  $(\sigma(h_{i})|\tau|\sigma(h_{i-1}))(\sigma(h_{i-1})|g^{d(h_{i-1})}|h_{i-1})$ and $(\sigma(h_i)|g^{d(h_i)}|h_i)(h_i|\tau|h_{i-1})$ are walks of $E_{\gamma}$. It follows that
\begin{equation*} (v_1,\cdots,v_m,(\sigma(h_{i})|\tau|\sigma(h_{i-1}))(\sigma(h_{i-1})|g^{d(h_{i-1})}|h_{i-1}),u_1,\cdots,u_k,E_{\beta_1},\cdots,
E_{\beta_m},E_{\gamma},E_{\alpha_1},\cdots,E_{\alpha_k})
\end{equation*} is a division of $r_{i-1}$ and
\begin{equation*}(v_1,\cdots,v_m,(\sigma(h_i)|g^{d(h_i)}|h_i)(h_i|\tau|h_{i-1}),u_1,\cdots,u_k,E_{\beta_1},\cdots,
E_{\beta_m},E_{\gamma},E_{\alpha_1},\cdots,E_{\alpha_k})
\end{equation*} is a division of $r_i$. Since $(\sigma(h_{i})|\tau|\sigma(h_{i-1}))(\sigma(h_{i-1})|g^{d(h_{i-1})}|h_{i-1})$ and $(\sigma(h_i)|g^{d(h_i)}|h_i)(h_i|\tau|h_{i-1})$ are homotopic walks of $E_{\gamma}$, it can be shown that the values of $F(r_{i-1})$ and $F(r_i)$ given by above divisions are equal.

\medskip
{\it Case 4: $\widetilde{w}$ is obtained from $w$ by replacing a subwalk with another, using relation $(h4)$ in Definition \ref{homotopy of walks}.}

Let $w=u(g\cdot e_2|\tau|g\cdot e_1)(g\cdot e_1|g|e_1)v$ and $\widetilde{w}=u(g\cdot e_2|g|e_2)(e_2|\tau|e_1)v$ for some walks $u$, $v$ of $E$ with $t(u)$, $s(v)\in A$ and $L(e_1)=L(e_2)$. As in Step 3, we may replace $u$, $v$ by divisible walks $u'$, $v'$ respectively. Since $E=\bigcup_{\alpha\in I}E_{\alpha}$ as f-BCs, there exists a sequence of elements $h_0=e_1$, $h_1$, $\cdots$, $h_{n-1}$, $h_n=e_2$ of $E$ such that for each $1\leq i\leq n$, there exists some $\gamma_{i}\in I$ with $h_{i-1}$, $h_i\in E_{\gamma_{i}}$ and $L_{\gamma_{i}}(h_{i-1})=L_{\gamma_{i}}(h_{i})$. Then $g\cdot h_0=g\cdot e_1$, $g\cdot h_1$, $\cdots$, $g\cdot h_{n-1}$, $g\cdot h_n=g\cdot e_2$ becomes a connecting sequence from $g\cdot e_1$ to $g\cdot e_2$. For each $0\leq i\leq n$, define a divisible walk $r_i$ of $E$ as follows:
\begin{equation*}
r_i=u'(g\cdot h_n|\tau|g\cdot h_{n-1})\cdots(g\cdot h_{i+1}|\tau|g\cdot h_i)(g\cdot h_i|g|h_i)(h_i|\tau|h_{i-1})\cdots(h_1|\tau|h_{0})v'.
\end{equation*}
Since $F(w)=F(r_0)$ and $F(\widetilde{w})=F(r_n)$, we need to show that $F(r_{i-1})=F(r_i)$ for each $1\leq i\leq n$. Choose a division $(u_1,\cdots,u_k,E_{\alpha_1},\cdots,E_{\alpha_k})$ of $u'(g\cdot h_{n}|\tau|g\cdot h_{n-1})\cdots(g\cdot h_{i+1}|\tau|g\cdot h_i)$ and a division $(v_1,\cdots,v_m,E_{\beta_1},\cdots,E_{\beta_m})$ of $(h_{i-1}|\tau|h_{i-2})\cdots(h_1|\tau|h_0)v'$. Then
\begin{equation*} (v_1,\cdots,v_m,(g\cdot h_i|\tau|g\cdot h_{i-1})(g\cdot h_{i-1}|g|h_{i-1}),u_1,\cdots,u_k,E_{\beta_1},\cdots,
E_{\beta_m},E_{\gamma_{i}},E_{\alpha_1},\cdots,E_{\alpha_k})
\end{equation*}
is a division of $r_{i-1}$ and
\begin{equation*}(v_1,\cdots,v_m,(g\cdot h_i|g|h_i)(h_i|\tau|h_{i-1}),u_1,\cdots,u_k,E_{\beta_1},\cdots,
E_{\beta_m},E_{\gamma_{i}},E_{\alpha_1},\cdots,E_{\alpha_k})
\end{equation*}
is a division of $r_i$. Since $(g\cdot h_i|\tau|g\cdot h_{i-1})(g\cdot h_{i-1}|g|h_{i-1})$ and $(g\cdot h_i|g|h_i)(h_i|\tau|h_{i-1})$ are homotopic walks of $E_{\gamma_{i}}$, it can be shown that the values of $F(r_{i-1})$ and $F(r_i)$ defined by above divisions are equal.

By Step 5, $F$ induce maps $\Pi(E,A)(a_1,a_2)\rightarrow \mathscr{G}(F(a_1),F(a_2))$ for all $a_1$, $a_2\in A$. We also denote these maps by $F$.

\medskip
{\it Step 6: To show that $F:\Pi(E,A)\rightarrow \mathscr{G}$ becomes a functor.}

For $\overline{w}\in\Pi(E,A)(a_1,a_2)$ and $\overline{w'}\in\Pi(E,A)(a_0,a_1)$, assume that $w$ and $w'$ are divisible. Choose a division $(w_1,\cdots,w_n,E_{\alpha_1},\cdots,E_{\alpha_n})$ of $w$ and a division \\ $(w'_1,\cdots,w'_m,E_{\beta_1},\cdots,E_{\beta_m})$ of $w'$. Then $(w'_1,\cdots,w'_m,w_1,\cdots,w_n,E_{\beta_1},\cdots,E_{\beta_m},E_{\alpha_1},\cdots,E_{\alpha_n})$ becomes a division of $ww'$. Choose a walk $u_i$ of $E_{\alpha_i}\cap E_{\alpha_{i+1}}$ from some $b_i\in A$ to $t(w_i)$ for each $1\leq i\leq n-1$, and choose a walk $u'_i$ of $E_{\beta_i}\cap E_{\beta_{i+1}}$ from some $c_i\in A$ to $t(w'_i)$ for each $1\leq i\leq m-1$. Moreover, choose a closed walk $u$ of $E_{\gamma}=E_{\alpha_1}\cap E_{\beta_m}$ at $a_1$ (For example, let $u$ be the trivial walk of $E_{\gamma}$ at $a_1$). Then
\begin{multline*}
F(\overline{ww'})=F_{\alpha_n}(\overline{w_n u_{n-1}})F_{\alpha_{n-1}}(\overline{u_{n-1}^{-1}w_{n-1}u_{n-2}})\cdots F_{\alpha_{2}}(\overline{u_{2}^{-1}w_{2}u_{1}})F_{\alpha_{1}}(\overline{u_{1}^{-1}w_{1}u}) \\
F_{\beta_m}(\overline{u^{-1}w'_m u'_{m-1}})F_{\beta_{m-1}}(\overline{(u'_{m-1})^{-1}w'_{m-1}u'_{m-2}})\cdots F_{\beta_{2}}(\overline{(u'_{2})^{-1}w'_{2}u'_{1}})F_{\beta_{1}}(\overline{(u'_{1})^{-1}w'_{1}}) \\
=F_{\alpha_n}(\overline{w_n u_{n-1}})F_{\alpha_{n-1}}(\overline{u_{n-1}^{-1}w_{n-1}u_{n-2}})\cdots F_{\alpha_{2}}(\overline{u_{2}^{-1}w_{2}u_{1}})F_{\alpha_{1}}(\overline{u_{1}^{-1}w_{1}})F_{\alpha_{1}}(\overline{u}) \\
F_{\beta_m}(\overline{u^{-1}})F_{\beta_m}(\overline{w'_m u'_{m-1}})F_{\beta_{m-1}}(\overline{(u'_{m-1})^{-1}w'_{m-1}u'_{m-2}})\cdots F_{\beta_{2}}(\overline{(u'_{2})^{-1}w'_{2}u'_{1}})F_{\beta_{1}}(\overline{(u'_{1})^{-1}w'_{1}}) \\
=F_{\alpha_n}(\overline{w_n u_{n-1}})F_{\alpha_{n-1}}(\overline{u_{n-1}^{-1}w_{n-1}u_{n-2}})\cdots F_{\alpha_{2}}(\overline{u_{2}^{-1}w_{2}u_{1}})F_{\alpha_{1}}(\overline{u_{1}^{-1}w_{1}})F_{\gamma}(\overline{u}) \\
F_{\gamma}(\overline{u^{-1}})F_{\beta_m}(\overline{w'_m u'_{m-1}})F_{\beta_{m-1}}(\overline{(u'_{m-1})^{-1}w'_{m-1}u'_{m-2}})\cdots F_{\beta_{2}}(\overline{(u'_{2})^{-1}w'_{2}u'_{1}})F_{\beta_{1}}(\overline{(u'_{1})^{-1}w'_{1}}) \\
=F_{\alpha_n}(\overline{w_n u_{n-1}})F_{\alpha_{n-1}}(\overline{u_{n-1}^{-1}w_{n-1}u_{n-2}})\cdots F_{\alpha_{2}}(\overline{u_{2}^{-1}w_{2}u_{1}})F_{\alpha_{1}}(\overline{u_{1}^{-1}w_{1}}) \\
F_{\beta_m}(\overline{w'_m u'_{m-1}})F_{\beta_{m-1}}(\overline{(u'_{m-1})^{-1}w'_{m-1}u'_{m-2}})\cdots F_{\beta_{2}}(\overline{(u'_{2})^{-1}w'_{2}u'_{1}})F_{\beta_{1}}(\overline{(u'_{1})^{-1}w'_{1}}) \\
=F(\overline{w})F(\overline{w'}).
\end{multline*}

For each $a\in A$, choose some $\alpha\in I$. Then $((a||a),E_{\alpha})$ is a division of $(a||a)$. We have $F(1_{a})=F(\overline{(a||a)})=F_{\alpha}(\overline{(a||a)})=F_{\alpha}(1_a)=1_{F_{\alpha}(a)}=1_{F(a)}$.

\medskip
{\it Step 7: To show that $F:\Pi(E,A)\rightarrow \mathscr{G}$ is the unique functor which satisfies $F_{\alpha}=Fj_{\alpha *}$ for all $\alpha\in I$, where $j_{\alpha}:E_{\alpha}\rightarrow E$ is the inclusion map.}

By the definition of $F$, we have $F(a)=F_{\alpha}(a)$ for all $\alpha\in I$ and $a\in A$. For each $\overline{w}\in\Pi(E_{\alpha},A)(a_1,a_2)$, $(w,E_{\alpha})$ is a division of $w$, and $F(\overline{w})=F_{\alpha}(\overline{w})$. Therefore $F_{\alpha}=Fj_{\alpha *}$. Let $F':\Pi(E,A)\rightarrow \mathscr{G}$ be another functor which satisfies $F_{\alpha}=F'j_{\alpha *}$ for all $\alpha\in I$. For each $a\in A$, choose some $\alpha\in I$, then $F'(a)=F_{\alpha}(a)=F(a)$. For each morphism $\overline{w}\in\Pi(E,A)(a_1,a_2)$, assume that $w$ is divisible (otherwise we may replace $w$ by a divisible walk $w'$ of $E$ which is homotopic to $w$, see Step 3), choose a division $(w_1,\cdots,w_n,E_{\alpha_1},\cdots,E_{\alpha_n})$ of $w$ and choose a walk $u_i$ of $E_{\alpha_i}\cap E_{\alpha_{i+1}}$ from some $b_i\in A$ to $t(w_i)$ for each $1\leq i\leq n-1$. Then
\begin{multline*}
F'(\overline{w})=F'(\overline{w_n u_{n-1}}\overline{u_{n-1}^{-1}w_{n-1}u_{n-2}}\cdots\overline{u_{2}^{-1}w_{2}u_{1}}\overline{u_{1}^{-1}w_{1}}) \\
=F'(\overline{w_n u_{n-1}})F'(\overline{u_{n-1}^{-1}w_{n-1}u_{n-2}})\cdots F'(\overline{u_{2}^{-1}w_{2}u_{1}})F'(\overline{u_{1}^{-1}w_{1}}) \\ =F_{\alpha_n}(\overline{w_n u_{n-1}})F_{\alpha_{n-1}}(\overline{u_{n-1}^{-1}w_{n-1}u_{n-2}})\cdots F_{\alpha_2}(\overline{u_{2}^{-1}w_{2}u_{1}})F_{\alpha_1}(\overline{u_{1}^{-1}w_{1}})=F(\overline{w}).
\end{multline*}

Therefore $F'=F$ and $F$ is unique.
\end{proof}

The following lemma shows that the fundamental group of an f-BC $E$ is isomorphic to the fundamental group of $E'$, where $E'$ is an f-BC obtained from $E$ by deleting some $1$-gons.

\begin{Lem}\label{isomorphism of fundamental groupoids}
Let $E=(E,P,L,d)$ be an f-BC. Let $D$ be a subset of $E$ which satisfies
\begin{itemize}
\item [$(1)$] for each $e\in E$, $e\in D$ if and only if $g^{d(e)}(e)\in D$ (here we write the action of $g^n$ on $h$ as $g^n(h)$);
\item [$(2)$] $P(e)=\{e\}$ for each $e\in D$.
\end{itemize}
Let $E'=E-D$ and assume that $E'\neq\emptyset$. Define an f-BC structure $(E',P',L',d')$ on $E'$ as follows: The action of $G=\langle g\rangle$ on $E'$ is given by
\begin{equation*}
g\cdot h= \begin{cases}
g(h), & \text{ if } g(h)\in E'; \\
g^{N}(h), & \text{ if } g(h)\notin E', \\
& \text{ where } N \text{ is the minimal positive integer such that } g^{N}(h)\in E'.
\end{cases}
\end{equation*}
Define $P'(h)=P(h)$ and $L'(h)=L(h)$ for every $h\in E'$. Define the degree function $d'$ by $d'(h)=d(h)-|\{i\mid 1\leq i\leq d(h)-1 $ and $g^{i}(h)\notin E'\}|$ for every $h\in E'$. Then $E'$ is an f-BC, and the groupoids $\Pi(E,E')$ and $\Pi(E',E')$ are isomorphic.
\end{Lem}

\medskip
\begin{proof}
Note that $g^{d'(h)}\cdot h=g^{d(h)}(h)$ for each $h\in E'$, it is straightforward to show that $E'$ is satisfies $(f1)$-$(f5)$ in Definition \ref{f-BC}. Suppose $(f6)$ does not hold for $E'$, then there exist $e$, $h\in E'$ with $d'(e)<d'(h)$ and $L'(g^{i}\cdot e)=L'(g^{i}\cdot h)$ for $0\leq i\leq d'(e)-1$. Since $L(e)=L(h)$, $P(g(e))=P(g(h))$. If $g(e)\notin E'$, since $g(e)\in D$, $P(g(e))=\{g(e)\}$. Therefore $g(e)=g(h)$ and $e=h$, a contradiction. Therefore both $g(e)$ and $g(h)$ belong to $E'$, and $g\cdot e=g(e)$, $g\cdot h=g(h)$. By induction, for each $0\leq i\leq d'(e)$, $g^{i}(e),g^{i}(h)\in E'$, and $g^{i}\cdot e=g^{i}(e)$ (resp. $g^{i}\cdot h=g^{i}(h)$). Therefore $d(e)=d'(e)<d'(h)\leq d(h)$ and $L(g^i(e))=L(g^i(h))$ for $0\leq i\leq d(e)-1$, which contradicts the fact that $E$ satisfies $(f6)$. Then $(f6)$ holds for $E'$ and $E'$ is an f-BC.

To show that $\Pi(E,E')$ and $\Pi(E',E')$ are isomorphic, first we define a functor $F:\Pi(E',E')\rightarrow \Pi(E,E')$ as follows: For each walk $w'$ of $E'$, define a walk $f(w')$ of $E$ by replacing each subwalk $(g\cdot h|g|h)$ of $w'$ with $g(h)\notin E'$ by $(g^{N}(h)|g^{N}|h)$, where $N$ is the minimal positive integer such that $g^{N}(h)\in E'$, and replacing each subwalk $(g^{-1}\cdot h|g|h)$ of $w'$ with $g^{-1}(h)\notin E'$ by $(g^{-N'}(h)|g^{-N'}|h)$, where $N'$ is the minimal positive integer such that $g^{-N'}(h)\in E'$. By definition, $f$ preserves the composition of walks. For each $h\in E'$, $f((g^{d'(h)}\cdot h|g^{d'(h)}|h))=(g^{d(h)}(h)|g^{d(h)}|h)$, therefore $f$ preserves relation $(h3)$ in Definition \ref{homotopy of walks}. For $e$, $h\in E'$ such that $L'(e)=L'(h)$, if $g(e)\notin E'$ or $g(h)\notin E'$, then $P(g(e))=\{g(e)\}$ or $P(g(h))=\{g(h)\}$. Since $P(g(e))=P(g(h))$, $g(e)=g(h)$ and $e=h$, so $f((g\cdot h|\tau|g\cdot e)(g\cdot e|g|e))$ is homotopic to $f((g\cdot h|g|h)(h|\tau|e))$. If $g(e)$, $g(h)\in E'$, then
$$f((g\cdot h|\tau|g\cdot e)(g\cdot e|g|e))=(g(h)|\tau|g(e))(g(e)|g|e)\sim(g(h)|g|h)(h|\tau|e)=f((g\cdot h|g|h)(h|\tau|e)).$$
Therefore $f$ preserves relation $(h4)$ in Definition \ref{homotopy of walks}. It is straightforward to show that $f$ preserves relations $(h1)$ and $(h2)$ in Definition \ref{homotopy of walks}, and therefore $f$ sends homotopic walks of $E'$ to homotopic walks of $E$. Define $F(h)=h$ for each $h\in E'$ and $F(\overline{w'})=\overline{f(w')}$ for each walk $w'$ of $E'$. Then $F$ is a functor from $\Pi(E',E')$ to $\Pi(E,E')$.

To show that $F$ is an isomorphism, we need to define a functor $G:\Pi(E,E')\rightarrow \Pi(E',E')$. For each walk $w$ of $E$ with $s(w)$, $t(w)\in E'$, $w$ can be written uniquely as the form
\begin{equation}\label{decompose-walk}w_m v_m w_{m-1} v_{m-1}\cdots w_2 v_2 w_1 v_1 w_0,\end{equation}
where $w_i$ are walks of $E'$ and
$$v_i=e_{i,n_i}\frac{\delta_{i,n_i}}{}e_{i,n_{i}-1}\frac{\delta_{i,n_{i}-1}}{}\cdots\frac{\delta_{i2}}{}e_{i1}\frac{\delta_{i1}}{}e_{i0}$$
are walks of $E$ such that $n_i\geq 2$, $e_{i,0}$, $e_{i,n_i}\in E'$ and $e_{ij}\notin E'$ for $1\leq j\leq n_{i}-1$. Define a number $N(v_i)$ for each $1\leq i\leq m$: $$N(v_i):=|\{j\mid 1\leq j\leq n_i \text{ with } \delta_{ij}=g\}|-|\{j\mid 1\leq j\leq n_i \text{ with } \delta_{ij}=g^{-1}\}|.$$
Since $P(e)=\{e\}$ for each $e\in E-E'$, $e_{i,n_i}=g^{N(v_i)}(e_{i0})$. When $N(v_i)>0$, since $e_{i0}$, $e_{i,n_i}\in E'$ and $g^{i}(e_{i0})\notin E'$ for $1\leq i\leq N(v_i)-1$, $e_{i,n_i}=g^{N(v_i)}(e_{i0})=g\cdot e_{i0}$. When $N(v_i)<0$, similarly $e_{i,n_i}=g^{N(v_i)}(e_{i0})=g^{-1}\cdot e_{i0}$. Define a walk $v'_i$ of $E'$ for each $1\leq i\leq m$ as follows:
\begin{equation*}
v'_i= \begin{cases}
(g\cdot e_{i0}|g|e_{i0}), &\text{ if } N(v_i)>0; \\
(e_{i,0}||e_{i0}), &\text{ if } N(v_i)=0; \\
(g^{-1}\cdot e_{i,0}|g^{-1}|e_{i0}), &\text{ if } N(v_i)<0.
\end{cases}
\end{equation*}
Define a walk $\phi(w)$ of $E'$ by the formula $\phi(w)=w_m v'_m w_{m-1} v'_{m-1}\cdots w_2 v'_2 w_1 v'_1 w_0$. It can be shown that $\phi$ preserves the composition of walks of $E$ whose sources and terminals belong to $E'$.

We need to show that $\phi$ preserves homotopic relations of walks. If $t_1$, $t_2$ are two walks of $E$ whose sources and terminals belong to $E'$, such that $t_2$ is obtained from $t_1$ by replacing a subwalk with another, using relation $(h1)$ or $(h2)$ in Definition \ref{homotopy of walks}, then it is straightforward to show that $\phi(t_1)$ is homotopic to $\phi(t_2)$.

If $t_1$, $t_2$ are two walks of $E$ whose sources and terminals belong to $E'$, such that $t_2$ is obtained from $t_1$ by replacing a subwalk with another, using relation $(h3)$ in Definition \ref{homotopy of walks}, then we may assume that $t_1=u(g^{d(h)}(h)|\tau|g^{d(e)}(e))(g^{d(e)}(e)|g^{d(e)}|e)v$ and $t_2=u(g^{d(h)}(h)|g^{d(h)}|h)(h|\tau|e)v$ with $s(v),t(u)\in E'$, where $P(e)=P(h)$. If $e\notin E'$ or $h\notin E'$, since $P(e)=P(h)$, $e=h$. Then $t_1=u(g^{d(e)}(e)|\tau|g^{d(e)}(e))(g^{d(e)}(e)|g^{d(e)}|e)v$ and $t_2=u(g^{d(e)}(e)|g^{d(e)}|e)(e|\tau|e)v$. Write $t_1=w_m v_m w_{m-1} v_{m-1}\cdots w_2 v_2 w_1 v_1 w_0$ as in formula (\ref{decompose-walk}). Since $g^{d(e)}(e)\notin E'$, the subwalk \\ $(g^{d(e)}(e)|\tau|g^{d(e)}(e))$ of $t_1$ is a subwalk of some $v_i$. Let $v_i=p(g^{d(e)}(e)|\tau|g^{d(e)}(e))q$ and let $\widetilde{v_i}=p q$. Then $$u(g^{d(e)}(e)|g^{d(e)}|e)v=w_m v_m w_{m-1} v_{m-1}\cdots w_i \widetilde{v_i} w_{i-1}\cdots w_2 v_2 w_1 v_1 w_0$$
is the decomposition of $u(g^{d(e)}(e)|g^{d(e)}|e)v$ as in formula (\ref{decompose-walk}). We have $v'_i=\widetilde{v_i}'$ and
\begin{multline*}\phi(t_1)=\phi(u(g^{d(e)}(e)|\tau|g^{d(e)}(e))(g^{d(e)}(e)|g^{d(e)}|e)v)=w_m v'_m w_{m-1} v'_{m-1}\cdots w_2 v'_2 w_1 v'_1 w_0 = \\ w_m v'_m w_{m-1} v'_{m-1}\cdots w_i \widetilde{v_i}' w_{i-1}\cdots w_2 v'_2 w_1 v'_1 w_0=\phi(u(g^{d(e)}(e)|g^{d(e)}|e)v).\end{multline*}
Similarly, we also have $\phi(t_2)=\phi(u(g^{d(e)}(e)|g^{d(e)}|e)v)$.

If $e\in E'$, then so does $h$. Therefore the sources and terminals of walks $u$, $v$, \\ $(g^{d(h)}(h)|\tau|g^{d(e)}(e))(g^{d(e)}(e)|g^{d(e)}|e)$, $(g^{d(h)}(h)|g^{d(h)}|h)(h|\tau|e)$ belong to $E'$. Since $\phi$ preserves the composition of walks whose sources and terminals belong to $E'$, $$\phi(t_1)=\phi(u)\phi((g^{d(h)}(h)|\tau|g^{d(e)}(e))(g^{d(e)}(e)|g^{d(e)}|e))\phi(v)$$ and  $$\phi(t_2)=\phi(u)\phi((g^{d(h)}(h)|g^{d(h)}|h)(h|\tau|e))\phi(v).$$ Since $$\phi((g^{d(h)}(h)|\tau|g^{d(e)}(e))(g^{d(e)}(e)|g^{d(e)}|e))=(g^{d'(h)}\cdot h|\tau|g^{d'(e)}\cdot e)(g^{d'(e)}\cdot e|g^{d'(e)}|e)$$
and
$$\phi((g^{d(h)}(h)|g^{d(h)}|h)(h|\tau|e))=(g^{d'(h)}\cdot h|g^{d'(h)}|h)(h|\tau|e),$$
$\phi(t_1)$ is homotopic to $\phi(t_2)$.

If $t_1$, $t_2$ are two walks of $E$ whose sources and terminals belong to $E'$, such that $t_2$ is obtained from $t_1$ by replacing a subwalk with another, using relation $(h4)$ in Definition \ref{homotopy of walks}, then we may assume that $t_1=u(g(h)|\tau|g(e))(g(e)|g|e)v$ and $t_2=u(g(h)|g|h)(h|\tau|e)v$ with $s(v),t(u)\in E'$, where $L(e)=L(h)$. If $g(e)\notin E'$ or $g(h)\notin E'$, since $P(g(e))=P(g(h))$, $g(e)=g(h)$ and $e=h$. Therefore $\phi(t_1)=\phi(u(g(e)|g|e)v)\sim\phi(t_2)$. If $e\notin E'$ or $h\notin E'$, similarly we have $e=h$ and $\phi(t_1)\sim\phi(t_2)$. If $e$, $h$, $g(e)$, $g(h)\in E'$, then
$$\phi(t_1)=\phi(u)(g\cdot h|\tau|g\cdot e)(g\cdot e|g|e)\phi(v)\sim\phi(u)(g\cdot h|g|h)(h|\tau|e)\phi(v)=\phi(t_2).$$

Let $G:\Pi(E,E')\rightarrow \Pi(E',E')$ be the functor sending each object $h$ of $\Pi(E,E')$ to itself and sending each morphism $\overline{w}$ of $\Pi(E,E')$ to $\overline{\phi(w)}$, where $w$ is a walk of $E$ with source and terminal in $E'$. It can be shown that $GF=id_{\Pi(E',E')}$ and $FG=id_{\Pi(E,E')}$.
\end{proof}

Next we will reduce the calculation of the fundamental group of a Brauer configuration $E$ to the calculation of the fundamental group of a Brauer graph $C$. To show that the fundamental groups $\Pi(E)$ and $\Pi(C)$ are isomorphic, we will show that they are the direct limits of two isomorphic direct systems by using Proposition \ref{Van-Kampen}.

Let $E=(E,P,L,d)$ be a Brauer configuration, where we denote the action of $g^n$ on $e\in E$ by $g^n(e)$. For each polygon $P(e)$ of $E$ with $|P(e)|\geq 3$, label the elements of $P(e)$ by $e_1,\cdots,e_n$, and define a set $\overline{P(e)}=\{e'_2,\cdots,e'_{n-1}\}$. Define a Brauer graph $C=(C,\widetilde{P},\widetilde{L},\widetilde{d})$ as follows:
\begin{equation*} C=E\sqcup\bigsqcup_{|P(e)|\geq 3}\overline{P(e)}. \end{equation*}
The action of $G$ on $C$ is given by
\begin{equation*}
g\cdot h= \begin{cases}
g(h), &\text{ if } h\in E \text{ and } h' \text{ is not defined}; \\
h', &\text{ if } h\in E \text{ and } h' \text{ is defined}; \\
g(e), &\text{ if } h=e' \text{ for some } e\in E.
\end{cases}
\end{equation*}
For $h\in C$, if $h\in E$ and $|P(h)|=2$, define $\widetilde{P}(h)=P(h)$. For each polygon $P(e)=\{e_1,\cdots,e_n\}$ of $E$ with $|P(e)|\geq 3$, define the partition $\widetilde{P}$ of the subset $P(e)\sqcup \overline{P(e)}$ of $C$ by $\{e_1,e_2\}\sqcup\{e'_2,e_3\}\sqcup\{e'_3,e_4\}\sqcup\cdots\sqcup\{e'_{n-2},e_{n-1}\}\sqcup\{e'_{n-1},e_n\}$. Define $\widetilde{L}(h)=\{h\}$ for each $h\in C$. Define
$$\widetilde{d}(h)=\frac{d(h)\cdot\mid\{G\text{-orbit of } h\text{ in } C\}\mid}{\mid\{G\text{-orbit of } h\text{ in } E\}\mid}$$
for each $h\in E$, and define $\widetilde{d}(h)=\widetilde{d}(e)$ for $h=e'$ with $e\in E$.

For each polygon $\widetilde{P}(e)$ of $C$, define a sub-f-BC $C_{\widetilde{P}(e)}=(C_{\widetilde{P}(e)},\widetilde{P}_{\widetilde{P}(e)},\widetilde{L}_{\widetilde{P}(e)},\widetilde{d}_{\widetilde{P}(e)})$ of $C$ as follows: $C_{\widetilde{P}(e)}=C$ as $G$-sets and $\widetilde{d}_{\widetilde{P}(e)}=\widetilde{d}$. The partition $\widetilde{P}_{\widetilde{P}(e)}$ is given by
\begin{equation*}
\widetilde{P}_{\widetilde{P}(e)}(h)= \begin{cases}
\widetilde{P}(e), &\text{ if } h\in \widetilde{P}(e); \\
\{h\}, &\text{ otherwise},
\end{cases}
\end{equation*}
and the partition $\widetilde{L}_{\widetilde{P}(e)}$ is given by $\widetilde{L}_{\widetilde{P}(e)}(h)=\{h\}$ for each $h\in C_{\widetilde{P}(e)}$.

For each polygon $\widetilde{P}(e)$ of $C$, we also define a sub-f-BC $E_{\widetilde{P}(e)}=(E_{\widetilde{P}(e)},P_{\widetilde{P}(e)},L_{\widetilde{P}(e)},d_{\widetilde{P}(e)})$ of $E$: $E_{\widetilde{P}(e)}=E$ as $G$-sets and $d_{\widetilde{P}(e)}=d$. The partition $L_{\widetilde{P}(e)}$ is given by $L_{\widetilde{P}(e)}(h)=\{h\}$ for each $h\in E_{\widetilde{P}(e)}$. The partition $P_{\widetilde{P}(e)}$ is defined as follows: If $e\in E$ and $|P(e)|=2$, then $\widetilde{P}(e)=P(e)$, and the partition $P_{\widetilde{P}(e)}$ of $E_{\widetilde{P}(e)}$ is given by
\begin{equation*}
P_{\widetilde{P}(e)}(h)= \begin{cases}
P(e), &\text{ if } h\in P(e); \\
\{h\}, &\text{ otherwise}.
\end{cases}
\end{equation*}
Otherwise, then there exists some $P(e_1)=\{e_1,\cdots,e_n\}\subseteq E$ with $n\geq 3$ such that $\widetilde{P}(e)$ is equal to one of the following sets: $\{e_1,e_2\}$, $\{e'_2,e_3\}$, $\{e'_3,e_4\}$, $\cdots$, $\{e'_{n-1},e_n\}$. If $\widetilde{P}(e)=\{e_1,e_2\}$, define
\begin{equation*}
P_{\widetilde{P}(e)}(h)= \begin{cases}
\{e_1,e_2\}, &\text{ if } h\in \{e_1,e_2\}; \\
\{h\}, &\text{ otherwise}.
\end{cases}
\end{equation*}
If $\widetilde{P}(e)=\{e'_i,e_{i+1}\}$ with $2\leq i\leq n-1$, define
\begin{equation*}
P_{\widetilde{P}(e)}(h)= \begin{cases}
\{e_i,e_{i+1}\}, &\text{ if } h\in \{e_i,e_{i+1}\}; \\
\{h\}, &\text{ otherwise}.
\end{cases}
\end{equation*}

Moreover, define a sub-f-BC $C'=(C',\widetilde{P'},\widetilde{L'},\widetilde{d'})$ of $C$ (resp. a sub-f-BC $E'=(E',P',L',d')$ of $E$) as follows: $C'=C$ (resp. $E'=E$) as $G$-sets and $\widetilde{d'}=\widetilde{d}$ (resp. $d'=d$). The partition $\widetilde{P'}$ and the partition $\widetilde{L'}$ of $C'$ (resp. the partition $P'$ and the partition $L'$ of $E'$) are given by $\widetilde{P'}(e)=\widetilde{L'}(e)=\{e\}$ for each $e\in C$ (resp. $P'(e)=L'(e)=\{e\}$ for each $e\in E$). For each polygon $\widetilde{P}(e)$ of $C$, let $f_{\widetilde{P}(e)}:E'\rightarrow E_{\widetilde{P}(e)}$ (resp. $\widetilde{f}_{\widetilde{P}(e)}:C'\rightarrow C_{\widetilde{P}(e)}$) be the inclusion morphism. Then $f_{\widetilde{P}(e)}$ (resp. $\widetilde{f}_{\widetilde{P}(e)}$) induces a functor $F_{\widetilde{P}(e)}:\Pi(E',E)\rightarrow\Pi(E_{\widetilde{P}(e)},E)$ (resp. $\widetilde{F}_{\widetilde{P}(e)}:\Pi(C',E)\rightarrow\Pi(C_{\widetilde{P}(e)},E)$).

\begin{Prop}\label{iso. of direct systems}
The direct systems $\{F_{\widetilde{P}(e)}:\Pi(E',E)\rightarrow\Pi(E_{\widetilde{P}(e)},E)\}_{\widetilde{P}(e)\subseteq C}$ and $\{\widetilde{F}_{\widetilde{P}(e)}:\Pi(C',E)\rightarrow\Pi(C_{\widetilde{P}(e)},E)\}_{\widetilde{P}(e)\subseteq C}$ are isomorphic.
\end{Prop}

\begin{proof}
Since $E'$ is obtained from $C'$ by removing $D$, where $D=\bigsqcup_{|P(e)|\geq 3}\overline{P(e)}$ is a subset of $C'$ which satisfies conditions $(1)$ and $(2)$ in Lemma \ref{isomorphism of fundamental groupoids}, there is an isomorphism $\Phi':\Pi(E',E)\rightarrow\Pi(C',E)$ of groupoids.

For each $\widetilde{P}(e)\subseteq C$, if $e\in E$ and $|P(e)|< 3$, then $\widetilde{P}(e)=P(e)$, and $E_{\widetilde{P}(e)}$ is obtained from $C_{\widetilde{P}(e)}$ by removing $D$, where $D=\bigsqcup_{|P(h)|\geq 3}\overline{P(h)}$ is a subset of $C_{\widetilde{P}(e)}$ which satisfies conditions $(1)$ and $(2)$ in Lemma \ref{isomorphism of fundamental groupoids}. Then there is an isomorphism $$\Phi_{\widetilde{P}(e)}:\Pi(E_{\widetilde{P}(e)},E)\rightarrow\Pi(C_{\widetilde{P}(e)},E)$$
of groupoids. Otherwise, there exists some $P(e_1)=\{e_1,\cdots,e_n\}\subseteq E$ with $n\geq 3$ such that $\widetilde{P}(e)$ is equal to one of the following sets:
$$\{e_1,e_2\}, \{e'_2,e_3\}, \{e'_3,e_4\}, \cdots, \{e'_{n-1},e_n\}.$$
If $\widetilde{P}(e)=\{e_1,e_2\}$, similarly, $E_{\widetilde{P}(e)}$ is obtained from $C_{\widetilde{P}(e)}$ by removing $D$, where $D=\bigsqcup_{|P(h)|\geq 3}\overline{P(h)}$ is a subset of $C_{\widetilde{P}(e)}$ which satisfies conditions $(1)$ and $(2)$ in Lemma \ref{isomorphism of fundamental groupoids}. Then there is an isomorphism
$$\Phi_{\widetilde{P}(e)}:\Pi(E_{\widetilde{P}(e)},E)\rightarrow\Pi(C_{\widetilde{P}(e)},E)$$
of groupoids. If $\widetilde{P}(e)=\{e'_i,e_{i+1}\}$ for some $2\leq i\leq n-1$, $E_{\widetilde{P}(e)}$ is isomorphic to the f-BC $H_{\widetilde{P}(e)}$, where $H_{\widetilde{P}(e)}$ is obtained from $C_{\widetilde{P}(e)}$ by removing $D$, where
$$D=(\bigsqcup_{|P(h)|\geq 3}\overline{P(h)}-\{e'_i\})\sqcup\{e_i\}$$
is a subset of $C_{\widetilde{P}(e)}$ which satisfies conditions $(1)$ and $(2)$ in Lemma \ref{isomorphism of fundamental groupoids}. The isomorphism $f:E_{\widetilde{P}(e)}\rightarrow H_{\widetilde{P}(e)}$ is given by $f(e_i)=e'_i$ and $f(x)=x$ for $x\in E-\{e_i\}$. Let $\Phi_{\widetilde{P}(e)}$ be the composition of isomorphisms $$\Pi(E_{\widetilde{P}(e)},E)\xrightarrow{f_{*}}\Pi(H_{\widetilde{P}(e)},H)\xrightarrow{\psi}\Pi(C_{\widetilde{P}(e)},H)\xrightarrow{\phi}\Pi(C_{\widetilde{P}(e)},E),$$
where $H$ denotes the underlying set of $H_{\widetilde{P}(e)}$, $\psi$ is given by Lemma \ref{isomorphism of fundamental groupoids}, and $\phi$ is defined as follows: for each $h\in H$,
\begin{equation*}
\phi(h)= \begin{cases}
h, &\text{ if } h\neq e'_i; \\
e_i, &\text{ if } h=e'_i,
\end{cases}
\end{equation*}
for each walk $w$ of $C_{\widetilde{P}(e)}$ with $s(w)$, $t(w)\in H$,
\begin{equation*}
\phi(\overline{w})= \begin{cases}
\overline{w}, &\text{ if } s(w)\neq e'_i \text{ and } t(w)\neq e'_i; \\
\overline{w(e'_i|g|e_i)}, &\text{ if } s(w)=e'_i \text{ and } t(w)\neq e'_i; \\
\overline{(e_i|g^{-1}|e'_i)w}, &\text{ if } s(w)\neq e'_i \text{ and } t(w)=e'_i; \\
\overline{(e_i|g^{-1}|e'_i)w(e'_i|g|e_i)}, &\text{ if } s(w)=e'_i \text{ and } t(w)=e'_i.
\end{cases}
\end{equation*}

To finish the proof, we need to show that for each $\widetilde{P}(e)\subseteq C$, the diagram
$$ \xymatrix@R=1.3pc{
\Pi(E',E)\ar[r]^(0.45){F_{\widetilde{P}(e)}}\ar[d]_{\Phi'} & \Pi(E_{\widetilde{P}(e)},E)\ar[d]^{\Phi_{\widetilde{P}(e)}} \\
\Pi(C',E)\ar[r]^(0.45){\widetilde{F}_{\widetilde{P}(e)}} & \Pi(C_{\widetilde{P}(e)},E) \\
} $$
is commutative. For the case $e\in E$ and $|P(e)|=2$, or the case $\widetilde{P}(e)=\{e_1,e_2\}$ with $P(e_1)=\{e_1,\cdots,e_n\}\subseteq E$ and $n\geq 3$, since both functors $\Phi'$ and $\Phi_{\widetilde{P}(e)}$ are defined as in the proof of Lemma \ref{isomorphism of fundamental groupoids}, it is straightforward to show the diagram above is commutative. For the case $\widetilde{P}(e)=\{e'_i,e_{i+1}\}$ with $2\leq i\leq n-1$, where $P(e_i)=\{e_1,\cdots,e_n\}\subseteq E$ with $n\geq 3$, it can be shown that for each $x\in E$, $$\widetilde{F}_{\widetilde{P}(e)}\Phi'(x)=x=\Phi_{\widetilde{P}(e)} F_{\widetilde{P}(e)}(x)$$
and
$$\widetilde{F}_{\widetilde{P}(e)}\Phi'(\overline{(x|g^{N}|x)})=\overline{(x|g^{N'}(x)|x)}=\Phi_{\widetilde{P}(e)}F_{\widetilde{P}(e)}
(\overline{(x|g^{N}|x)}),$$
where $N$ and $N'$ denote the cardinal of the $G$-orbit of $x$ in $E'$ and the cardinal of the $G$-orbit of $x$ in $C_{\widetilde{P}(e)}$, respectively. For each morphism $\overline{w}$ of $\Pi(E',E)$, since $w$ is homotopic to a walk of the form $(y|g^n|x)$, where $x$, $y\in E$ and $n\in\mathbb{Z}$, we may assume that $\overline{w}=\overline{(y|g^d|x)(x|g^{N}|x)^{r}}$, where $N$ denotes the cardinal of the $G$-orbit of $x$ in $E'$, $0\leq d< N$, and $r\in\mathbb{Z}$. To show $\widetilde{F}_{\widetilde{P}(e)}\Phi'(f)=\Phi_{\widetilde{P}(e)} F_{\widetilde{P}(e)}(f)$, it suffices to show $\widetilde{F}_{\widetilde{P}(e)}\Phi'(\overline{(y|g^d|x)})=\Phi_{\widetilde{P}(e)} F_{\widetilde{P}(e)}(\overline{(y|g^d|x)})$.

If $x\neq e_i$ and $y\neq e_i$, it can be shown straightforward that $$\widetilde{F}_{\widetilde{P}(e)}\Phi'(\overline{(y|g^d|x)})=\overline{(y|g^k|x)}=\Phi_{\widetilde{P}(e)}F_{\widetilde{P}(e)}(\overline{(y|g^d|x)}),$$
where $k$ is the minimal non-negative integer such that $g^{k}\cdot x=y$ in $C_{\widetilde{P}(e)}$. If $x=e_i$ and $y\neq e_i$, then $\widetilde{F}_{\widetilde{P}(e)}\Phi'(\overline{(y|g^d|x)})=\overline{(y|g^k|x)}$, where $k$ is the minimal non-negative integer such that $g^{k}\cdot x=y$ in $C_{\widetilde{P}(e)}$. By the definition of $\Phi_{\widetilde{P}(e)}$,
$$\Phi_{\widetilde{P}(e)}F_{\widetilde{P}(e)}(\overline{(y|g^d|x)})=\phi\psi f_{*}F_{\widetilde{P}(e)}(\overline{(y|g^d|x)}).$$
We have $\psi f_{*}F_{\widetilde{P}(e)}(\overline{(y|g^d|x)})=\overline{(y|g^l|e'_i)}$, where $l$ is the minimal non-negative integer such that $g^{l}\cdot e'_i=y$ in $C_{\widetilde{P}(e)}$. Since $e'_i=g\cdot x$ in $C_{\widetilde{P}(e)}$, $l=k-1$. Therefore
$$\Phi_{\widetilde{P}(e)}F_{\widetilde{P}(e)}(\overline{(y|g^d|x)})=\phi\psi f_{*}F_{\widetilde{P}(e)}(\overline{(y|g^d|x)})=\overline{(y|g^{k-1}|e'_i)(e'_i|g|e_i)}=\overline{(y|g^k|x)}=\widetilde{F}_{\widetilde{P}(e)}\Phi'(\overline{(y|g^d|x)}).$$ If $x\neq e_i$ and $y=e_i$, the proof of $\widetilde{F}_{\widetilde{P}(e)}\Phi'(\overline{(y|g^d|x)})=\Phi_{\widetilde{P}(e)} F_{\widetilde{P}(e)}(\overline{(y|g^d|x)})$ is similar to the case $x=e_i$ and $y\neq e_i$. If $x=y=e_i$, since $0\leq d<N$, we have $d=0$. Therefore
$$\Phi_{\widetilde{P}(e)} F_{\widetilde{P}(e)}(\overline{(y|g^d|x)})=\overline{(e_i|g^{-1}|e'_i)(e'_i|g|e_i)}=\overline{(e||e)}=\widetilde{F}_{\widetilde{P}(e)}\Phi'(\overline{(y|g^d|x)}).$$
\end{proof}

\begin{Cor}\label{fundamental group of BC and BG}
The groupoids $\Pi(E,E)$ and $\Pi(C,E)$ are isomorphic. Especially, if the Brauer configuration $E$ is connected, then the Brauer graph $C$ is also connected, and the fundamental groups of $E$ and $C$ are isomorphic.
\end{Cor}

\begin{proof}
By definition, $C$ is an admissible union of the family $\mathscr{C}=\{C_{\widetilde{P}(e)}\}_{\widetilde{P}(e)\subseteq C}\cup\{C'\}$ of sub-f-BCs. Moreover, $E$ is a union of the family $\mathscr{E}=\{E_{\widetilde{P}(e)}\}_{\widetilde{P}(e)\subseteq C}\cup\{E'\}$ of sub-f-BCs which is closed under finite intersections. We need to show this union is also admissible. For each polygon $P(e)$ of $E$ with $|P(e)|\geq 3$, assume that the elements of $P(e)$ is labeled by $e_1$, $\cdots$, $e_n$, then $K_e$ is the simplicial complex on $P(e)$ generated by simplices $\{e_1,e_2\}$, $\{e_2,e_3\}$, $\cdots$, $\{e_{n-1},e_n\}$. It is straightforward to show that $\Pi(K_e)$ is simply connected. For each polygon $P(e)$ of $E$ with $|P(e)|=2$, $K_e$ is the simplicial complex on $P(e)$ generated by the simplex $P(e)$. Therefore $\Pi(K_e)$ is also simply connected.

Since $E$ meets each connected components of each sub-f-BC of $E$ in $\mathscr{E}$ and each sub-f-BC of $C$ in $\mathscr{C}$, by Proposition \ref{Van-Kampen}, $\Pi(E,E)$ is the direct limit of the direct system $\{F_{\widetilde{P}(e)}:\Pi(E',E)\rightarrow\Pi(E_{\widetilde{P}(e)},E)\}_{\widetilde{P}(e)\subseteq C}$ (resp. $\Pi(C,E)$ is the direct limit of the direct system $\{\widetilde{F}_{\widetilde{P}(e)}:\Pi(C',E)\rightarrow\Pi(C_{\widetilde{P}(e)},E)\}_{\widetilde{P}(e)\subseteq C}$ ). By Proposition \ref{iso. of direct systems}, $\Pi(E,E)$ is isomorphic to $\Pi(C,E)$.

If $E$ is connected, then it is straightforward to show that $C$ is also connected. Since the groupoids $\Pi(E,E)$ and $\Pi(C,E)$ are isomorphic, the fundamental groups of $E$ and $C$ are isomorphic.
\end{proof}

Next we will calculate the fundamental group of a Brauer graph $E$. To do this we first calculate the fundamental group (groupoid) of $E$ in two special cases, then we use Proposition \ref{Van-Kampen} to deal with the general case. We denote by $F\langle x_1,x_2,\cdots,x_n\rangle$ the free group on the set $\{x_1,x_2,\cdots,x_n\}$.

\begin{Lem}\label{a calculation of fundamental group}
Let $E$ be the Brauer graph given by the diagram
$$
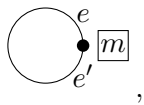
\begin{tikzpicture}
\draw (0,0) circle (0.5);
\fill (0.5,0) circle (0.5ex);
\node at(0.9,0) {$m$}; \node at(0.5,0.4) {$e$}; \node at(0.5,-0.4) {$e'$};
\draw (0.7,-0.2) rectangle (1.1,0.2);
\end{tikzpicture},$$
where the f-degree of the unique vertex of $E$ is $m$. Then $\Pi(E,e)\cong F\langle x,y\rangle/\langle x^{m}y=yx^{m}\rangle$.
\end{Lem}

\begin{proof}
Each walk of $E$ is homotopic to a walk of the form $$e_{n}\frac{\delta_{n}}{}e_{n-1}\frac{\delta_{n-1}}{}\cdots\frac{\delta_{3}}{}e_{2}\frac{\delta_{2}}{}e_{1}\frac{\delta_{1}}{}e_{0},$$
where $e_{0},e_{1},\cdots,e_{n}\in E$, $\delta_{1},\cdots,\delta_{n}\in\{g,g^{-1},\tau\}$, such that if $\delta_i =\tau$, then $e_{i-1}\neq e_i$. We will simply write a walk of this form as $(e_n|\delta_n \cdots \delta_2 \delta_1|e_0)$. Since $E$ contains only two half-edges, there is no confusion if we write a walk like that. Let $u=(e|g^2|e)$, $v=(e|\tau g|e)$ be closed walks of $E$ at $e$. First we need show that $\Pi(E,e)$ is generated by $\overline{u}$ and $\overline{v}$. Since each closed walk of $E$ at $e$ is homotopic to a walk of the form $u^{nm}w$, where $n\in\mathbb{Z}$ and $w$ is a closed special walk of $E$ at $e$ (for the definition of special walk, see Definition \ref{special-walk}), it suffices to show that for each closed special walk $w$ of $E$ at $e$, $\overline{w}$ belongs to the subgroup of $\Pi(E,e)$ generated by $\overline{u}$ and $\overline{v}$. We shall give a proof by induction on the length of $w$. Let
$$w=(e|g^{i_k}\tau g^{i_{k-1}}\tau\cdots \tau g^{i_1}\tau g^{i_0}|e),$$
where $0\leq i_0,i_k<2m$, and $0<i_j<2m$ for $1\leq j\leq k-1$. If $i_0\geq 2$, then $w=w'u$, where
$$w'=(e|g^{i_k}\tau g^{i_{k-1}}\tau\cdots \tau g^{i_1}\tau g^{i_{0}-2}|e)$$
is a closed special walk of $E$ at $e$ with $l(w')<l(w)$. By induction, $\overline{w'}$ belongs to the subgroup of $\Pi(E,e)$ generated by $\overline{u}$ and $\overline{v}$, and so is $\overline{w}$. If $i_0=1$, then $w=w'v$, where
$$w'=(e|g^{i_k}\tau g^{i_{k-1}}\tau\cdots \tau g^{i_1}|e)$$
is a closed special walk of $E$ at $e$ with $l(w')<l(w)$. Therefore by induction, $\overline{w}$ belongs to the subgroup of $\Pi(E,e)$ generated by $\overline{u}$ and $\overline{v}$. When $i_0=0$, since $uv^{-1}\sim (e|g\tau|e)$, $w\sim w'uv^{-1}$, where
$$w'=(e|g^{i_k}\tau g^{i_{k-1}}\tau\cdots \tau g^{i_{1}-1}|e)$$
is a closed special walk of $E$ at $e$ with $l(w')<l(w)$. By induction, $\overline{w}$ belongs to the subgroup of $\Pi(E,e)$ generated by $\overline{u}$ and $\overline{v}$.

Define a group homomorphism $f':F\langle x,y\rangle\rightarrow\Pi(E,e)$, $x\mapsto\overline{u}$, $y\mapsto\overline{v}$. Since $\Pi(E,e)$ is generated by $\overline{u}$ and $\overline{v}$, $f'$ is surjective. Since $u^m v$ is homotopic to $v u^m$, $f'$ induces a surjective group homomorphism
$$f:F\langle x,y\rangle/\langle x^{m}y=yx^{m}\rangle\rightarrow\Pi(E,e).$$
We need to show that $f$ is injective. For $a\in F\langle x,y\rangle$, denote by $\overline{a}$ the image of $a$ in $F\langle x,y\rangle/\langle x^{m}y=yx^{m}\rangle$.

When $m>1$, if $f(\overline{a})=1$, we may assume that
\begin{equation}\label{formula-a}\overline{a}=\overline{x^{mn}x^{i_k}y^{j_k}x^{i_{k-1}}y^{j_{k-1}}\cdots y^{j_2}x^{i_1}y^{j_1}x^{i_0}},\end{equation}
where $n\in\mathbb{Z}$, $0\leq i_0,i_k<m$, $0<i_l<m$ for $1\leq l\leq k-1$, $j_1,\cdots,j_k\neq 0$. Denote by $z=xy^{-1}$, then $y^{-1}=x^{-1}z$. For each $1\leq d\leq k$ with $j_d<0$, we may replace $xy^{j_d}$ by $(x^{m})^{j_d +1}z(x^{m-1}z)^{-j_d -1}$ in formula (\ref{formula-a}). Since $\overline{x^{m}y}=\overline{yx^{m}}$ in $F\langle x,y\rangle/\langle x^{m}y=yx^{m}\rangle$, we may express $\overline{a}$ as
$$\overline{x^{mN}a_{r}^{n_r}a_{r-1}^{n_{r-1}}\cdots a_{2}^{n_2}a_{1}^{n_1}},$$
where $N\in\mathbb{Z}$, $n_1,\cdots,n_r >0$, $a_j\in\{x,y,z\}$ for $1\leq j\leq r$, such that
\begin{itemize}
\item [$(1)$] $a_{j-1}\neq a_j$ for each $2\leq j\leq r$;
\item [$(2)$] if $a_j =x$, then $0<n_j<m$;
\item [$(3)$] there exists no $2\leq j\leq r$ such that $a_{j-1}=y$ and $a_j =z$.
\end{itemize}
Since $f(\overline{x})=\overline{(e|g^2|e)}$, $f(\overline{y})=\overline{(e|\tau g|e)}$, $f(\overline{z})=\overline{(e|g\tau|e)}$, it can be shown that $$f(\overline{a_{r}^{n_r}a_{r-1}^{n_{r-1}}\cdots a_{2}^{n_2}a_{1}^{n_1}})=\overline{w},$$
where $w$ is a closed special walk of $E$ at $e$ of length $2(n_1 +\cdots +n_r)$. Therefore $f(\overline{a})=\overline{(e|g^{N d(e)}|e)w}$. By Proposition \ref{unique factorization}, $N=0$ and $w=(e||e)$. Then $r =0$ and $\overline{a}=1$.

When $m=1$, if $f(\overline{a})=1$, we may assume that $\overline{a}=\overline{x^{n}y^{l}}$ for some $n$, $l\in\mathbb{Z}$. If $l\geq 0$, then $f(\overline{a})=\overline{(e|g^{n d(e)}|e)w}$, where $w=(e|(\tau g)^{l}|e)$ is a closed special walk of $E$ at $e$. By Proposition \ref{unique factorization}, $n=0$ and $l=0$. Therefore $\overline{a}=1$. If $l\leq 0$, denote by $z=xy^{-1}$. Then $\overline{x^{n}y^{l}}=\overline{x^{n+l}z^{-l}}$. Since $f(\overline{z})=\overline{(e|g\tau|e)}$, $f(\overline{a})=\overline{(e|g^{(n+l) d(e)}|e)w}$, where $w=(e|(g\tau)^{-l}|e)$ is a closed special walk of $E$ at $e$. By Proposition \ref{unique factorization}, $n+l=0$ and $-l=0$. Therefore $\overline{a}=1$.
\end{proof}

The following fact should be well-known.

\begin{Lem}\label{isomorphism functor}
Let $\mathscr{G}$, $\mathscr{G}'$ be two connected groupoids. Let $F:\mathscr{G}\rightarrow\mathscr{G}'$ be a functor, which induces a bijection between objects of $\mathscr{G}$ and objects of $\mathscr{G}'$. If there exists an object $x$ of $\mathscr{G}$ such that $F$ induces an isomorphism between $\mathscr{G}(x,x)$ and $\mathscr{G}'(Fx,Fx)$, then $F$ is an isomorphism functor.
\end{Lem}

\begin{proof}
Since $F$ induces a bijection on objects, to show that $F$ is an isomorphism, it suffices to show that $F$ is fully faithful. Since $\mathscr{G}$ is connected, for each object $y$ of $\mathscr{G}$, we can choose a morphism $f_y \in\mathscr{G}(x,y)$. For $g$, $h\in\mathscr{G}(y,z)$, if $F(g)=F(h)$, then $F(f_{z}^{-1}gf_{y})=F(f_z)^{-1}F(g)F(f_y)
=F(f_z)^{-1}F(h)F(f_y)=F(f_{z}^{-1}hf_{y})$. Since $F$ induces an isomorphism between $\mathscr{G}(x,x)$ and $\mathscr{G}'(Fx,Fx)$, $f_{z}^{-1}gf_{y}=f_{z}^{-1}hf_{y}$. Therefore $g=h$ and $F$ is faithful. For each $g'\in\mathscr{G}'(Fy,Fz)$, we have $g'=F(f_z)F(f_{z}^{-1})g'F(f_y)F(f_{y}^{-1})$. Since $F(f_{z}^{-1})g'F(f_y)\in\mathscr{G}'(Fx,Fx)$, there exists $f\in\mathscr{G}(x,x)$ such that $F(f)=F(f_{z}^{-1})g'F(f_y)$. Let $g=f_z f f_{y}^{-1}\in\mathscr{G}(y,z)$. Then $F(g)=F(f_z)F(f)F(f_{y}^{-1})=g'$. Therefore $F$ is full.
\end{proof}

In the following of this paper, by abuse of notation, we also denote by $w$ the corresponding morphism in $\Pi(Q)$ for a walk $w$ in the quiver $Q$, where $\Pi(Q)$ is the fundamental groupoid of the quiver $Q$ (see Section \ref{sec:isomorphic-fundamental-groups}).

\begin{Lem}\label{a calculation of fundamental groupoid}
Let $E$ be the Brauer graph given by the diagram
$$\begin{tikzpicture}
\draw (0,0)--(2,0);
\fill (0,0) circle (0.5ex);
\fill (2,0) circle (0.5ex);
\node at(-0.4,0) {$m$};
\draw (-0.6,-0.2) rectangle (-0.2,0.2);
\node at(2.4,0) {$n$};
\draw (2.2,-0.2) rectangle (2.6,0.2);
\end{tikzpicture},$$
where the f-degree of the vertex on the left (resp. right) is $m$ (resp. $n$). Then the fundamental groupoid $\Pi(E,E)$ is isomorphic to $\mathscr{F}/\langle z x^{m}=y^{n} z\rangle$, where $\mathscr{F}$ is the fundamental groupoid of the quiver
$$
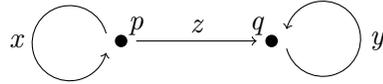
\begin{tikzpicture}
\draw[->] (0.2,0) -- (1.8,0);
\draw[->] (-0.2,0.1) arc (15:345:0.5);
\draw[->] (2.2,-0.1) arc (195:525:0.5);
\fill (0,0) circle (0.5ex);
\fill (2,0) circle (0.5ex);
\node at(-1.4,0) {$x$};
\node at(3.4,0) {$y$};
\node at(1,0.2) {$z$};
\node at(0.2,0.2) {$p$};
\node at(1.8,0.2) {$q$};
\end{tikzpicture}.$$
\end{Lem}

\begin{proof}
Each walk of $E$ is homotopic to a walk of the form $$e_{n}\frac{\delta_{n}}{}e_{n-1}\frac{\delta_{n-1}}{}\cdots\frac{\delta_{3}}{}e_{2}\frac{\delta_{2}}{}e_{1}\frac{\delta_{1}}{}e_{0},$$
where $e_{0},e_{1},\cdots,e_{n}\in E$, $\delta_{1},\cdots,\delta_{n}\in\{g,g^{-1},\tau\}$, such that if $\delta_i =\tau$, then $e_{i-1}\neq e_i$. We will simply write a walk of this form as $(e_n|\delta_n \cdots \delta_2 \delta_1|e_0)$. Let $e$ (resp. $e'$) be the half-edge of $E$ on the left (resp. right) of the diagram. First we need to show that $\Pi(E,e)$ is isomorphic to $F\langle a,b\rangle/\langle a^m =b^n\rangle$. Define a group homomorphism $f':F\langle a,b\rangle\rightarrow\Pi(E,e)$, $a\mapsto \overline{(e|g|e)}$, $b\mapsto \overline{(e|\tau g\tau|e)}$. It is straightforward to show that $f'$ is surjective. Since $(e'|\tau g^m|e)\sim(e'|g^n\tau|e)$, $f'$ induces a surjective group homomorphism $f:F\langle a,b\rangle/\langle a^m =b^n\rangle\rightarrow\Pi(E,e)$. We need to show that $f$ is injective. For $c\in F\langle a,b\rangle$, denote by $\overline{c}$ the image of $c$ in $F\langle a,b\rangle/\langle a^m =b^n\rangle$. If $f(\overline{c})=1$ for some $c\in F\langle a,b\rangle$, we may assume that
$$\overline{c}=\overline{a^{rm}a^{p_k}b^{q_k}a^{p_{k-1}}b^{q_{k-1}}\cdots a^{p_1}b^{q_1}a^{p_0}},$$
where $r\in\mathbb{Z}$, $0\leq p_0,p_k<m$, $0<p_i<m$ for $1\leq i\leq k-1$, $0<q_j<n$ for $1\leq j\leq k$. Then $f(\overline{c})=\overline{(e|g^{rd(e)}|e)w}$, where $$w=(e|g^{p_k}\tau g^{q_k}\tau g^{p_{k-1}}\tau g^{q_{k-1}}\tau\cdots g^{p_1}\tau g^{q_1}\tau g^{p_0}|e)$$
is a closed special walk of $E$ at $e$. Since $(e|g^{rd(e)}|e)w\sim (e||e)$, by Proposition \ref{unique factorization}, $w=(e||e)$ and $r=0$. Therefore $\overline{c}=1$ and $f$ is injective.

Let $\mathscr{F'}$ be the fundamental groupoid of the quiver
$$
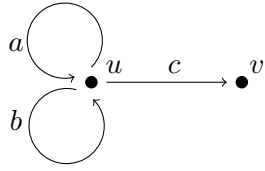
\begin{tikzpicture}
\draw[->] (0.2,0) -- (1.8,0);
\draw[->] (0,0.2) arc (-45:285:0.5);
\draw[->] (-0.2,-0.1) arc (75:405:0.5);
\fill (0,0) circle (0.5ex);
\fill (2,0) circle (0.5ex);
\node at(-1,0.5) {$a$};
\node at(-1,-0.5) {$b$};
\node at(1.1,0.2) {$c$};
\node at(0.3,0.2) {$u$};
\node at(2.2,0.2) {$v$};
\end{tikzpicture}.$$
Define a functor $F:\mathscr{F'}/\langle a^{m}=b^{n}\rangle\rightarrow \Pi(E,E)$, $u\mapsto e$, $v\mapsto e'$, $a\mapsto \overline{(e|g|e)}$, $b\mapsto \overline{(e|\tau g\tau|e)}$, $c\mapsto \overline{(e'|\tau|e)}$. Since the group $(\mathscr{F'}/\langle a^{m}=b^{n}\rangle)(u,u)$ is isomorphic to the group $F\langle a,b\rangle/\langle a^m =b^n\rangle$, $F$ induces an isomorphism $(\mathscr{F'}/\langle a^{m}=b^{n}\rangle)(u,u)\rightarrow \Pi(E,e)$. By Lemma \ref{isomorphism functor}, $F$ is an isomorphism functor.

Define a functor $G:\mathscr{F'}/\langle a^{m}=b^{n}\rangle\rightarrow \mathscr{F}/\langle z x^{m}=y^{n} z\rangle$, $u\mapsto p$, $v\mapsto q$, $a\mapsto x$, $b\mapsto z^{-1}yz$, $c\mapsto z$. It can be shown that $G$ is an isomorphism functor. Therefore the fundamental groupoid $\Pi(E,E)$ is isomorphic to $\mathscr{F}/\langle z x^{m}=y^{n} z\rangle$, where the isomorphism is given by $\mathscr{F}/\langle z x^{m}=y^{n} z\rangle\rightarrow\Pi(E,E)$, $p\mapsto e$, $q\mapsto e'$, $x\mapsto \overline{(e|g|e)}$, $y\mapsto \overline{(e'|g|e')}$, $z\mapsto \overline{(e'|\tau|e)}$.
\end{proof}

\begin{Prop}\label{fundamental group of BG}
Let $E=(E,P,L,d)$ be a connected Brauer graph with $n$ vertices $v_1$, $\cdots$, $v_n$ and $k$ edges. Let $d_i=d_f(v_i)$ be the f-degree of $v_i$ for each $1\leq i\leq n$, and let $r=k-n+1$. Then the fundamental group of $E$ is isomorphic to
$$F\langle a_1,\cdots,a_n,b_1,\cdots,b_r\rangle/\langle a_{1}^{d_1}=\cdots=a_{n}^{d_n}, a_{1}^{d_1} b_1=b_1 a_{1}^{d_1},\cdots, a_{1}^{d_1} b_r=b_r a_{1}^{d_1}\rangle.$$
\end{Prop}

\begin{proof}
If $k=1$ then it follows from Lemma \ref{a calculation of fundamental group} and Lemma \ref{a calculation of fundamental groupoid}. Therefore we may assume that $k>1$.

Let $P(e_1)$, $\cdots$, $P(e_k)$ be all edges of $E$. For each $1\leq i\leq k$, define a sub-f-BC $E_i=(E_i,P_i,L_i,d_i)$ of $E$ as follows: $E_i=E$ as $G$-sets; $P_i(e_i)=P(e_i)$ and $P_i(e)=\{e\}$ for each $e\notin P(e_i)$; $L_i$ is trivial and $d_i=d$. Let $E'=(E',P',L',d')$ be the intersection of $E_i$'s. Since $k>1$, $P'(e)=\{e\}$ for each $e\in E'$. For each $1\leq j\leq n$, choose $h_j\in v_j$, and let $A=\{h_1,\cdots,h_n\}$ be a subset of $E$. Denote by $N_j$ the cardinal of $v_j$ for each $1\leq j\leq n$.

For each $1\leq i\leq k$, if the two half-edges of $P(e_i)$ belong to the same vertex $v_j$, then we denote by $\mathscr{F}_{i}$ the fundamental groupoid of quiver
$$
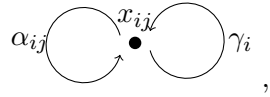
\begin{tikzpicture}
\draw[->] (-0.2,0.1) arc (15:345:0.5);
\draw[->] (0.2,-0.1) arc (195:525:0.5);
\fill (0,0) circle (0.5ex);
\node at(-1.4,0) {$\alpha_{ij}$};
\node at(0,0.3) {$x_{ij}$};
\node at(1.4,0) {$\gamma_{i}$};
\end{tikzpicture},$$
and for each $1\leq l\leq n$ with $l\neq j$, denote by $\mathscr{F}_{il}$ the fundamental groupoid of quiver
$$
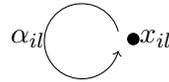
\begin{tikzpicture}
\draw[->] (-0.2,0.1) arc (15:345:0.5);
\fill (0,0) circle (0.5ex);
\node at(-1.4,0) {$\alpha_{il}$};
\node at(0.3,0) {$x_{il}$};
\end{tikzpicture}.$$
Let $\Sigma_i$ be the groupoid
$$\bigsqcup_{1\leq l\leq n, l\neq j}\mathscr{F}_{il}\sqcup (\mathscr{F}_{i}/\langle \alpha_{ij}^{d_j}\gamma_i=\gamma_i\alpha_{ij}^{d_j}\rangle).$$
By Lemma \ref{a calculation of fundamental group} and Lemma \ref{isomorphism of fundamental groupoids}, there is an isomorphism of groupoids $F_i:\Sigma_i\rightarrow\Pi(E_i,A)$ such that $F_i(x_{il})=h_l$ and $F_i(\alpha_{il})=\overline{(h_l|g^{N_l}|h_l)}$ for each $1\leq l\leq n$. If the two half-edges of $P(e_i)$ belong to two different vertices $v_{j_1}$, $v_{j_2}$ with $j_1<j_2$, then we denote by $\mathscr{F}_{i}$ the fundamental groupoid of quiver
$$
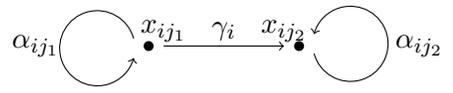
\begin{tikzpicture}
\draw[->] (0.2,0) -- (1.8,0);
\draw[->] (-0.2,0.1) arc (15:345:0.5);
\draw[->] (2.2,-0.1) arc (195:525:0.5);
\fill (0,0) circle (0.4ex);
\fill (2,0) circle (0.4ex);
\node at(-1.5,0) {$\alpha_{i j_1}$};
\node at(3.6,0) {$\alpha_{i j_2}$};
\node at(1,0.2) {$\gamma_i$};
\node at(0.2,0.2) {$x_{i j_1}$};
\node at(1.8,0.2) {$x_{i j_2}$};
\end{tikzpicture},$$
and for each $1\leq l\leq n$ with $l\neq j_1,j_2$, denote by $\mathscr{F}_{il}$ the fundamental groupoid of quiver
$$
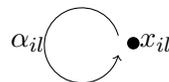
\begin{tikzpicture}
\draw[->] (-0.2,0.1) arc (15:345:0.5);
\fill (0,0) circle (0.5ex);
\node at(-1.4,0) {$\alpha_{il}$};
\node at(0.3,0) {$x_{il}$};
\end{tikzpicture}.$$
Let $\Sigma_i$ be the groupoid
$$\bigsqcup_{1\leq l\leq n, l\neq j_1,j_2}\mathscr{F}_{il}\sqcup (\mathscr{F}_{i}/\langle \alpha_{i j_2}^{d_{j_2}}\gamma_i=\gamma_i\alpha_{i j_1}^{d_{j_1}}\rangle).$$
By Lemma \ref{a calculation of fundamental groupoid} and Lemma \ref{isomorphism of fundamental groupoids}, there is an isomorphism of groupoids $F_i:\Sigma_i\rightarrow\Pi(E_i,A)$ such that $F_i(x_{il})=h_l$ and $F_i(\alpha_{il})=\overline{(h_l|g^{N_l}|h_l)}$ for each $1\leq l\leq n$.

For each $1\leq l\leq n$, let $\mathscr{G}_l$ be the fundamental groupoid of quiver
$$
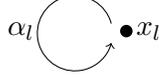
\begin{tikzpicture}
\draw[->] (-0.2,0.1) arc (15:345:0.5);
\fill (0,0) circle (0.5ex);
\node at(-1.4,0) {$\alpha_{l}$};
\node at(0.3,0) {$x_{l}$};
\end{tikzpicture}.$$
Let $\Sigma'$ be the groupoid $\bigsqcup_{1\leq l\leq n}\mathscr{G}_l$. It is straightforward to show that there is an isomorphism of groupoids $F':\Sigma'\rightarrow\Pi(E',A)$ such that $F'(x_{l})=h_l$ and $F'(\alpha_{l})=\overline{(h_l|g^{N_l}|h_l)}$ for each $1\leq l\leq n$. For each $1\leq i\leq k$, let $\tau_i:E'\rightarrow E_i$ be the inclusion morphism, and let $\mu_i=F_{i}^{-1}\tau_{i*}F':\Sigma'\rightarrow\Sigma_i$. Then $F'$ together with $F_i$'s define an isomorphism of direct systems $\{\mu_i:\Sigma'\rightarrow\Sigma_i\}_{1\leq i\leq k}$ and $\{\tau_{i*}:\Pi(E',A)\rightarrow\Pi(E_i,A)\}_{1\leq i\leq k}$.

Define a quiver $Q$ as follows: $Q_0=\{v_1,\cdots,v_n\}$, $Q_1=\{\alpha'_1,\cdots,\alpha'_n,\gamma'_1,\cdots,\gamma'_k\}$; define $s(\alpha'_l)=t(\alpha'_l)=v_l$ for each $1\leq l\leq n$; for each $1\leq i\leq k$, if the two half-edges of $P(e_i)$ belong to the same vertex $v_j$, define $s(\gamma'_i)=t(\gamma'_i)=v_j$; if the two half-edges of $P(e_i)$ belong to different vertices $v_{j_1}$, $v_{j_2}$ with $j_1<j_2$, define $s(\gamma'_i)=v_{j_1}$ and $t(\gamma'_i)=v_{j_2}$. Let $\Pi(Q)$ be the fundamental groupoid of the quiver $Q$, and let $\Sigma$ be the groupoid
$$\Pi(Q)/\langle(\alpha'_{q(i)})^{d_{q(i)}}\gamma'_i=\gamma'_i(\alpha'_{p(i)})^{d_{p(i)}}\mid 1\leq i\leq k \rangle,$$
where $s(\gamma'_i)=v_{p(i)}$ and $t(\gamma'_i)=v_{q(i)}$. For each $1\leq i\leq k$, define a functor $\nu_i:\Sigma_i\rightarrow\Sigma$ as $\nu_i(x_{il})=v_l$, $\nu_i(\alpha_{il})=\alpha'_l$ for $1\leq l\leq n$ and $\nu_i(\gamma_i)=\gamma'_i$. Since for each $1\leq i\leq k$, $\mu_i(x_l)=x_{il}$ and $\mu_i(\alpha_l)=\alpha_{il}$ for each $1\leq l\leq n$, $\nu_i\mu_i$ are equal for different $1\leq i\leq k$. We need to show that the groupoid $\Sigma$ together with functors $\nu_i:\Sigma_i\rightarrow\Sigma$ form the direct limit of the direct system $\{\mu_i:\Sigma'\rightarrow\Sigma_i\}_{1\leq i\leq k}$.

Let $\mathscr{G}$ be a groupoid and $\Phi_i:\Sigma_i\rightarrow\mathscr{G}$ be a functor for each $1\leq i\leq k$, such that $\Phi_i\mu_i$ are equal for different $1\leq i\leq k$. Define a functor $\Phi:\Sigma\rightarrow\mathscr{G}$ as follows. $\Phi(v_l):=\Phi_1(x_{1l})$, $\Phi(\alpha'_l):=\Phi_1(\alpha_{1l})$ for $1\leq l\leq n$; $\Phi(\gamma'_i):=\Phi_i(\gamma_i)$ for $1\leq i\leq k$. Since
$$\Phi(\alpha'_l)=\Phi_1(\alpha_{1l})=\Phi_1\mu_1(\alpha_{l})=\Phi_i\mu_i(\alpha_l)=\Phi_i(\alpha_{il})$$
for each $1\leq i\leq k$ and for each $1\leq l\leq n$, we have $$\Phi((\alpha'_{q(i)})^{d_{q(i)}}\gamma'_i)=\Phi_i(\alpha_{i,q(i)})^{d_{q(i)}}\Phi_i(\gamma_i)=\Phi_i(\alpha_{i,q(i)}^{d_{q(i)}}\gamma_i)
=\Phi_i(\gamma_i\alpha_{i,p(i)}^{d_{p(i)}})=\Phi(\gamma'_i (\alpha'_{p(i)})^{d_{p(i)}}),$$
where $s(\gamma'_i)=v_{p(i)}$ and $t(\gamma'_i)=v_{q(i)}$. Therefore $\Phi$ is well-defined. It is straightforward to show that $\Phi\nu_i=\Phi_i$ for each $1\leq i\leq k$. If there exists another functor $\Psi:\Sigma\rightarrow\mathscr{G}$ such that $\Psi\nu_i=\Phi_i$ for each $1\leq i\leq k$, then $\Psi(v_l)=\Psi\nu_1(x_{1l})=\Phi_1(x_{1l})=\Phi(v_l)$, $\Psi(\alpha'_l)=\Psi\nu_1(\alpha_{1l})=\Phi_1(\alpha_{1l})=\Phi(\alpha'_l)$ for $1\leq l\leq n$, and  $\Psi(\gamma'_i)=\Psi\nu_i(\gamma_i)=\Phi_i(\gamma_i)=\Phi(\gamma'_i)$ for each $1\leq i\leq k$. Therefore $\Psi=\Phi$ and such functor is unique.

Since $E$ is an admissible union of sub-f-BCs $\{E_i\}_{1\leq i\leq k}\cup\{E'\}$, and since $A$ meets each connected component of $E'$ and $E_i$'s, by Proposition \ref{Van-Kampen}, $\Pi(E,A)$ is the direct limit of the direct system $\{\tau_{i*}:\Pi(E',A)\rightarrow \Pi(E_i,A)\}_{1\leq i\leq k}$. Since the direct systems $\{\mu_i:\Sigma'\rightarrow\Sigma_i\}_{1\leq i\leq k}$ and $\{\tau_{i*}:\Pi(E',A)\rightarrow\Pi(E_i,A)\}_{1\leq i\leq k}$ are isomorphic, their direct limits are also isomorphic. Therefore the groupoid $\Pi(E,A)$ is isomorphic to $\Sigma$.

Let $\Gamma(E)$ be the diagram of $E$. Since $E$ is a connected Brauer graph, $\Gamma(E)$ is a connected graph. Choose a spanning tree $T$ of $\Gamma(E)$, and let $R$ be the subquiver of $Q$ such that $R_0=Q_0$ and $R_1=\{\gamma'_i\mid P(e_i)$ is a edge of $T\}$. Then $R$ is a directed tree. For each $1\leq l\leq n$, let $w_l$ be the unique reduced walk of $R$ (that is, a walk containing no subwalk of the form $\alpha^{-1}\alpha$ or $\alpha\alpha^{-1}$ for each arrow $\alpha$ of $R$) from $v_1$ to $v_l$. We construct a new quiver $B$ such that $R$ is a subquiver of $B$ as follows:
$$B_0=Q_0, B_1=\{\epsilon_l\mid 1\leq l\leq n\}\cup\{\gamma'_i\mid 1\leq i\leq k\},$$
where $s(\epsilon_l)=t(\epsilon_l)=v_1$ for each $1\leq l\leq n$. Let $\Delta$ be the groupoid
$$\Pi(B)/\langle\epsilon_{1}^{d_1}=\cdots=\epsilon_{n}^{d_n},w_{q(i)}^{-1}\gamma'_{i}w_{p(i)}\epsilon_{1}^{d_1}=\epsilon_{1}^{d_1}w_{q(i)}^{-1}\gamma'_{i}w_{p(i)}\text{ } (1\leq i\leq k \text{ and } \gamma'_i\notin R_1)\rangle,$$
where $s(\gamma'_i)=v_{p(i)}$ and $t(\gamma'_i)=v_{q(i)}$. Recall that $\Sigma$ is the groupoid $$\Pi(Q)/\langle(\alpha'_{q(i)})^{d_{q(i)}}\gamma'_i=\gamma'_i(\alpha'_{p(i)})^{d_{p(i)}}\mid 1\leq i\leq k \rangle,$$
where $s(\gamma'_i)=v_{p(i)}$ and $t(\gamma'_i)=v_{q(i)}$. Since the relations
$$\{(\alpha'_{q(i)})^{d_{q(i)}}\gamma'_i=\gamma'_i(\alpha'_{p(i)})^{d_{p(i)}}\mid 1\leq i\leq k\}$$
of $\Pi(Q)$ are equivalent to the relations
$$\{w_l(\alpha'_1)^{d_1}=(\alpha'_l)^{d_l}w_l\mid 2\leq l\leq n\}\cup\{w_{q(i)}^{-1}\gamma'_{i}w_{p(i)}(\alpha'_1)^{d_1}=(\alpha'_1)^{d_1}w_{q(i)}^{-1}\gamma'_{i}w_{p(i)}\mid \gamma'_i\notin R_1\},$$ the functor $\Delta\rightarrow\Sigma$, $v_l\mapsto v_l$, $\gamma'_i\mapsto\gamma'_i$, $\epsilon_1\mapsto\alpha'_1$, $\epsilon_l\mapsto w_{l}^{-1}\alpha'_{l}w_l$ for $2\leq l\leq n$ is an isomorphism of groupoids.

Since the fundamental group of a quiver is isomorphic to the fundamental group of its geometric realization, the fundamental group of $B$ at $v_1$ is the free group with generators $\epsilon_1$, $\cdots$, $\epsilon_n$ and $w_{q(i)}^{-1}\gamma'_{i}w_{p(i)}$ for each $1\leq i\leq k$ with $\gamma'_i\notin R_1$. Since $\Pi(E,A)$ is isomorphic to $\Delta$, the fundamental group of $E$ is isomorphic to $\Delta(v_1,v_1)$. Since the group $\Delta(v_1,v_1)$ is isomorphic to the fundamental group of $B$ at $v_1$ modulo the relations $\epsilon_{1}^{d_1}=\cdots=\epsilon_{n}^{d_n}$ and $w_{q(i)}^{-1}\gamma'_{i}w_{p(i)}\epsilon_{1}^{d_1}=\epsilon_{1}^{d_1}w_{q(i)}^{-1}\gamma'_{i}w_{p(i)}$ for each $1\leq i\leq k$ with $\gamma'_i\notin R_1$, and since the number of $1\leq i\leq k$ with $\gamma'_i\notin R_1$ is $k-n+1$, $\Delta(v_1,v_1)$ is isomorphic to $$F\langle a_1,\cdots,a_n,b_1,\cdots,b_r\rangle/\langle a_{1}^{d_1}=\cdots=a_{n}^{d_n}, a_{1}^{d_1} b_1=b_1 a_{1}^{d_1},\cdots, a_{1}^{d_1} b_r=b_r a_{1}^{d_1}\rangle,$$
where $r=k-n+1$.
\end{proof}

We illustrate the above proof by the following example.

\begin{Ex1}
Let $E$ be the Brauer graph given by the diagram
$$
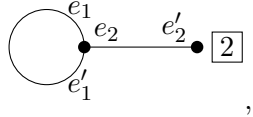
\begin{tikzpicture}
\draw (0,0) circle (0.5);
\fill (0.5,0) circle (0.5ex);
\draw (0.5,0)--(2,0);
\fill (2,0) circle (0.5ex);
\node at(2.4,0) {$2$};
\node at(0.45,0.5) {$e_1$};
\node at(0.45,-0.5) {$e'_1$};
\node at(0.8,0.2) {$e_2$};
\node at(1.7,0.27) {$e'_2$};
\draw (2.2,-0.2) rectangle (2.6,0.2);
\end{tikzpicture},$$
which has two vertices $v_1=G\cdot e_{1}$, $v_2=G\cdot e'_{2}$ and two edges $P(e_1),P(e_2)$. Let $E_1$ be the sub-f-BC of $E$ given by the diagram
$$
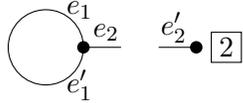
\begin{tikzpicture}
\draw (0,0) circle (0.5);
\fill (0.5,0) circle (0.5ex);
\draw (0.5,0)--(1,0);
\draw (1.5,0)--(2,0);
\fill (2,0) circle (0.5ex);
\node at(2.4,0) {$2$};
\node at(0.45,0.5) {$e_1$};
\node at(0.45,-0.5) {$e'_1$};
\node at(0.8,0.2) {$e_2$};
\node at(1.7,0.27) {$e'_2$};
\draw (2.2,-0.2) rectangle (2.6,0.2);
\end{tikzpicture}$$
and let $E_2$ be the sub-f-BC of $E$ given by the diagram
$$
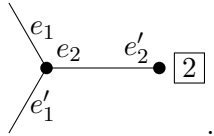
\begin{tikzpicture}
\draw (0.5,0)--(0,0.866);
\draw (0.5,0)--(0,-0.866);
\fill (0.5,0) circle (0.5ex);
\draw (0.5,0)--(2,0);
\fill (2,0) circle (0.5ex);
\node at(2.4,0) {$2$};
\node at(0.45,0.5) {$e_1$};
\node at(0.45,-0.5) {$e'_1$};
\node at(0.8,0.2) {$e_2$};
\node at(1.7,0.27) {$e'_2$};
\draw (2.2,-0.2) rectangle (2.6,0.2);
\end{tikzpicture}.$$
Moreover, let $E'=E_1\cap E_2$, which a sub-f-BC of $E$ given by the diagram
$$
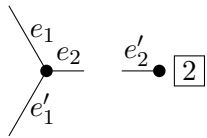
\begin{tikzpicture}
\draw (0.5,0)--(0,0.866);
\draw (0.5,0)--(0,-0.866);
\fill (0.5,0) circle (0.5ex);
\draw (0.5,0)--(1,0);
\draw (1.5,0)--(2,0);
\fill (2,0) circle (0.5ex);
\node at(2.4,0) {$2$};
\node at(0.45,0.5) {$e_1$};
\node at(0.45,-0.5) {$e'_1$};
\node at(0.8,0.2) {$e_2$};
\node at(1.7,0.27) {$e'_2$};
\draw (2.2,-0.2) rectangle (2.6,0.2);
\end{tikzpicture}.$$
Let $A=\{e_2,e'_2\}$ be a subset of $E$. Then $\Pi(E,A)$ is the direct limit of the direct system
\begin{equation}\label{direct-system}\Pi(E_1,A)\leftarrow\Pi(E',A)\rightarrow\Pi(E_2,A).\end{equation}

Let $Q^{(1)}$ be the quiver
$$\begin{tikzpicture}
\draw[->] (-0.2,0.1) arc (15:345:0.5);
\draw[->] (0.2,-0.1) arc (195:525:0.5);
\fill (0,0) circle (0.5ex);
\node at(-1.5,0) {$\alpha_{11}$};
\node at(0,0.4) {$x_{11}$};
\node at(1.45,0) {$\gamma_{1}$};
\node at(2,0) {$\sqcup$};
\draw[->] (4,0.1) arc (15:345:0.5);
\fill (4.2,0) circle (0.5ex);
\node at(4.6,0) {$x_{12}$};
\node at(2.7,0) {$\alpha_{12}$};
\end{tikzpicture},$$
$Q^{(2)}$ be the quiver
$$\begin{tikzpicture}
\draw[->] (0.2,0) -- (1.8,0);
\draw[->] (-0.2,0.1) arc (15:345:0.5);
\draw[->] (2.2,-0.1) arc (195:525:0.5);
\fill (0,0) circle (0.4ex);
\fill (2,0) circle (0.4ex);
\node at(-1.5,0) {$\alpha_{2 1}$};
\node at(3.6,0) {$\alpha_{2 2}$};
\node at(1,0.2) {$\gamma_2$};
\node at(0.2,0.2) {$x_{2 1}$};
\node at(1.75,0.25) {$x_{2 2}$};
\end{tikzpicture},$$
$Q'$ be the quiver
$$\begin{tikzpicture}
\draw[->] (-0.2,0.1) arc (15:345:0.5);
\fill (0,0) circle (0.5ex);
\node at(-1.4,0) {$\alpha_{1}$};
\node at(0.4,0) {$x_{1}$};
\node at(1,0) {$\sqcup$};
\draw[->] (3,0.1) arc (15:345:0.5);
\fill (3.2,0) circle (0.5ex);
\node at(3.6,0) {$x_{2}$};
\node at(1.7,0) {$\alpha_{2}$};
\end{tikzpicture}.$$
Let $\Sigma_1=\Pi(Q^{(1)})/\langle\alpha_{11}\gamma_1=\gamma_1\alpha_{11}\rangle$, $\Sigma_2=\Pi(Q^{(1)})/\langle\gamma_2\alpha_{21}=\alpha_{22}^{2}\gamma_2\rangle$, $\Sigma'=\Pi(Q')$.
Then the direct system (\ref{direct-system}) is isomorphic to the direct system
\begin{equation}\label{direct-system-2}\Sigma_1\leftarrow\Sigma'\rightarrow\Sigma_2,\end{equation}
where the functor $\Sigma'\rightarrow\Sigma_i$ maps each $x_j$ to $x_{ij}$ and maps each $\alpha_j$ to $\alpha_{ij}$.

Let $Q$ be the quiver
$$
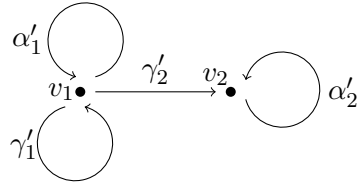
\begin{tikzpicture}
\draw[->] (0.2,0) -- (1.8,0);
\draw[->] (0.2,0.2) arc (-75:255:0.5);
\draw[->] (-0.2,-0.2) arc (105:435:0.5);
\draw[->] (2.2,-0.1) arc (195:525:0.5);
\fill (0,0) circle (0.4ex);
\fill (2,0) circle (0.4ex);
\node at(-0.7,0.7) {$\alpha'_{1}$};
\node at(3.5,0) {$\alpha'_{2}$};
\node at(-0.75,-0.7) {$\gamma'_1$};
\node at(1,0.3) {$\gamma'_2$};
\node at(-0.25,0) {$v_1$};
\node at(1.8,0.2) {$v_2$};
\end{tikzpicture}$$ and let $\Sigma=\Pi(Q)/\langle\alpha'_{1}\gamma'_{1}=\gamma'_{1}\alpha'_{1},\gamma'_2\alpha'_{1}=(\alpha'_{2})^{2}\gamma'_2\rangle$. Then $\Sigma$ is the direct limit of the direct system (\ref{direct-system-2}), and therefore $\Pi(E,A)\cong\Sigma$. Denote by $a_1=\alpha'_1$, $a_2=(\gamma'_{2})^{-1}\alpha'_2\gamma'_2$, $b_1=\gamma'_1$. Then $\Sigma(v_1,v_1)$ is the group generated by $a_1,a_2,b_1$ with relations $a_1=a_{2}^{2}$, $a_1 b_1=b_1 a_1$. Therefore
$$\Pi(E)\cong\Sigma(v_1,v_1)\cong F\langle a_1,a_2,b\rangle/\langle a_1=a_{2}^{2}, a_1 b_1=b_1 a_1\rangle\cong F\langle x,y\rangle/\langle x^2 y=yx^2\rangle.$$
\end{Ex1}

\begin{Thm}\label{fundamental group of BC}
Let $E=(E,P,L,d)$ be a connected Brauer configuration with $n$ vertices $v_1$, $\cdots$, $v_n$. Let $k_l$ be the number of $l$-gons of $E$ for each $l\geq 2$. Let $d_i=d_f(v_i)$ be the f-degree of $v_i$ for each $1\leq i\leq n$, and let $r=\sum_{l\geq 2} (l-1)k_l-n+1$. Then the fundamental group of $E$ is isomorphic to $F\langle a_1,\cdots,a_n,b_1,\cdots,b_r\rangle/\langle a_{1}^{d_1}=\cdots=a_{n}^{d_n}, a_{1}^{d_1} b_1=b_1 a_{1}^{d_1},\cdots, a_{1}^{d_1} b_r=b_r a_{1}^{d_1}\rangle$.
\end{Thm}

\begin{proof}
By Corollary \ref{fundamental group of BC and BG}, the fundamental group of $E$ is isomorphic to the fundamental group of $C$, where $C=(C,\widetilde{P},\widetilde{L},\widetilde{d})$ is a connected Brauer graph with $n$ vertices $\widetilde{v_1}$, $\cdots$, $\widetilde{v_n}$ and $\sum_{l\geq 2} (l-1)k_l$ edges, where the f-degree of $\widetilde{v_i}$ is $d_i$ for each $1\leq i\leq n$. It follows from Proposition \ref{fundamental group of BG} that the fundamental group of $C$ is isomorphic to
$$F\langle a_1,\cdots,a_n,b_1,\cdots,b_r\rangle/\langle a_{1}^{d_1}=\cdots=a_{n}^{d_n}, a_{1}^{d_1} b_1=b_1 a_{1}^{d_1},\cdots, a_{1}^{d_1} b_r=b_r a_{1}^{d_1}\rangle,$$
where $r=\sum_{l\geq 2} (l-1)k_l-n+1$.
\end{proof}

\begin{Rem1}
(1) The theorem above shows that the fundamental group of a connected BC is finitely presented (for the definition of finitely presented group, see \cite[page 106]{M}).

(2) According to Corollary \ref{iso-of-fundamental-gp-of-f-BC-and-fundamental-gp-of-quiver-with-aadmissible-relation}, the fundamental group of a connected BC $E$ is isomorphic to the fundamental group of the quiver with admissible relations $(Q'_E,I'_E)$, where the quiver $Q'_E$ and the ideal $I'_E$ are just the quiver and the ideal given in \cite[Section 2]{GS} for the corresponding BCA. So using Theorem \ref{fundamental group of BC} we can also calculate the fundamental group of $(Q'_E,I'_E)$ for a connected BC $E$.
\end{Rem1}

Finally, we would like to mention that the discussions in this paper suggest a general method to calculate the fundamental group of given f-BC $E$ using a covering $E\rightarrow E'$. This method will be used in the forthcoming paper \cite[Section 3]{LL2} and the idea is as follows.

Let $f:E\rightarrow E'$ be a covering of f-BCs. For every $e\in E$ and for every walk $w'$ of $E'$ with $s(w')=f(e)$, denote by $w$ the unique walk of $E$ with $s(w)=e$ and $f(w)=w'$. We denote by $w'\cdot e$ the terminal of $w$. By Proposition \ref{homotopy lifting}, we have $w'\cdot e=v'\cdot e$ if $w',v'$ are homotopic walks of $E'$ with $s(w')=s(v')=f(e)$, so we may define $\overline{w'}\cdot e$ as $w'\cdot e$. Under this operation $f^{-1}(e')$ becomes a $\Pi(E',e')$-set for every $e'\in E'$. Moreover, for every $e\in f^{-1}(e')$, the stabilizer subgroup of $e$ in $\Pi(E',e')$ is equal to the image of $f_{*}:\Pi(E,e)\rightarrow \Pi(E',e')$, which is isomorphic to $\Pi(E,e)$. This reduces the calculation of the fundamental group $\Pi(E,e)$ to the calculation of the stabilizer subgroup of $e$ in $\Pi(E',e')$. In particular, if $E$ is a finite connected $f_{ms}$-BC such that the group of automorphisms $\langle\sigma\rangle$ of $E$ generated by the Nakayama automorphism $\sigma$ acts admissibly on $E$, then $E/\langle\sigma\rangle$ is a connected BC and we have a covering $E\rightarrow E/\langle\sigma\rangle$, where the fundamental group of $E/\langle\sigma\rangle$ is given by Corollary \ref{fundamental group of BC}. Therefore to calculate the fundamental group $\Pi(E,e)$ of $E$, it suffices to calculate the stablizer subgroup of $e$ in $\Pi(E/\langle\sigma\rangle,e^{\langle\sigma\rangle})$.

\end{document}